\documentclass[reqno,a4paper, 11pt]{amsart}

\usepackage[a4paper=true,pdfpagelabels]{hyperref}
\usepackage{graphicx}

\usepackage[ansinew]{inputenc}
\usepackage{amsfonts,epsfig}
\usepackage{latexsym}
\usepackage{mathabx}
\usepackage{amsmath}
\usepackage{amssymb}

\usepackage{color}   

\newtheorem{theorem}{Theorem}
\newtheorem{lemma}[theorem]{Lemma}
\newtheorem{corollary}[theorem]{Corollary}
\newtheorem{proposition}[theorem]{Proposition}

\newtheorem{lettertheorem}{Theorem}

\theoremstyle{definition}

\theoremstyle{remark}

\numberwithin{equation}{section}

\newcommand{\B}{\mathcal{B}}
\newcommand{\D}{\mathbb{D}}
\newcommand{\DD}{\widehat{\mathcal{D}}}
\newcommand{\Dd}{\widecheck{\mathcal{D}}}
\newcommand{\M}{\mathcal{M}}
\newcommand{\DDD}{\mathcal{D}}

\newcommand{\N}{\mathbb{N}}

\newcommand{\RR}{\mathbb{R}}

\newcommand{\C}{\mathbb{C}}

\newcommand{\e}{\varepsilon}

\renewcommand{\phi}{\varphi}

\def\BMOA{\mathord{\rm BMOA}}

\def\a{\alpha}       \def\b{\beta}        \def\g{\gamma}
           \def\e{\varepsilon}
     \def\om{\omega}      
       \def\t{\theta}       
         \def\r{\rho}         \def\z{\zeta}
                  \def\vp{\varphi}

\def\omg{\widehat{\omega}}

\renewcommand{\H}{\mathcal{H}}

\newenvironment{Prf}{\noindent{\emph{Proof of}}}
{\hfill$\Box$ }

\begin{document}

\title[Bergman projection induced by radial weight]{Bergman projection induced by radial weight}


\author{Jos\'e \'Angel Pel\'aez}
\address{Departamento de An\'alisis Matem\'atico, Universidad de M\'alaga, Campus de
Teatinos, 29071 M\'alaga, Spain} \email{japelaez@uma.es}

\author{Jouni R\"atty\"a}
\address{University of Eastern Finland, P.O.Box 111, 80101 Joensuu, Finland}
\email{jouni.rattya@uef.fi}

\thanks{This research was supported in part by Ministerio de Econom\'{\i}a y Competitividad, Spain, projects
PGC2018-096166-B-100; La Junta de Andaluc{\'i}a,
project FQM210  and UMA18-FEDERJA-002; Academy of Finland project no. 268009; Vilho, Yrjö ja Kalle Foundation}


\maketitle

\tableofcontents \thispagestyle{empty}

\begin{abstract}

We establish characterizations of the radial weights $\omega$ on the unit disc such that the Bergman projection $P_\omega$, induced by $\omega$, is bounded and/or acts surjectively from $L^\infty$ to the Bloch space $\mathcal{B}$, or the dual of the weighted Bergman space $A^1_\omega$ is isomorphic to the Bloch space under the $A^2_\omega$-pairing. We also solve the problem posed by Dostani\'c in 2004 of describing the radial weights~$\omega$ such that~$P_\omega$ is bounded on the Lebesgue space~$L^p_\omega$, under a weak regularity hypothesis on the weight involved. With regard to Littlewood-Paley estimates, we characterize the radial weights~$\omega$ such that the norm of any function in $A^p_\omega$ is comparable to the norm in $L^p_\omega$ of its derivative times the distance from the boundary. This last-mentioned result solves another well-known problem on the area. All characterizations can be given in terms of doubling conditions on moments and/or tail integrals $\int_r^1\omega(t)\,dt$ of $\omega$, and are therefore easy to interpret.

\end{abstract}

\section{Introduction and main results}

Projections are essential building blocks of the concrete operator theory on spaces of analytic functions. One cornerstone is to characterize their boundedness which together with its multiple consequences makes it a pivotal topic in the theory~\cite{AlPoRe,B1981,BB,Nazarov}. Indeed, bounded analytic projections can be used to establish duality relations and to obtain useful equivalent norms in spaces of analytic functions~\cite{AlCo,DurSchus,HKZ,LueckingIndiana,Pabook2,PelRathg,PelRatproj,Zhu}.

The question of when the Bergman projection~$P_\om$ induced by a radial weight $\om$ on the unit disc is a bounded operator from one space into another is of primordial importance in the theory of Bergman spaces. The long-standing problem of describing the radial weights~$\om$ such that~$P_\om$ is bounded on the Lebesgue space~$L^p_\om$ had been known to experts since decades before it was formally posed by Dostani\'c in 2004~\cite{Dostanic}. A natural limit case of this setting is when $P_\om$ acts from $L^\infty$ to the Bloch space. The surjectivity of the operator becomes another relevant question in this limit case.

The main findings of this study are shortly listed as follows. We establish characterizations of the radial weights $\om$ on the unit disc such that $P_\om:L^\infty\to\B$ is bounded and/or acts surjectively, or the dual of the weighted Bergman space $A^1_\om$ is isomorphic to the Bloch space $\B$ under the $A^2_\om$-pairing. We also solve the problem posed by Dostani\'c under a weak regularity hypothesis on the weight involved. With regard to Littlewood-Paley estimates, we describe the radial weights~$\om$ such that the norm of any function in the weighted Bergman space $A^p_\om$ is comparable to the norm in $L^p_\om$ of its derivative times the distance from the boundary. This last-mentioned result solves another well-known problem on the area. All characterizations can be given in terms of doubling conditions on moments and/or tail integrals $\int_r^1\om(t)\,dt$ of $\om$, and are therefore easy to interpret.

Let $\H(\D)$ denote the space of analytic functions in the unit disc $\D=\{z\in\C:|z|<1\}$. For a nonnegative function $\om\in L^1([0,1))$, the extension to $\D$, defined by $\om(z)=\om(|z|)$ for all $z\in\D$, is called a radial weight. For $0<p<\infty$ and such an $\omega$, the Lebesgue space $L^p_\om$ consists of measurable functions such that
    $$
    \|f\|_{L^p_\omega}^p=\int_\D|f(z)|^p\omega(z)\,dA(z)<\infty,
    $$
where $dA(z)=\frac{dx\,dy}{\pi}$ is the normalized area measure on $\D$. The corresponding weighted Bergman space is $A^p_\om=L^p_\omega\cap\H(\D)$. Throughout this paper $\widehat{\om}(z)=\int_{|z|}^1\om(s)\,ds>0$ for all $z\in\D$, for otherwise $A^p_\om=\H(\D)$.

For a radial weight $\om$, the orthogonal Bergman projection $P_\om$ from $L^2_\om$ to $A^2_\om$ is
		\begin{equation*}
    P_\om(f)(z)=\int_{\D}f(\z) \overline{B^\om_{z}(\z)}\,\om(\z)dA(\z),
    \end{equation*}
where $B^\om_{z}$ are the reproducing kernels of $A^2_\om$. As usual,~$A^p_\alpha$ stands for the classical weighted
Bergman space induced by the standard radial weight $\omega(z)=(\alpha+1)(1-|z|^2)^\alpha$, $B^\alpha_z$ are the kernels of $A^2_\alpha$, and $P_\alpha$ denotes the corresponding Bergman projection.

It is well-known that the standard weighted Bergman projection $P_\alpha$ is bounded and onto from the space $L^\infty$ of bounded measurable functions to the Bloch space~$\B$, which consists of $f\in\H(\mathbb D)$ such that $\|f\|_\B=\sup_{z\in\D}|f'(z)|(1-|z|^2)+|f(0)|<\infty$. This raises the question of which properties of a radial weight $\om$ are determinative for $P_\om:L^\infty\to\B$ to be bounded? And what makes the embedding $\B\subset P_\om(L^\infty)$ continuous? These questions maybe paraphrased in terms of the dual space of $A^1_\om$ because the dual $(A^1_\om)^\star$ can be identified with $P_\om(L^\infty)$ under the $A^2_\om$-pairing
    $$
    \langle f,g\rangle_{A^2_\om}=\lim_{r\to 1^-}\int_{\D}f_r(z)\overline{g(z)}\om(z)\,dA(z).
    $$
It is known that $(A^1_\alpha)^\star\simeq\B$ and $\B_0^\star\simeq A^1_\alpha$ under the $A^2_\alpha$-pairing. Here $\B_0$ stands for the little Bloch space which consists of $f\in\H(\mathbb D)$ such that $f'(z)(1-|z|^2)\to0$, as $|z|\to1^-$.

The first main result of this paper characterizes the class of radial weights such that $P_\om:L^\infty\to\B$ is bounded, and establishes the corresponding natural duality relations, under our initial hypothesis on the positivity of $\widehat{\om}$. To state the result, some more notation is needed. A radial weight $\om$ belongs to the class~$\DD$ if the tail integral $\widehat{\om}$ satisfies the doubling property $\widehat{\om}(r)\le C\widehat{\om}(\frac{1+r}{2})$ for some constant $C=C(\om)\ge1$ and for all $0\le r<1$. We give a glimpse of how weights in the class $\DD$ behave. If $\om$ tends to zero too fast, for example exponentially, when approaching the boundary, then obviously $\om$ cannot belong to $\DD$. However, the growth of $\om\in\DD$ to infinity is not restricted because each increasing weight belongs to $\DD$, and hence, in particular, $\widehat{\om}$ may tend to zero much slower than any positive power of the distance to the boundary. The containment in $\DD$ also restricts the oscillation of $\om$, yet permits it to vanish identically in a relatively large part of each outer annulus of $\D$. Concrete examples illustrating some characteristic properties and/or strange behavior of weights in $\DD$ will be represented in the forthcoming sections, see also~\cite[Chapter~1]{PelRat}.

\begin{theorem}\label{ELTEOREMA}
Let $\om$ be a radial weight. Then the following statements are equivalent:
\begin{itemize}
\item[\rm(i)] $P_\om:L^\infty\to\B$ is bounded;
\item[\rm(ii)] There exists $C=C(\om)>0$ such that for each $L\in(A^1_\om)^\star$ there is a unique $g\in\B$ such that $L(f)=\langle f,g\rangle_{A^2_\om}$ for all $f\in A^1_\om$ with $\|g\|_{\B}\le C\|L\|$, that is, $(A^1_\om)^\star\subset\B$ via the
$A^2_\om$-pairing;
\item[\rm(iii)] There exists $C=C(\om)>0$ such that for each $L\in\mathcal{B}_0^\star$ there is a unique $g\in A^1_\om$ such that $L(f)=\langle f,g\rangle_{A^2_\om}$ for all $f\in\B_0$ with $\|g\|_{A^1_\om}\le C\|L\|$, that is, $\mathcal{B}_0^\star\subset A^1_\om$ via the
$A^2_\om$-pairing;
\item[\rm(iv)]$\om\in\DD$.
\end{itemize}
\end{theorem}

One of the main obstacles in the proof of Theorem~\ref{ELTEOREMA} and throughout this work is the lack of explicit expressions for
the Bergman reproducing kernel $B^\om_z$. For a radial weight $\om$, the kernel has the representation $B^\om_z(\z)=\sum \overline{e_n(z)}e_n(\z)$ for each orthonormal basis $\{e_n\}$ of $A^2_\om$, and therefore we are basically forced to work with the formula $B^\om_z(\z)=\sum_{n=0}^\infty\frac{\left(\overline{z}\z\right)^n}{2\om_{2n+1}}$ induced by the normalized monomials. Here $\om_{2n+1}$ are the odd moments of $\om$, and in general from now on we write $\om_x=\int_0^1r^x\om(r)\,dr$ for all $x\ge0$. Therefore the influence of the weight to the kernel is transmitted by its moments through this infinite sum, and nothing more than that can be said in general. This is in stark contrast with the neat expression $(1-\overline{z}\z)^{-(2+\alpha)}$ of the standard Bergman kernel $B^\alpha_z$ which is easy to work with as it is zero-free and its modulus is essentially constant in hyperbolically bounded regions. In general, the Bergman reproducing kernel induced by a radial weight may have a wild behavior in the sense that for a given radial weight $\om$ there exists another radial weight $\nu$ such that $A^2_\om=A^2_\nu$, but $B^\nu_z$ have zeros, see~\cite[Proof of Theorem~2]{BoudreauxJGA19} and also \cite{PeralaJGA17}. The proof of Theorem~\ref{ELTEOREMA} draws strongly on \cite[Theorem~1]{PelRatproj}, which says, in particular, that
	\begin{equation}\label{12345}
	\left\|(B^\om_z)^{(N)}\right\|_{A^p_\nu}^p\asymp\int_0^{|z|}\frac{\widehat{\nu}(t)}{\widehat{\om}(t)^p(1-t)^{p(N+1)}}dt+1
	\end{equation}
for $\om,\nu\in\DD$. It is of course the special case $N=1=p$ that plays a role in the proof of Theorem~\ref{ELTEOREMA}. For $\om\in\DD$ we have $\om_x\asymp\widehat\om\left(1-\frac1{x+1}\right)$, and therefore, as the moments are involved in the kernel, the appearance of the tail integrals on the right hand side of \eqref{12345} is natural. This also implies $\om\in\DD$ if and only if there exists a constant $C=C(\om)>1$ such that $\om_x\le C\om_{2x}$ for all $0\le x<\infty$. Thus the containment in $\DD$ can be equally well stated in terms of the tail integrals as the moments.

As for characterizing the continuous embedding of $\B$ into $P_\om(L^\infty)$, we write $\om\in\M$ if there exist constants $C=C(\om)>1$ and $K=K(\om)>1$ such that $\om_{x}\ge C\om_{Kx}$ for all $x\ge1$.

\begin{theorem}\label{th:otromundo}
Let $\om$ be radial weight. Then the following statements are equivalent:
\begin{itemize}
\item[\rm(i)] There exists $C=C(\om)>0$ such that for each $f\in\B$ there is $g\in L^\infty$ such that $P_\om(g)=f$ and $\|g\|_{L^\infty}\le C\|f\|_{\B}$, that is, $\B$ is continuously embedded into $P_\om(L^\infty)$;
\item[\rm(ii)] There exists $C=C(\om)>0$ such that $\left|\langle f,g\rangle_{A^2_\om}\right|\le C\|g\|_{\B}\|f\|_{A^1_\om}$ for all $f\in A^1_\om$ and $g\in\B$, that is, each $g\in\B$ induces an element of $(A^1_\om)^\star$ via the $A^2_\om$-pairing;
\item[\rm(iii)] There exists $C=C(\om)>0$ such that $\left|\langle g,f\rangle_{A^2_\om}\right|\le C\|g\|_{\B}\|f\|_{A^1_\om}$ for all $f\in A^1_\om$ and $g\in\B_0$, that is, each $f\in A^1_\om$ induces an element of $\B_0^\star$ via the $A^2_\om$-pairing;
\item[\rm(iv)]$\om\in\M$.
\end{itemize}
\end{theorem}

In view of Theorems~\ref{ELTEOREMA} and \ref{th:otromundo}, it is tempting to think that a kind of reverse doubling condition for tail integrals would characterize the continuous embedding $\B\subset P_\om(L^\infty)$, and thus the class~$\M$ as well. But unfortunately this is not the case. Namely, by writing $\om\in\Dd$ if there exist constants $K=K(\om)>1$ and $C=C(\om)>1$ such that $\widehat{\om}(r)\ge C\widehat{\om}\left(1-\frac{1-r}{K}\right)$ for all $0\le r<1$, we have $\Dd\subsetneq\M$ by Proposition~\ref{pr:counterxample} below. Therefore it is not enough to consider the tail integrals but we must indeed work with a true moment condition, and that is of course more difficult. The reason why $\M$ contains in a sense much worse weights than~$\Dd$ is that the global moment integral may easily eat up irregularities of the weight that make the tail condition of $\Dd$ to fail. That is how a family of counter examples is constructed to get the proposition.

To prove Theorem~\ref{th:otromundo} we first show that (iii) implies (iv) by testing with two appropriate families of analytic (lacunary) polynomials. Then, assuming $\om\in\M$, for each $f\in\B$ we construct an explicit preimage $g\in L^\infty$ such that $\|g\|_{L^\infty}\le C \|P_\om(g)\|_{\B}= C\|f\|_{\B}$, and thus obtain (i). In order to do this, we use properties of smooth universal Ces\'aro basis of polynomials~\cite[Section 5.2]{Pabook2}, a description of the coefficient multipliers of the Bloch space and several characterizations of the class $\M$. In particular,  we will show that $\om\in\M$ if and only if for some (equivalently for each) $\b>0$, there exists a constant $C=C(\om,\b)>0$ such that
    \begin{equation}\label{Intro:M}
    \om_x \le Cx^\b(\om_{[\b]})_x,\quad 1\le x<\infty,
    \end{equation}
a description of $\M$ which will be useful for our purposes also in other instances. The proof is then completed by passing from (i) to (ii), which is easy, and then observing that (ii) trivially implies (iii).

With Theorems~\ref{ELTEOREMA} and~\ref{th:otromundo} in hand, we characterize the radial weights $\om$ such that $P_\om: L^\infty\to \B$ is bounded and onto, or equivalently $(A^1_\om)^\star$ can be identified with the Bloch space via the $A^2_\om$-pairing. To this end we write $\DDD=\DD\cap\Dd$.

\begin{theorem}\label{th:projectionboundedonto}
Let $\om$ be a radial weight. Then the following statements are equivalent:
    \begin{enumerate}
    \item[(i)] $P_\om:L^\infty\to\B$ is bounded and onto;
    \item[(ii)] $(A^1_\om)^\star\simeq\B$ via the $A^2_\om$-pairing with equivalence of norms;
    \item[(iii)] $\B_0^\star\simeq A^1_\om$ via the $A^2_\om$-pairing with equivalence of norms;
    \item[(iv)] $\om\in\DDD$;
		\item[(v)] $\om\in\DD\cap\M$.
\end{enumerate}
\end{theorem}

The proof of Theorem~\ref{th:projectionboundedonto} boils down to showing that $\DDD=\DD\cap\M$. While $\DD\cap\M\subset\DDD$ is proved with the aid of characterizations of classes of weights obtained en route to Theorem~\ref{th:otromundo}, the proof of the opposite inclusion is more involved and will be actually established by showing the inequality in Theorem~\ref{th:otromundo}(ii) under the hypothesis $\om\in\Dd$.

It is worth underlining that $\DD\cap\M=\DD\cap\Dd$, despite the fact that $\Dd$ is a proper subclass of~$\M$. Roughly speaking the phenomenon behind this identity is that the severe oscillation that is possible for weights in $\M$ is ruled out by intersecting with $\DD$, and that makes the sets of weights equal.

We will show in Theorem~\ref{theorem:non-regular weights close to D} below that for each $\om\in\DDD$ there exists a weight $W$ which induces a weighted Bergman space only slightly larger than $A^p_\om$ but the weight itself lies outside of $\DD\cup\M$. Therefore $P_W:L^\infty\to\B$ is neither bounded nor onto by Theorems~\ref{ELTEOREMA} and~\ref{th:otromundo}. In particular, for given $-1<\a,\b<\infty$ one may pick up a weight $W$ such that $A^p_\alpha\subset A^p_W\subset A^p_\b$, but $P_W$ is neither bounded nor onto meanwhile of course $P_\alpha$ and $P_\b$ both have these properties. This shows in very concrete terms that $P_\om:L^\infty\to\B$ being bounded and/or onto depends equally well on the growth/decay of the inducing weight as on its regularity.

Theorem~\ref{th:projectionboundedonto} raises the problem of finding a useful description of the dual of $A^1_\om$ when $\om\in\DD\setminus\Dd$. To deal with this question we will need some more notation. For a radial weight $\om$ and $f,g\in\H(\D)$ with Maclaurin series $f(z)=\sum_{k=0}^\infty\widehat{f}(k)z^k$ and $g(z)=\sum_{k=0}^\infty\widehat{g}(k)z^k$, denote
    \begin{equation*}
    \langle f,g\rangle_{\om\circ\om}=\lim_{r\to 1^-}\sum_{k=0}^{\infty} \widehat{f}(k)\overline{\widehat{g}(k)}r^{2k+1}(\om_{2k+1})^2,
    \end{equation*}
whenever the limit exists. For $f\in\H(\D)$ and $0<r<1$, set
    \begin{equation*}
    \begin{split}
    M_p(r,f)&=\left(\frac{1}{2\pi}\int_{0}^{2\pi} |f(re^{it})|^p\,dt\right)^{\frac1p},\quad
    0<p<\infty,
    \end{split}
    \end{equation*}
and $M_\infty(r,f)=\max_{|z|=r}|f(z)|$. For $0<p\le\infty$, the Hardy space $H^p$ consists of $f\in \H(\mathbb D)$ such that $\|f\|_{H^p}=\sup_{0<r<1}M_p(r,f)<\infty$, and $\BMOA$ is the space of functions in $H^1$ that have bounded mean oscillation on the boundary. We define
    $$
    \BMOA(\infty,\om)
		=\left\{f\in\H(\D):\|f\|_{\BMOA(\infty,\om)}
		=\sup_{0<r<1}\|f_r\|_{\BMOA}\widehat{\om}(r)<\infty\right\},
    $$
where $f_r(z)=f(rz)$. Our next result describes the dual of $A^1_\om$ when $\om\in\DD$. We will discuss the case $1<p<\infty$ when $P_\om$ on $L^p$ is considered. We also note that it is not hard to identify the dual of $A^p_\om$ with the Bloch space under an appropriate paring when $0<p<1$ and $\om\in\DD$, see \cite[Theorem~1]{PeralaRattya2017} for details.

\begin{theorem}\label{th:A1dual}
Let $\om\in\DD$. Then $(A^1_\om)^\star$ can be identified with $\BMOA(\infty,\om)$ via the
$\langle\cdot,\cdot\rangle_{\om\circ\om}$-pairing with equivalence of norms.
\end{theorem}

The proof of Theorem~\ref{th:A1dual} is quite involved, and requires a large number of technical lemmas and                    known results from the existing literature. The first observation is that the dual of~$A^1_\om$ can be identified with the space of coefficient multipliers from $A^1_\om$ to $H^\infty$. Second, by using \cite[Theorem~2.1]{PavMixnormI}, it is shown that there exists $K>1$ and a sequence of polynomials $\{P_n\}_{n=0}^\infty$ such that $\|f\|_{\BMOA(\infty,\om)}\asymp\sup_{n\in\N\cup\{0\}}K^{-n}\|P_n\ast f\|_{\BMOA}$ and $\|f\|_{A^1_\om}\asymp\sum_{n=0}^\infty K^{-n}\|P_n\ast f\|_{H^1}$, where $\ast$ denotes the convolution. These are the new norms that we will work with. Then, it is shown that the operator $I^\om(g)(z)=\sum_{k=0}^\infty\widehat{g}(k)\om_{2k+1}z^{k}$ satisfies $\|I^\om(P_n\ast f)\|_{\BMOA}\asymp K^{-n}\|P_n\ast f\|_{\BMOA}$. In this step we strongly use the hypothesis $\om\in\DD$ and the fact that the Ces\'aro means are uniformly bounded on $\BMOA$~\cite[Theorem~2.2]{PavMixnormI}. Now that Fefferman's duality relation $(H^1)^\star\simeq\BMOA$ allows us to identify the coefficient multipliers from~$H^1$ to~$H^\infty$ with~$\BMOA$, the operator $I^\om\circ I^\om$ finally gives the isomorphism we are after. One of the crucial steps in the argument is the norm estimate of $I^\om(P_n\ast f)$ on $\BMOA$ stated above, and it deserves special attention for a reason. Namely, a part of the proof of this relies on the technical fact that $\om\in\DD$ if and only if for some (equivalently for each) $\b>0$, there exists a constant $C=C(\om,\b)>0$ such that
    \begin{equation}\label{Intro:D-hat}
    x^\b(\om_{[\b]})_x\le C\om_x,\quad 0\le x<\infty.
    \end{equation}
This analogue of the characterization \eqref{Intro:M} of the class $\M$ for $\DD$ is obtained as a consequence of the Littlewood-Paley estimates to be discussed next.

It is well known that bounded Bergman projections on weighted $L^p$-spaces can be used to establish equivalent norms on weighted Bergman spaces in terms of derivatives~\cite{AlCo,Zhu}. Before we proceed to consider the weighted inequalities for the Bergman projection $P_\om$, which is one of the main topics of the paper, we discuss such norm estimates, usually called the Littlewood-Paley formulas. The arguments we employ do not appeal to bounded Bergman projections, but when we will return to $P_\om$, we will also explain how projections can be used to obtain certain norm estimates in our setting.

Probably the most known Littlewood-Paley formula states that for each standard radial weight $\om(z)=(\alpha+1)(1-|z|^2)^\alpha$ we have
	$$
	\|f\|_{A^p_\om}\asymp\sum_{k=0}^{n-1}|f^{(k)}(0)|+\left\|f^{(n)}\right\|_{A^p_{\om_{[np]}}},\quad f\in\H(\D),
	$$
where $\om_{[\b]}(z)=\om(z)(1-|z|)^{\b}$. Different extensions of this equivalence as well as related partial results can be found in
\cite{AlCo,AlPoRe19,BaoWulanZhu18,PavP}. The question for which radial weights the above equivalence is valid has been a known open problem for decades, and gets now completely answered by our next result.

\begin{theorem}\label{th:L-P-D}
Let $\om$ be a radial weight, $0<p<\infty$ and $k\in\N$. Then
    \begin{equation}\label{Eq:L-P-D}
    \|f\|_{A^p_\om}^p\asymp\int_\D|f^{(k)}(z)|^p(1-|z|)^{kp}\om(z)\,dA(z)+\sum_{j=0}^{k-1}|f^{(j)}(0)|^p,\quad f\in\H(\D),
    \end{equation}
if and only if $\om\in\DDD$.
\end{theorem}

We can actually do a bit better than what is stated in Theorem~\ref{th:L-P-D}. Namely, on the way to the proof we will establish the following result.

\begin{theorem}\label{theorem:L-P-D-hat}
Let $\om$ be a radial weight, $0<p<\infty$ and $k\in\N$. Then there exists a constant $C=C(\om,p,k)>0$ such that
    \begin{equation}\label{eq:desLP}
    \int_\D|f^{(k)}(z)|^p(1-|z|)^{kp}\om(z)\,dA(z)+\sum_{j=0}^{k-1}|f^{(j)}(0)|^p\le C\|f\|_{A^p_\om}^p,\quad f\in\H(\D),
    \end{equation}
     if and only if $\om\in\DD$.
\end{theorem}

The proof of \eqref{eq:desLP} for $\om\in\DD$ is based on standard techniques, but the other implication is much more involved. In particular, the proof reveals that \eqref{eq:desLP} holds if and only if the inequality there holds for all monomials only, and that \eqref{Intro:D-hat} is a characterization of $\DD$. Thus, in our approach, this latter consequence of the proof of Theorem~\ref{theorem:L-P-D-hat} is used to achieve Theorem~\ref{th:A1dual}.

As for the proof of Theorem~\ref{th:L-P-D}, we will show that
	$$
	\|f\|_{A^p_\om}^p\lesssim \int_\D|f^{(k)}(z)|^p(1-|z|)^{kp}\om(z)\,dA(z)+\sum_{j=0}^{k-1}|f^{(j)}(0)|^p,\quad f\in\H(\D),
	$$
is satisfied whenever $\om\in\Dd$. Since $\om\in\Dd$ implies that $(1-s)^\g\om(s)\in\Dd$, $\g\ge 0$, this is obtained by reducing the question to the case $k=1$,  
using an appropriate exponential partition of $[0,1)$ depending on the weight $\om$ and the inequality \cite[Lemma~2]{Pacific} which states that
	\begin{equation}\label{dinkeli}
	M_p^s(\rho,f)-M_p^s(r,f)\le C_p(\rho-r)^sM^s_p(\rho,f'),\quad 0<r<\rho<1,\quad s=\min\{p,1\}.
	\end{equation}
The last step in the proof of Theorem~\ref{th:L-P-D} consists of testing with monomials in \eqref{Eq:L-P-D}, and then using \eqref{Intro:M}, \eqref{Intro:D-hat} and Theorem~\ref{th:projectionboundedonto} to deduce $\om\in\DDD$. It is also worth mentioning that unlike $\Dd$, the class~$\DD$ is not closed under multiplication by $(1-s)^\g$ for any $\g>0$~\cite[Theorem~3]{PRHarmonic}.

A description of the radial weights such that	
	$$
	\|f\|_{A^p_\om}^p\lesssim\int_\D|f^{(k)}(z)|^p(1-|z|)^{kp}\om(z)\,dA(z)+\sum_{j=0}^{k-1}|f^{(j)}(0)|^p,\quad f\in\H(\D),
	$$
remains an open problem for $p\ne2n$ with $n\in\N$. This matter will be briefly discussed at the end of the paper. For Littlewood-Paley estimates in the case of non-radial weight under additional regularity hypotheses, see \cite{AlCo,AlPoRe19,BaoWulanZhu18}.

The question of describing the radial weights such that
	\begin{equation}\label{IntroPomega}
	\|P_\om(f)\|_{L^p_\om}\le C \| f\|_{L^p_\om},\quad f\in L^p_\om, \quad 1<p<\infty,
	\end{equation}
was formally posed by Dostani\'c~\cite[p.~116]{Dostanic}. In the recent years there has been an increasing activity concerning this norm inequality~\cite{CP2,Dostanic,D09,PelRatproj,ZeyTams2012}. Our next result provides a sufficient condition, much weaker than known conditions so far and of correct order of magnitude, for \eqref{IntroPomega} to hold.

\begin{theorem}\label{th:dualityAp}
Let $\om\in\DD$ and $1<p<\infty$. Then $P_\om:L^p_\om\to L^p_\om$ is bounded.
\end{theorem}

The proof of Theorem~\ref{th:dualityAp} draws on the fact that for any radial weight $\om$ and $1<p<\infty$, the Bergman projection $P_\om$ is bounded on $L^p_\om$ if and only if $(A^{p}_\om)^\star\simeq A^{p'}_\om$ with equivalence norms under the $A^2_\om$-pairing. We will prove this duality relation by following a similar scheme to that in the proof of Theorem~\ref{th:A1dual}. However, in this case the preparations are easier because of our previous result concerning equivalent norms on $A^p_\om$ in terms of $\ell^p$-type norm of the Hardy norms of blocks of the Maclaurin series, whose size depend on the weight $\om$, see~\cite[Theorem~3.4]{PRAntequera} and also \cite[Theorem~4]{PelRathg}.

In the opposite direction to Theorem~\ref{th:dualityAp}, Dostani\'c~\cite{Dostanic} showed that for a radial weight $\om$ and $1<p<\infty$, the reverse H\"older's inequality
    \begin{equation}\label{Intro:Dostanic-condition}
    \om_{np+1}^\frac1p\om_{np'+1}^\frac1{p'}\lesssim \om_{2n+1},\quad n\in\N,
    \end{equation}
is a necessary condition for $P_\om:L^p_\om\to L^p_\om$ to be bounded. We will offer two simple proofs of this fact as well as two equivalent statements in Proposition~\ref{proposition:dostanic} below. The natural question now is that does the Dostani\'c' condition \eqref{Intro:Dostanic-condition} imply $\om\in\DD$? Unfortunately, we do not know an answer to this, but we can certainly say that these conditions are the same if sufficient regularity on $\om$ is required. To state the result, write $L_\om(x)=-\log\om_x$ for all $0\le x<\infty$. The proof of the following result and more can be found in Section~\ref{sec:conjetures}.

\begin{corollary}\label{Cor-BergmanProjection}
Let $\om$ be a radial weight and $1<p<2$ such that either
	\begin{equation}\label{Eq:extraconditionEpsilon}
	\om_{2x}\lesssim\left(\om_{(p+\e)x}\right)^\frac1p\left(\om_{p'x}\right)^\frac1{p'},\quad 0\le x<\infty,
	\end{equation}
for some $\e=\e(\om)>0$, or for each $0\le y<\infty$ there exists $x=x(y)\in[0,\infty)$ such that
    \begin{equation}\label{oiuy}
    -L_\om''(y)y^2\le Cx\int_{\frac{2+p}{2}x}^{2x}\left(-L_\om''(t)\right)\,dt
    \end{equation}
for some constant $C=C(\om,p)>0$. Then $P_\om:L^p_\om\to L^p_\om$ is bounded if and only if $\om\in\DD$.
\end{corollary}
		
The innocent looking condition \eqref{Eq:extraconditionEpsilon} is not easy to work with in the case of irregularly behaving weights because it requires precise control over the moments. We have failed to construct a weight which does not satisfy this condition. With regard to \eqref{oiuy}, we mention here that the auxiliary function $L_\om$ is an increasing unbounded concave function, and it is easy to see that if $L_\om''$ is essentially increasing, then \eqref{oiuy} is trivially satisfied.

Concerning the maximal Bergman projection
	\begin{equation*}
	P^+_\om(f)(z)=\int_{\D}f(\z)|B^\om_{z}(\z)|\,\om(\z)dA(\z)
  \end{equation*}
we will show the following result.

\begin{theorem}\label{th:P+1peso}
Let $\om\in\DD$ and $1<p<\infty$. Then $P^+_\om:L^p_\om\to L^p_\om$ is bounded if and only if $\om\in\DDD$.
\end{theorem}

Theorem~\ref{th:P+1peso} is of its own interest, but we underline that it combined with Theorem~\ref{th:dualityAp} shows that the cancellation of the kernel plays an essential role in the boundedness of $P_\om:L^p_\om\to L^p_\om$ when $\om\in\DD\setminus\Dd$. In particular, we have the following immediate consequence of the previous two theorems.

\begin{corollary}\label{co:P+P}
Let $1<p<\infty$ and $\om\in\DD\setminus\Dd$. Then $P_\om:L^p_\om\to L^p_\om$ is bounded but
$P^+_\om:L^p_\om\to L^p_\om$ is not.
\end{corollary}

We also prove that the boundedness of $P_\om$ is equally much related to the regularity of the weight $\om$ as to its growth, see the note after Proposition~\ref{proposition:dostanic fails} below.

The inequality \eqref{IntroPomega} can of course also be interpreted as $\|P_\om(f)\|_{A^p_\om}\le C\|f\|_{L^p_\om}$. In view of the Littlewood-Paley formulas, it is natural to ask for which radial weights $\om$, the Bergman projection $P_\om$ is bounded from $L^p_\om$ to the Dirichlet-type space~$D^p_{\om,k}$. Here and from now on, for $0<p<\infty$ and $k\in\N$, $D^p_{\om,k}$ denotes the space of $f\in\H(\D)$ such that
    $$
    \|f\|_{D^p_{\om,k}}=\left\|f^{(k)}\right\|_{A^p_{\om_{[kp]}}}+\sum_{j=0}^{k}|f^{(j)}(0)|<\infty.
    $$
The following result gives an answer this question. Its proof uses Theorem~\ref{theorem:L-P-D-hat} and the norm estimate \eqref{12345}.

\begin{theorem}\label{th:PomegaDp}
Let $\om$ be a radial weight, $1<p<\infty$ and $k\in\N$. Then the following statements are equivalent:
\begin{enumerate}
\item[(i)] $P_\om:L^p_\om\to D^p_{\om,k}$ is bounded;
\item[(ii)] There exists $C=C(\om)>0$ such that for each $L\in (A^{p'}_\om)^\star$ there is a unique $g\in D^p_{\om,k}$ such that
$L(f)=\langle f,g\rangle_{A^2_\om}$ and $\| g\|_{D^p_{\om,k}}\le C\|L\|$. That is, $ (A^{p'}_\om)^\star$ is continuously embedded into
$ D^p_{\om,k}$;
\item[(iii)] $I_d:A^p_\om\to D^p_{\om,k}$ is bounded;
\item[(iv)] $\om\in\DD$.
\end{enumerate}
\end{theorem}

In the proof we show that $P_\om:L^p_\om\to D^p_{\om,k}$ is bounded whenever $\om\in\DD$. Since this trivially implies (iii), we deduce that \eqref{eq:desLP} is satisfied when $\om\in\DD$. This gives an example of how bounded projections can be used to get Littlewood-Paley estimates in our context.

Concerning Theorem~\ref{th:PomegaDp}, we can actually take a step further and characterize the radial weights $\om$ such that the dual of $A^{p'}_\om$ can be identified with $D^p_{\om,k}$ under the $A^2_\om$-pairing.

\begin{theorem}\label{th:PwDpwonto}
Let $\om$ be a radial weight, $1<p<\infty$ and $k\in\N$. Then the following statements are equivalent:
\begin{enumerate}
\item[\rm(i)]  $P_\om: L^p_\om \to D^p_{\om,k}$ is bounded and onto;
\item[\rm(ii)] $(A^{p'}_\om)^\star \simeq D^p_{\om,k}$ via the $A^2_\om$-pairing with equivalence of norms;
\item[\rm(iii)] $A^p_\om=D^p_{\om,k}$ with equivalence of norms;
\item[\rm(iv)] $\om\in\DDD$.
\end{enumerate}
\end{theorem}

The proof of Theorem~\ref{th:PwDpwonto} is strongly based on Theorems~~\ref{th:L-P-D},~\ref{theorem:L-P-D-hat} and~\ref{th:PomegaDp}. In addition, to show that (ii) implies (iv), two appropriate families of lacunary series are constructed to deduce $\om\in\Dd$.

Before we move on to norm inequalities involving different weights, we point out two things. First, all the results presented so far show that each of the classes of weights $\DD$, $\M$ and $\DDD$ appear in a very natural manner in the operator theory of Bergman spaces induced by radial weights. Second, the conditions describing the radial weights such that either \eqref{Eq:L-P-D} or \eqref{eq:desLP} is satisfied, or $P_\om:L^p_\om \to D^p_{\om,k}$ is bounded and/or onto do not depend on $p$.

Our next goal is to study the question of when for given radial weights $\om$ and $\nu$ we have
	\begin{equation}\label{IntroPomeganu}
	\|P_\om(f)\|_{L^p_\nu}\le C\|f\|_{L^p_\nu},\quad f\in L^p_\nu, \quad 1<p<\infty.
	\end{equation}
The most commonly known result concerning the one weight inequality \eqref{IntroPomeganu} for the Bergman projection $P_\om$ is undoubtedly due to Bekoll\'e and Bonami~\cite{B1981,BB}, and concerns the case when~$\nu$ is an arbitrary weight and the inducing weight $\om$ is standard. See \cite{ACJFA12,PRW,PottRegueraJFA13} for recent extensions of this result. Moreover, the main result in \cite{PelRatproj} characterizes \eqref{IntroPomeganu} under the hypothesis that both $\om$ and $\nu$ are regular weights i.e. they satisfy $\eta(r)\asymp\frac{\widehat{\eta}(r)}{1-r}$ for all $0\le r<1$. Regular weights form a subclass of $\DDD$ whose elements do not have zeros and are pointwise comparable to radial $C^n$-weights for $n\in\N$. Therefore these weights are really smooth and not that hard to work with.

Our next result substantially improves \cite[Theorem~3]{PelRatproj}. In particular, it characterizes the pairs $(\om,\nu)$ such that \eqref{IntroPomeganu} holds under the assumption $\om\in\DD$ and $\nu\in\M$. To state the result, some more notation is needed. For $1<p<\infty$ and radial weights $\om$ and $\nu$, let $\sigma(r)=\sigma_{p,\om,\nu}(r)=r\left(\frac{\om(r)}{\nu(r)^{\frac1p}}\right)^{p'}$ for all $0\le r<1$, and define
    \begin{equation}\label{eq:apcondition}
    A_p(\om,\nu)=\sup_{0\le r<1}\frac{\widehat{\nu}(r)^{\frac{1}{p}}\widehat{\sigma}(r)^{\frac{1}{p'}}}{\widehat{\om}(r)}
    \end{equation}
and
    \begin{equation}\label{eq:Mpcondition}
    M_p(\om,\nu)=\sup_{0<r<1}
    \left(\int_0^r\frac{\nu(s)}{\widehat{\om}(s)^p}\,sds+1\right)^\frac1p\widehat{\sigma}(r)^{\frac{1}{p'}}.
    \end{equation}

\begin{theorem}\label{Theorem:P_w-L^p_v}
Let $1<p<\infty$, $\om\in\DD$ and $\nu$ a radial weight. Then the following statements are equivalent:
\begin{enumerate}
    \item[\rm(i)] $P^+_\om:L^p_\nu\to L^p_\nu$ is bounded;
    \item[\rm(ii)] $P_\om:L^p_\nu\to L^p_\nu$ is bounded and $\nu\in\M$;
    \item[\rm(iii)] $A_p(\om,\nu)<\infty$ and $\om,\nu\in\DDD$;
    \item[\rm(iv)] $M_p(\om,\nu)<\infty$ and $\om,\nu\in\DDD$.
\end{enumerate}
Moreover,
\begin{equation}\label{eq:quantativePdospesos}
M^{1-\frac{1}{p}}_p(\om,\nu)\lesssim A_p(\om,\nu)\lesssim \|P_\om\|_{L^p_\nu\to L^p_\nu}\le \|P^+_\om\|_{L^p_\nu\to L^p_\nu}
\lesssim M^{2-\frac{1}{p}}_p(\om,\nu).
\end{equation}
\end{theorem}

The proof of Theorem~\ref{Theorem:P_w-L^p_v} is divided into several steps, and many auxiliary results interesting in themselves are used. We first establish by a simple testing a necessary condition for $P_\om:L^p_\nu\to L^p_\eta$ to be bounded, provided all the three weights involved are radial. This is a natural generalization of the one weight case given by the Dostani\'c condition \eqref{Intro:Dostanic-condition}, and is stated as Proposition~\ref{proposition:P-necessary} below. In the case $\eta=\nu$ it shows that $\om$ and $\nu$ cannot be so different from each other neither in growth/decay nor in regularity if $P_\om:L^p_\nu\to L^p_\nu$ is bounded. The second auxiliary result is Proposition~\ref{pr:P+otromundo} below and it states that $\om,\nu\in\M$ if $P^+_\om:L^p_\nu\to L^p_\nu$ is bounded. This says that $\om$ and $\nu$ cannot induce very small Bergman spaces when $P^+_\om:L^p_\nu\to L^p_\nu$ is bounded. This is one of the significant differences between $P^+_\om$ and $P_\om$. Proposition~\ref{pr:P+otromundo} is achieved with the aid of Hardy's inequality and appropriate testing. The most involved and technical part of the proof of Theorem~\ref{Theorem:P_w-L^p_v} consists of showing that (iv) implies (i) via an instance of Shur's test. This step relies on a suitable choice of auxiliary function and the norm estimate \eqref{12345}. Actually more is true than what is stated in the theorem because we will show in Proposition~\ref{theorem:P+wLpnu} that the condition $M_p(\om,\nu)<\infty$ describes the pairs $(\om,\nu)$ such that $P^+_\om:L^p_\nu\to L^p_\nu$ is bounded provided one of the weights involved belongs to $\DD$. Finally, we deduce (iii) from (ii) via standard methods, and we observe that (iii) yields (iv) for any radial weight.

The two weight inequality $\|P_\om(f)\|_{L^p_\eta}\lesssim\|f\|_{L^p_\nu}$ with arbitrary weights $\nu,\eta$ for the Bergman projection remains an open problem even in the case $\om\equiv1$. However, it was recently characterized in \cite{KorhonenPelaezRattya2018} under the hypotheses that all the three weights involved are radial and satisfy certain regularity hypotheses. We note that the method of proof used there does not help us to prove the most difficult part of Theorem~\ref{Theorem:P_w-L^p_v} which consists of showing that (iv) implies (i).

The rest of the paper is organized as follows. Theorems~\ref{ELTEOREMA},~\ref{th:otromundo} and~\ref{th:projectionboundedonto} are proved in
Section~\ref{sec:bloch}, but the proof of Theorem~\ref{th:A1dual} is postponed to Section~\ref{Sec:duality-A^p}. In Section~\ref{Sec:L-P-estimates} we prove Theorems~\ref{th:L-P-D} and~\ref{theorem:L-P-D-hat}, while Section~\ref{sec:projDp} contains the proofs of Theorems~\ref{th:PomegaDp} and~\ref{th:PwDpwonto}. Section~\ref{Sec:P_w} is devoted to the study of the inequality \eqref{IntroPomega}. It contains the proofs of Theorems~\ref{th:dualityAp} and~\ref{th:P+1peso} and Corollary~\ref{co:P+P}. Theorem~\ref{Theorem:P_w-L^p_v} is proved in Section~\ref{sec:Pwnu}. Finally, in Section~\ref{sec:conjetures} we discuss two open problems cited in the introduction and pose conjectures for them.

For clarity, a word about the notation already used in this section and to be used throughout the paper. The letter $C=C(\cdot)$ will denote an absolute constant whose value depends on the parameters indicated in the parenthesis, and may change from one occurrence to another.
We will use the notation $a\lesssim b$ if there exists a constant
$C=C(\cdot)>0$ such that $a\le Cb$, and $a\gtrsim b$ is understood
in an analogous manner. In particular, if $a\lesssim b$ and
$a\gtrsim b$, then we write $a\asymp b$ and say that $a$ and $b$ are comparable.

\section{$P_\om:L^\infty\to\B$ plus dualities $(A^1_\om)^\star\simeq\B$ and $\B_0^\star\simeq A^1_\om$}\label{sec:bloch}

In this section we will prove Theorems~\ref{ELTEOREMA},~\ref{th:otromundo} and~\ref{th:projectionboundedonto} in the order of appearance.

\medskip\par
\begin{Prf}{\em{Theorem~\ref{ELTEOREMA}.}}
First observe that if $\om$ is a weight such that $A^2_\om$ is continuously embedded
into $\H(\D)$, then
    \begin{equation}\label{Eq:L^infty-Bloch-condition-proposition}
    \|P_\om\|_{L^\infty\to\B}\asymp\|B^\om_0\|_{A^1_\om}+\sup_{z\in\D}(1-|z|^2)\int_\D\left|(B_\z^\om)'(z)\right|\om(\z)\,dA(\z).
    \end{equation}
To prove that (i) and (iv) are equivalent, assume first $\om\in\DD$. By using $z(B_\z^\om)'(z)=\overline{\z(B_z^\om)'(\z)}$ and \cite[Theorem~1]{PelRatproj} we deduce
    \begin{equation*}
    \begin{split}
    \int_\D|(B_\z^\om)'(z)|\om(\z)\,dA(\z)
    &\le\frac{1}{|z|}\|(B_z^\om)'\|_{A^1_\om}\asymp\frac{1}{1-|z|^2},\quad |z|\to 1^-,
    \end{split}
    \end{equation*}
and hence $P_\om:L^\infty\to\B$ is bounded by \eqref{Eq:L^infty-Bloch-condition-proposition}.
Conversely, if $P_\om:L^\infty\to\B$ is bounded, then \eqref{Eq:L^infty-Bloch-condition-proposition} and the Hardy's inequality \cite[Chapter~3]{Duren} yield
    \begin{equation*}
    \begin{split}
    \infty
    &>\sup_{z\in\D}(1-|z|^2)\int_\D\left|(B_\z^\om)'(z)\right|\om(\z)\,dA(\z)
    \gtrsim\sup_{z\in\D}(1-|z|^2)\sum_{n=1}^\infty\frac{|z|^{n-1}n}{\om_{2n+1}}\frac1{n}\int_0^1r^{n+1}\om(r)\,dr\\
    &\ge\sup_{z\in\D}(1-|z|)\left|\sum_{n=0}^\infty\frac{\om_{n+2}}{\om_{2n+3}}z^n\right|
    \ge\sup_{N\in\N}\frac1N\sum_{n=0}^N\frac{\om_{n+2}}{\om_{2n+3}}\left(1-\frac1N\right)^n
    \gtrsim\sup_{N\in\N}\frac1N\sum_{n=0}^N\frac{\om_{n+2}}{\om_{2n+3}}.
    \end{split}
    \end{equation*}
For $x\in \mathbb{R}$, let $E[x]\in \mathbb{Z}$ be defined by $E[x]\le x<E[x]+1$. Since
    \begin{equation*}
    \begin{split}
    1\gtrsim &\frac1N\sum_{n=0}^N\frac{\om_{n+2}}{\om_{2n+3}}\ge
    \frac{\om_{N+2}}{N}\sum_{n=E[\frac{3N}{4}]}^N\frac{1}{\om_{2n+3}}
    \ge \frac{\om_{N+2}}{4\om_{2E[\frac{3N}{4}]+3}},\quad N\in\N,
    \end{split}
    \end{equation*}
there exists a constant $C=C(\om)>0$ such that
    \begin{equation*}
    \begin{split}
    \om_{N+2}
    &\le C\om_{2E[\frac{3N}{4}]+3}
    \le C^2\om_{2E[\frac34(2E[\frac{3N}{4}]+1)]+3}\\
    &\le C^2\om_{2E[\frac{9N}{8}-\frac34]+3}
    \le C^2\om_{\frac{9N}{4}-\frac{1}{2}}\le C^2\om_{2N+4},\quad N\ge 18.
    \end{split}
    \end{equation*}
Hence $\om\in\DD$ by \cite[Lemma~2.1(ix)]{PelSum14}.

Next observe that for each radial weight $\om$ the dual space $(A^1_\om)^\star$ can be identified with $P_\om(L^\infty)$
under the $A^2_\om$-pairing,                                       and in particular, for each $L\in(A^1_\om)^\star$ there exists $h\in L^\infty$ such that
    \begin{equation}\label{eq:k1}
    L(f)=L_{P_\om(h)}(f)=\langle f,P_\om(h)\rangle_{A^2_\om},\quad f\in A^1_\om,\quad\|L\|=\|h\|_{L^\infty}.
    \end{equation}
To see this, note that		
		\begin{equation}\label{eq:fubinisubs}
		\langle f_r, P_\om(h)\rangle_{L^2_\om}=\langle f_r, h\rangle_{L^2_\om},\quad h\in L^1_\om,\quad f\in\H(\D),\quad 0<r<1.
		\end{equation}
If now $h\in L^\infty$ and $g=P_\om(h)$, then
		\begin{equation}
    \begin{split}\label{eq:dual1}
    |L_g(f)|
		&=\lim_{r\to 1^-}\left|\langle f_r, P_\om(h) \rangle_{L^2_\om}\right|
		=\lim_{r\to 1^-} \left| \langle f_r, h\rangle_{L^2_\om}\right|
		\le\|h\|_{L^\infty}\| f\|_{A^1_\om},\quad f\in A^1_\om,
    \end{split}
    \end{equation}
and hence $L_g\in(A^1_\om)^\star$ with $\|L_g\|\le\|h\|_{L^\infty}$. Conversely, if $L\in (A^1_\om)^\star$, then by the Hahn-Banach theorem $L$ can be extended to a bounded linear functional on $L^1_\om$ with the same norm. Since $(L^1_\om)^\star$
can be identified with $L^\infty$ under the $L^2_\om$-pairing, there
exists $h\in L^\infty$ such that $L(f)=\langle f,h\rangle_{L^2_\om}$ for
all $f\in L^1_\om$. In particular, for each $f\in A^1_\om$ we have $L(f_r)=\langle f_r,h\rangle_{L^2_\om}
    =\langle f_r,P_\om(h)\rangle_{A^2_\om}
    $
by \eqref{eq:fubinisubs}. Therefore
    \begin{equation}\label{eq:dual2}
    L(f)=\lim_{r\to 1^-}L(f_r)=\lim_{r\to 1^-}\langle
    f_r,P_\om(h)\rangle_{A^2_\om}=L_{P_\om(h)}(f),\quad f\in A^1_\om.
    \end{equation}

It is an immediate consequence of \eqref{eq:k1} that (i) implies (ii). For the converse, let $h\in L^\infty$, and consider the functional defined by $L_{P_\om(h)}(f)=\langle f,P_\om(h)\rangle_{A^2_\om}$ for all $f\in A^1_\om$. Then $L_{P_\om(h)}\in(A^1_\om)^\star$ with $\|L\|\le\|h\|_{L^\infty}$ by what we just proved. But by the hypothesis (ii), there exists $g\in\B$ such that $L_{P_\om(h)}(f)=L_g(f)$  for all $f\in A^1_\om$, and $\|g\|_{\B}\le C\|L\|$. Hence $L_{P_\om(h)-g}$ represents a zero functional, and thus $g=P_\om(h)$ with $\|P_\om(h)\|_{\B}=\|g\|_{\B}\le C\|L\|=C\|h\|_{L^\infty}$. Therefore (i) is satisfied.

We now show that (iii) and (iv) are equivalent. Assume first $\om\in\DD$ and let $L\in \mathcal{B}_0^\star$. We aim for constructing a function $g\in A^1_\om$ such that $L(f)=L_g(f)=\langle f,g\rangle_{A^2_\om}$ for all $f\in\B_0$, and
$\|g\|_{A^1_\om}\lesssim\|L\|$. To do this, let $g(z)=\sum_{n=0}^\infty\widehat{g}(n)z^n$, where $\widehat{g}(n)=\overline{ \frac{L(z^n)}{\om_{2n+1}}}$ for all $n\in\N\cup\{0\}$. Then
    $$
    |\widehat{g}(n)|\le\frac{\|L\|\|z^n\|_{\mathcal{B}}}{\om_{2n+1}}\lesssim\frac{1}{\om_{2n+1}},\quad n\in\N\cup\{0\},
    $$
and therefore $g\in\H(\D)$. If
$f\in\mathcal{B}$ with $f(z)=\sum_{n=0}^\infty\widehat{f}(n)z^n$, then
$f_r\in\mathcal{B}_0$ for each $r\in(0,1)$, and
    $$
    \sum_{n=0}^m\widehat{f}(n)(rz)^n\to f_r(z),\quad m\to\infty,
    $$
in $\left(\mathcal{B}_0,\|\cdot\|_\mathcal{B}\right)$. Therefore
    \begin{equation}\label{eq:d0n}
    L\left(\sum_{n=0}^m\widehat{f}(n)(rz)^n\right)
    =\sum_{n=0}^m\widehat{f}(n)\overline{\widehat{g}(n)}\om_{2n+1}r^n\to L(f_r),\quad m\to\infty,
    \end{equation}
which implies
    \begin{equation}\label{eq:d00n}
    L(f_r)= \sum_{n=0}^\infty \widehat{f}(n)\overline{\widehat{g}(n)}\om_{2n+1}r^n
    =\langle f,g_r\rangle_{A^2_\om},
    \end{equation}
and hence
    \begin{equation}\label{eq:d1n}
    \left|\langle f,g_r\rangle_{A^2_\om}\right|
    =|L(f_r)|
    \le\|L\|\|f_r\|_{\mathcal{B}}
    \le\|L\|\|f\|_{\mathcal{B}},\quad f\in \mathcal{B}.
    \end{equation}

Moreover, we just showed that $(A^1_\om)^\star\simeq P_\om(L^\infty)$,
and hence each $\Psi\in (A^1_\om)^\star$ is of the form $\Psi(G)=\langle G,P_\om(h)\rangle_{A^2_\om}$ for some $h\in L^\infty$ with $\|\Psi\|=\|h\|_{L^\infty}$. Therefore, as $P_\om:L^\infty\to\B$ is bounded by the first part of the proof, for each $\Psi\in (A^1_\om)^\star$ there exists $f\in\B$ such that $\Psi(G)=\langle G,f\rangle_{A^2_\om}$ with $\|f\|_{\B}\le\|P_\om\|_{L^\infty\to\B}\|\Psi\|$,
and hence a known byproduct of the Hahn-Banach theorem implies
    \begin{equation}\label{eq:d2n}
    \|g_r\|_{A^1_\om}
    =\sup_{\Psi\in (A^1_\om)^\star}\frac{|\Psi(g_r)|}{\|\Psi\|}
    \le\|P_\om\|_{L^\infty\to\B}\sup_{f\in\mathcal{B}\setminus\{0\}}\frac{|\langle f,g_r\rangle_{A^2_\om}|}{\|f\|_{\mathcal{B}}}.
    \end{equation}
By combining \eqref{eq:d1n} and \eqref{eq:d2n}, and then letting $r\to1^-$, we deduce  
$\|g\|_{A^1_\om}\le\|P_\om\|_{L^\infty\to\B}\|L\|$. Finally, by using that $\lim_{r\to 1^-}\|f_r-f\|_{\mathcal{B}}=0$ for all $f\in\mathcal{B}_0$,
and \eqref{eq:d00n} we deduce
    $$
    L(f)=\lim_{r\to 1}L(f_r)=\lim_{r\to 1}\sum_{n=0}^\infty \widehat{f}(n)\overline{\widehat{g}(n)}\om_{2n+1}r^n
    =\langle f,g\rangle_{A^2_\om},\quad f\in\mathcal{B}_0.
    $$

Conversely, assume (iii). For each $n\in\N\cup\{0\}$, consider $L_n\in \B_0^\star$ defined by $L_n(f)=\widehat{f}(n)\om_{2n+1}$, where $f(z)=\sum_{n=0}^\infty \widehat{f}(n)z^n$. There exists an absolute constant $C>0$ such that
    $$
    |L_n(f)|\le C\| f\|_{\B} \om_{2n+1},\quad n\in\N\cup\{0\},
    $$
and hence $\|L_n\|=\sup_{f\in \B_0, f\neq 0}\frac{|L_n(f)|}{\|f\|_{\B}}\le C \om_{2n+1}$. Moreover, $2L_n(f)=\langle f, e_n\rangle_{A^2_\om}$, where $e_n(z)=z^n$. Therefore the hypothesis (iii) yields
    $$
    \om_{n+1}=\|e_n\|_{A^1_\om}\lesssim\|L\|\lesssim\om_{2n+1},\quad n\in\N\cup\{0\},
    $$
and thus $\om\in\DD$ by \cite[Lemma~2.1(ix)]{PelSum14}.
\end{Prf}

\medskip

Before providing a proof of Theorem~\ref{th:otromundo}
some notation and auxiliary results are needed. For a given compactly supported $C^\infty$-function $\Phi:\mathbb{R}\to\C$, set
    $$
    A_{\Phi,m}=\max_{x\in\mathbb{R}}|\Phi(x)|+m\max_{x\in\mathbb{R}}|\Phi^{(m)}(x)|,\quad m\in\N\cup\{0\},
    $$
and define
    \begin{equation}\label{eq:polynomials}
    W_n^\Phi(z)=\sum_{k\in\mathbb
    Z}\Phi\left(\frac{k}{n}\right)z^{k},\quad n\in\N.
    \end{equation}
The Hadamard product of $f(z)=\sum_{k=0}^\infty \widehat{f}(k)z^k\in\H(\D)$ and $g(z)=\sum_{k=0}^\infty \widehat{g}(k)z^k\in\H(\D)$ is
    $$
    (f\ast g)(z)=\sum_{k=0}^\infty\widehat{f}(k)\widehat{g}(k)z^k,\quad z\in\D.
    $$
A direct calculation shows that
    \begin{equation}\label{eq:hadprod}
    (f\ast g)(r^2e^{it})
    =\frac{1}{2\pi}\int_{-\pi}^\pi f(re^{i(t+\theta)})g(re^{-i\t})\,d\t.
    \end{equation}
The next result follows from the considerations on \cite[pp.~111--113]{Pabook}, see also \cite[Section 5.2]{Pabook2}.

\begin{lettertheorem}\label{th:cesaro}
Let $\Phi:\mathbb{R}\to\C$ be a compactly supported $C^\infty$-function. Then the following statements hold:
\begin{itemize}
\item[\rm(i)] There exists a constant $C>0$ such that
    $$
    \left|W_n^\Phi(e^{i\theta})\right|\le C\min\left\{
    n\max_{s\in\mathbb{R}}|\Phi(s)|,
    n^{1-m}|\theta|^{-m}\max_{s\in\mathbb{R}}|\Phi^{(m)}(s)|
    \right\},
    $$
for all $m\in\N\cup\{0\}$, $n\in\N$ and $0<|\theta|<\pi$.
\item[\rm(ii)] If $0<p\le 1$ and $m\in\N$ with $mp>1$, there exists a constant $C(p)>0$ such that
    $$
    \left(\sup_{n}\left|(W_n^\Phi\ast f)(e^{i\theta})\right|\right)^p\le CA^p_{\Phi,m}
    M(|f|^p)(e^{i\theta})
    $$
for all $f\in H^p$. Here $M$ denotes the Hardy-Littlewood
maximal-operator
    $$
    M(f)(e^{i\theta})=\sup_{0<h<\pi}\frac{1}{2h}\int_{\theta-h}^{\theta+h}|f(e^{it})|\,dt.
    $$
\item[\rm(iii)]
Then, for each $p\in(0,\infty)$ and $m\in\N$ with $mp>1$, there exists a constant
$C=C(p)>0$ such that
    $$
    \|W_n^\Phi\ast f\|_{H^p}\le C A_{\Phi,m}\|f\|_{H^p}
    $$
for all $f\in H^p$ and $n\in\N$.
\end{itemize}
\end{lettertheorem}

The property (iii) shows that the polynomials $\{W_n^\Phi\}_{n\in\N}$ can be seen as a universal C\'esaro basis for $H^p$ for any $0<p<\infty$. A particular case of the previous construction is useful for our purposes. By following \cite[Section~2]{JevPac98}, let $\Psi:\mathbb{R}\to\mathbb{R}$ be a $C^\infty$-function such that
    \begin{enumerate}
    \item $\Psi\equiv1$ on $(-\infty,1]$,
    \item $\Psi\equiv0$ on $[2,\infty)$,
    \item $\Psi$ is decreasing and positive on $(1,2)$,
    \end{enumerate}
and set $\psi(t)=\Psi\left(\frac{t}{2}\right)-\Psi(t)$ for all $t\in\mathbb{R}$. Let $V_0(z)=1+z$
and
    \begin{equation}\label{vn}
    V_n(z)=W^\psi_{2^{n-1}}(z)=\sum_{k=0}^\infty
    \psi\left(\frac{k}{2^{n-1}}\right)z^k=\sum_{k=2^{n-1}}^{2^{n+1}-1}
    \psi\left(\frac{k}{2^{n-1}}\right)z^k,\quad n\in\N.
    \end{equation}
These polynomials have the following properties with regard to
smooth partial sums, see \cite[p. 175--177]{JevPac98} or \cite[p. 143--144]{Pabook2} for details:
    \begin{equation}
    \begin{split}\label{propervn}
    &f(z)=\sum_{n=0}^\infty (V_n\ast f)(z),\quad f\in\H(\D),\\
    &\|V_n\ast f\|_{H^p}\le C\|f\|_{H^p},\quad f\in H^p,\quad 0<p<\infty,\\
    &\|V_n\|_{H^p}\asymp 2^{n(1-1/p)}, \quad 0< p<\infty.
    \end{split}
    \end{equation}

\medskip

\begin{Prf}{\em{Theorem~\ref{th:otromundo}.}}
Obviously (ii) implies (iii). To see that (iii) implies $\om\in\M$, let $g(z)=\sum_{k=0}^\infty z^{2^k}$. Then $g\in\B$ because $g$ is a lacunary series with bounded Maclaurin coefficients. For each $j,N\in\N\cup\{0\}$, consider the polynomials  $f_{j,N}(z)=\sum_{k=j}^{N+j}\frac{z^{2^k}}{\om^2_{2^{k+1}}}$ and $g_{j,N}(z)=\sum_{k=j}^{N+j} z^{2^k}$. By applying (iii) to $f=f_{j,N}$ and $g=g_{j,N}$ we deduce
    \begin{equation}\label{eq:otromundo11}
    \sum_{k=j}^{N+j}\frac{1}{\om_{2^{k+1}}}
    \lesssim\sum_{k=j}^{N+j}\frac{\om_{2^{k+1}+1}}{\om_{2^{k+1}}^2}
    \lesssim\|g_{j,N}\|_{\B}\| f_{j,N}\|_{A^1_\om}
    \le\|g\|_{\B}\| f_{j,N}\|_{A^1_\om}
    \asymp\| f_{j,N}\|_{A^1_\om},
    \end{equation}
where. 
    \begin{equation*}
    \begin{split}
    \|f_{j,N}\|_{A^1_\om}
    &\le2\int_0^1M_1(s,f_{j,N})\om(s)\,ds
    \le2\int_0^1M_2(s,f_{j,N})\om(s)\,ds\\
    &=2\int_{0}^{1-\frac{1}{2^{N+j}}} \left(\sum_{k=j}^{N+j}\frac{s^{2^{k+1}}}{\om^4_{2^{k+1}}}\right)^{\frac12}\om(s)\,ds+
    2\int_{1-\frac{1}{2^{N+j}}}^1
    \left(\sum_{k=j}^{N+j} \frac{s^{2^{k+1}}}{\om^4_{2^{k+1}}}\right)^{\frac12}\om(s)\,ds\\
    &=2(I_1+I_2).
    \end{split}
    \end{equation*}
Since
    \begin{equation*}
    \begin{split}
    I_1&
    \le\int_{0}^{1-\frac{1}{2^{N+j}}} s^{2^j} \left(\sum_{k=j}^{N+j}\frac{1}{\om^4_{2^{k+1}}}\right)^{\frac12}\om(s)\,ds
    \le\left(\sum_{k=j}^{N+j}\frac{1}{\om_{2^{k+1}}}\right)^{\frac12}\frac{\om_{2^{j}}}{\om^{\frac32}_{2^{N+j+1}}}
    \end{split}
    \end{equation*}
and
    \begin{equation*}
    \begin{split}
    I_2&
    \le\left(\sum_{k=j}^{N+j}\frac{1}{\om^4_{2^{k+1}}}\right)^{\frac12}\omg\left( 1-\frac{1}{2^{N+j}}\right)
    \lesssim\left(\sum_{k=j}^{N+j}\frac{1}{\om_{2^{k+1}}}\right)^{\frac12}\frac{\om_{2^{j}}}{\om^{\frac32}_{2^{N+j+1}}},
    \end{split}
    \end{equation*}
we obtain
   $$
   \left(\sum_{k=j}^{N+j}\frac{1}{\om_{2^{k+1}}}\right)^{\frac12}
   \lesssim\frac{\om_{2^{j}}}{\om^{\frac32}_{2^{N+j+1}}},\quad j,N\in\N\cup\{0\},
    $$
and hence
    \begin{equation}\label{eq:otromundo51}
    \om_{2^{N+j+1}} \le \frac{C}{(N+1)^{\frac13}} \om_{2^{j}}, \quad j,N\in\N\cup\{0\},
    \end{equation}
for some constant $C=C(\om)>0$. Finally, for each $x>1/2$, take $j\in\N\cup\{0\}$ such that $2^{j-1}<x\le 2^{j}$. Then
    $$
    \frac{C}{(N+1)^{\frac13}}\om_x
    \ge\frac{C}{(N+1)^{\frac13}}\om_{2^{j}}
    \ge\om_{2^{N+j+1}}
    \ge\om_{2^{N+2}x},\quad x>\frac12,
    $$
and by choosing $N$ large enough such that $\frac{C}{(N+1)^{\frac13}}<1$ we deduce $\om\in\M$ with $K=2^{N+2}$.

Assume next $\om\in\M$. To see that (i) is satisfied, we first transform the definition of $\M$ to a more suitable form for our purposes. Our goal is to show that if $\om\in\M$ then for each $\b>0$, there exists $C_1=C_1(\om,\b)>0$ such that
    \begin{equation}\label{eq:characterization-M}
    \om_x\le C_1x^\b\left(\om_{[\b]}\right)_x,\quad 1\le x<\infty.
    \end{equation}
To do this, let first $1\le x\le y<\infty$. Then there exist $m,k\in\N\cup\{0\}$ such that $m\ge k$, $K^k\le x<K^{k+1}$ and $K^m\le y<K^{m+1}$, where $K=K(\om)>1$ is that of the definition of the class $\M$. If $m=k$, then
    $$
    \om_x\ge\om_y=\om_y\left(\frac{y}{x}\right)^\eta\left(\frac{x}{y}\right)^\eta\ge\om_y\left(\frac{y}{x}\right)^\eta K^{-\eta},\quad \eta>0,
    $$
while if $m>k$ we have
    \begin{equation*}
    \begin{split}
    \om_x
    &\ge\om_{K^{k+1}}
    \ge C\om_{K^{k+2}}
    \ge\cdots
    \ge C^{m-k-1}\om_{K^m}\\
    &\ge C^{m-k-1}\om_y
    =K^{(m-k-1)\log_KC}\om_y
    \ge\left(\frac{y}{x}\right)^{\log_KC}\frac{\om_y}{C^2}.
    \end{split}
    \end{equation*}
By combining these estimates we deduce that if $\om\in\M$, then there exist $C_2=C_2(\om)=C^{-2}>0$ and $\eta=\eta(\om)=\log_KC>0$ such that
    \begin{equation}\label{eq:11}
    \om_x\ge C_2\left(\frac{y}{x}\right)^\eta\om_y,\quad 1\le x\le y<\infty.
    \end{equation}
The choice $y=Mx$ gives $\om_x\ge C_2M^\eta\om_{Mx}$ for all $1\le x<\infty$.
This shows that the constant $C=C(\om)>1$ appearing in the definition of $\M$ can be taken as large as we please by choosing~$K$ sufficiently large. Let now $M=M(\om)>2$ be sufficiently large such that $M>\left(\frac4{C_2}\right)^\frac1\eta$.
By choosing $y=Mx$ and replacing $x$ by $x/M$ in \eqref{eq:11}, we deduce
    \begin{equation*}
    \begin{split}
    \frac14\int_{1-\frac1x}^1\om(r)\,dr
    &\le\int_{1-\frac1x}^1r^x\om(r)\,dr\\
    &\le\frac1{C_2M^\eta}\int_0^{1-\frac1x}r^\frac{x}M\om(r)\,dr
		+\frac1{C_2M^\eta}\int_{1-\frac1x}^1r^\frac{x}M\om(r)\,dr
		-\int_0^{1-\frac1x}r^x\om(r)\,dr\\
    &\le\frac1{C_2M^\eta}\int_0^{1-\frac1x}r^\frac{x}M\om(r)\,dr
		+\frac1{C_2M^\eta}\int_{1-\frac1x}^1\om(r)\,dr,\quad x\ge M,
    \end{split}
    \end{equation*}
and hence
    \begin{equation}\label{10}
    \widehat{\om}(t)\le C_3\int_0^ts^{\frac1{M(1-t)}}\om(s)\,ds,\quad 1-\frac{1}{M}\le t<1,
    \end{equation}
where $C_3=\frac{4}{C_2M^\eta-4}$.

To achieve \eqref{eq:characterization-M}, it suffices to show that
    $$
    \int_{1-\frac{1}{Mx}}^{1} s^{x}\om(s)\,ds\le C_1x^\b\left(\om_{[\b]}\right)_x-\int_0^{1-\frac{1}{Mx}} s^{x}\om(s),\quad 1\le x<\infty.
    $$
But \eqref{10} implies
    $$
    \int_{1-\frac{1}{Mx}}^{1} s^{x}\om(s)\,ds
		\le\omg\left(1-\frac{1}{Mx}\right)
		\le C_3\int_{0}^{1-\frac{1}{Mx}}s^{x}\om(s)\,ds,\quad 1\le x<\infty,
    $$
and therefore, to establish \eqref{eq:characterization-M}, it remains to prove
    \begin{equation*}
    \begin{split}
    C_3\int_{0}^{1-\frac{1}{Mx}}s^{x}\om(s)\,ds
    \le C_1x^\b\left(\om_{[\b]}\right)_{x}-\int_0^{1-\frac{1}{Mx}}s^{x}\om(s),
    \end{split}
    \end{equation*}
that is,
    \begin{equation}
    \begin{split}\label{eq:ot4}
    \int_{0}^{1-\frac{1}{Mx}} s^{x}\om(s)\,ds
    \le\frac{C_1}{C_3+1}x^\b\left(\om_{[\b]}\right)_{x},\quad 1\le x<\infty.
    \end{split}
    \end{equation}
But clearly,
    \begin{equation*}
    \begin{split}
    \int_{0}^{1-\frac{1}{Mx}} s^{x}\om(s)\,ds
    &=\int_{0}^{1-\frac{1}{Mx}}s^{x}\om(s)\frac{(1-s)^\beta}{(1-s)^\beta}\,ds\\
    &\le M^\b x^\b\int_{0}^{1-\frac{1}{Mx}}s^{x}\om(s)(1-s)^\beta\,ds
    \le M^\b x^\b\left(\om_{[\b]}\right)_{x},
    \end{split}
    \end{equation*}
which gives \eqref{eq:ot4} with $C_1=(C_3+1)M^\b$, where $M$ and $C_3$ are those of \eqref{10}.

We proceed to show (i) by using \eqref{eq:characterization-M}. Let $f\in\B$ and $\b>0$ be given. Denote 
$D^\b h(z)=\sum_{n=0}^\infty(n+1)^\b\widehat{h}(n)z^n$ for each $h(z)=\sum_{n=0}^\infty\widehat{h}(n)z^n$. Then
    \begin{equation}\label{2}
    g(z)=(1-|z|)^\b D^\b h(z)=(1-|z|)^\b\sum_{n=0}^\infty(n+1)^\b\widehat{h}(n) z^n,\quad z\in\D,
    \end{equation}
satisfies
    \begin{equation*}
    \begin{split}
    P_\om(g)(z)
		&=\int_{\D}(1-|\z|)^\b\left(\sum_{n=0}^\infty(n+1)^\b\widehat{h}(n)\z^n \right)\overline{B^\om_z(\z)}\om(\z)\,dA(\z)\\
    &=\int_0^1\sum_{n=0}^\infty\frac{(n+1)^\b}{\om_{2n+1}}\widehat{h}(n)z^n(1-r)^\b r^{2n+1}\omega(r)\,dr\\
    &=\sum_{n=0}^\infty\frac{(n+1)^\b(\om_{[\b]})_{2n+1}}{\om_{2n+1}} \widehat{h}(n) z^n,\quad z\in\D.
    \end{split}
    \end{equation*}
By choosing $\widehat{h}(n)=\widehat{f}(n)\frac{\om_{2n+1}}{(n+1)^\b(\om_{[\b]})_{2n+1}}$ for all $n\in\N\cup\{0\}$, we have $P_\om(g)=f$. It remains to prove $\|g\|_{L^\infty}\le C\|f\|_{\B}$ for some constant $C=C(\om)>0$. A calculation shows that
    \begin{equation*}
    \begin{split}
    \|g\|_{L^\infty}
    &=\sup_{z\in\D}|D^\b h(z)|(1-|z|)^\b\asymp\|f\ast\lambda\|_{\B}
    \asymp\sup_{z\in\D}|D^{2\b}h(z)|(1-|z|)^{2\b}\\
    &=\sup_{z\in\D}|D^{2\b}(f*\lambda)(z)|(1-|z|)^{2\b},
    \end{split}
    \end{equation*}
where $\lambda(z)=\sum_{n=0}^\infty\widehat{\lambda}(n)z^n=\sum_{n=0}^\infty
\frac{\om_{2n+1}}{(n+1)^\b(\om_{[\b]})_{2n+1}}z^n$. Therefore, what actually remains to be shown is that $\lambda$ is a coefficient multiplier from the Bloch space into itself, that is, $\|f*\lambda\|_\B\lesssim\|f\|_\B$. To see this, note first that \eqref{eq:hadprod} yields
    \begin{equation*}
    \begin{split}
    |D^\b(D^\b\lambda\ast f)(r^2e^{it})|(1-r)^{2\b}
    &=\left|\frac{1}{2\pi}\int_{-\pi}^\pi D^\b\lambda(re^{i(t+\theta)})D^\b f(re^{-i\t})\,d\t\right|(1-r)^{2\b}\\
    &\le M_\infty(r,D^\b f)M_1(r,D^\b\lambda)(1-r)^{2\b}
    \lesssim\|f\|_\B M_1(r,D^\b\lambda)(1-r)^\b,
    \end{split}
    \end{equation*}
and hence it suffices to show that
    \begin{equation}\label{eq:coefmult}
    \sup_{0<r<1}M_1(r, D^\b\lambda)(1-r)^\b<\infty.
    \end{equation}
In fact, the condition \eqref{eq:coefmult} describes the coefficient multipliers from the Bloch space into itself, see \cite[Theorem~3]{ShiWiTams71} for a proof of the case $\b=1$. To obtain \eqref{eq:coefmult} we will use the families of polynomials defined by \eqref{eq:polynomials} and \eqref{vn}. It follows from \eqref{propervn} that
    \begin{equation}\label{fbnorm}
    M_1(r,D^\b\lambda)=\|(D^\b\lambda)_r\|_{H^1}\le C(\om)+\sum_{n=2}^\infty\|V_n\ast(D^\b\lambda)_r\|_{H^1},
    \end{equation}
where $(D^\b\lambda)_r(z)=\sum_{n=0}^\infty\frac{\om_{2n+1}}{(\om_{[\b]})_{2n+1}}r^nz^n$. Next, for each $n\in\N\setminus\{1\}$ and
$r\in\left[\frac12,1\right)$, consider
    $$
    F_n(x)=\frac{\om_{2x+1}}{(\om_{[\b]})_{2x+1}}r^x\chi_{[2^{n-1},2^{n+1}]}(x),\quad x\in\RR.
    $$
Since for each radial weight $\nu$ there exists a constant $C=C(\nu)>0$ such that
    \begin{equation*}
    \int_0^1 s^x\left(\log\frac1s\right)^n\nu(s)\,ds\le C\nu_x, \quad n\in\{1,2\},\quad x\ge 2,
    \end{equation*}
it follows by a direct calculation that
    $$
    |F_n''(x)|\le C|F_n(x)|, \quad n\in\N\setminus\{1\}, \quad r\in\left[\frac12,1\right),\quad x\ge2,
    $$
for some constant $C=C(\om,\beta)>0$. Therefore \eqref{eq:characterization-M} yields
    \begin{equation*}
    \begin{split}
    A_{F_n,2}&=\max_{x\in[2^{n-1},2^{n+1}]}|F_n(x)|+\max_{x\in[2^{n-1},2^{n+1}]}|F_n''(x)|
		\lesssim \max_{x\in[2^{n-1},2^{n+1}]}|F_n(x)|\\
    &
    \lesssim \max_{x\in[2^{n-1},2^{n+1}]} (2x+1)^\beta r^x
    \lesssim 2^{n\beta} r^{2^{n-1}},\quad n\in\N\setminus\{1\}.
    \end{split}
    \end{equation*}
For each $n\in\N\setminus\{1\}$, choose a $C^\infty$-function $\Phi_n$ with compact support contained in $[2^{n-2},2^{n+2}]$ such that
$\Phi_n=F_n$ on $[2^{n-1},2^{n+1}]$ and
    \begin{equation}\label{phin}
    A_{\Phi_n,2}=\max_{x\in\mathbb{R}}|\Phi_n(x)|+\max_{x\in\mathbb{R}}|\Phi''_n(x)|
    \lesssim  2^{n\beta} r^{2^{n-1}},\quad n\in\N\setminus\{1\}.
    \end{equation}
Since
    $$
    W_1^{\Phi_n}(z)
    =\sum_{k\in\mathbb{Z}}\Phi_n\left(k\right)z^k
    =\sum_{k\in\mathbb{Z}\cap[2^{n-2},2^{n+2}]}\Phi_n\left(k\right)z^k,
    $$
the identity \eqref{vn} yields
    \begin{equation*}
    \begin{split}
    \left(V_n\ast(D^\b\lambda)_r\right)(z)
    &=\sum_{k=2^{n-1}}^{2^{n+1}-1} \psi\left(\frac{k}{2^{n-1}}\right)\frac{\om_{2k+1}}{(\om_{[\b]})_{2k+1}}r^kz^k
    =\sum_{k=2^{n-1}}^{2^{n+1}-1}\psi\left(\frac{k}{2^{n-1}}\right)\Phi_n(k)z^k\\
    &=\left(W_1^{\Phi_n}\ast V_n\right)(z),\quad n\in\N\setminus\{1\}.
    \end{split}
    \end{equation*}
This together with Theorem~\ref{th:cesaro}(iii), \eqref{phin} and \eqref{propervn} implies
    \begin{equation*}
    \begin{split}
    \|V_n\ast(D^\b\lambda)_r\|_{H^1}
    &=\|W_1^{\Phi_n}\ast V_n\|_{H^1}
    \lesssim A_{\Phi_n,2}\|V_n\|_{H^1}
    \lesssim  2^{n\beta} r^{2^{n-1}}\|V_n\|_{H^1}\\
    &\lesssim 2^{n\beta} r^{2^{n-1}},\quad r\in \left[\frac12,1\right),\quad n\in\N\setminus\{1\},
    \end{split}
    \end{equation*}
which combined with \eqref{fbnorm} gives
    \begin{equation}\label{endstep1new1}
    M_1(r,D^\b\lambda)
    \lesssim\sum_{n=2}^\infty2^{\b n}r^{2^{n-1}}
    \lesssim\frac{1}{(1-r)^\b}, \quad  r\in\left[\frac12,1\right).
    \end{equation}
This implies \eqref{eq:coefmult}, and thus (i) is satisfied.

To complete the proof, it remains to show that (i) implies (ii). Let $f\in A^1_\om$ and $g\in\B$. By the hypothesis (i), we may choose $h\in L^\infty$ such that $P_\om(h)=g$ and $\|h\|_{L^\infty}\le C\|g\|_\B$, where $C=C(\om)>0$. The identity
\eqref{eq:fubinisubs} yields
    \begin{equation*}
    \begin{split}
    |\langle f,g\rangle_{A^2_\om}|&=\lim_{r\to 1^-}\left|\langle f_r,P_\om(h)\rangle_{L^2_\om}\right|
		=\lim_{r\to 1^-}\left|\langle f_r,h\rangle_{L^2_\om}\right|
		\le\|h\|_{L^\infty}\|f\|_{A^1_\om}\le C\|g\|_\B\|f\|_{A^1_\om},
    \end{split}
    \end{equation*}
and thus (ii) is satisfied and the proof is complete.
\end{Prf}
\medskip

The proof of Theorem~\ref{th:otromundo} shows that $\om\in\M$ if and only if for some (equivalently for each) $\b>0$, there exists $C_1=C_1(\om,\b)>0$ such that \eqref{eq:characterization-M} is satisfied. Later, when Littlewood-Paley estimates are considered, we will find an analogous characterization for weights in $\DD$, see \eqref{Corollary:D-hat} below. Another fact that follows from the proof and is useful for our purposes states that $\om\in\M$ if and only if \eqref{10} is satisfied for some $C_3=C_3(\om)>0$ and $M=M(\om)>1$.

Next, we will describe the radial weights $\om$ such that $P_\om:L^\infty\to\B$ is bounded and onto. Recall that $\DDD=\DD\cap\Dd$ by the definition.

\medskip
\begin{Prf}{\em{Theorem~\ref{th:projectionboundedonto}.}}
In view of Theorems~\ref{ELTEOREMA} and~\ref{th:otromundo}, it suffices to show that $\DDD=\DD\cap\M$. Let first $\om\in\DD\cap\M$. By \cite[Lemma~2.1(iv)]{PelSum14} there exists $C=C(\om)>0$ such that
	\begin{equation*}
	\int_0^rs^{\frac1{1-r}}\om(s)\,ds\le C\widehat{\om}(r),\quad 0\le r<1.	
	\end{equation*}
This applied to $r=1-M(1-t)$ gives
	\begin{equation*}
	\begin{split}
	\int_0^ts^{\frac1{M(1-t)}}\om(s)\,ds
	&=\int_0^{1-M(1-t)}s^{\frac1{M(1-t)}}\om(s)\,ds
	+\int_{1-M(1-t)}^ts^{\frac1{M(1-t)}}\om(s)\,ds\\
	&\le C\widehat{\om}\left(1-M(1-t)\right)+\int_{1-M(1-t)}^t\om(s)\,ds\\
	&=\left(C+1\right)\widehat{\om}\left(1-M(1-t)\right)-\widehat{\om}\left(t\right),
	\end{split}
	\end{equation*}
which together with \eqref{10} yields
\begin{equation*}
	\begin{split}
	\widehat{\om}(t)
	\le\frac{4}{C_2M^\eta}\left(C+1\right)\widehat{\om}\left(1-M(1-t)\right).
	\end{split}
	\end{equation*}	
By choosing $M=M(\om)>2$ sufficiently large such that $M>\left(\frac{4(C+1)}{C_2}\right)^\frac1\eta$, we deduce $\widehat{\om}(1-M(1-t))\ge C_4\widehat{\om}(t)$ for some $C_4>1$ and all $1-\frac1M\le t<1$. It follows that $\om\in\Dd$ and thus $\om\in\DDD$.

To complete the proof it suffices to prove $\Dd\subset\M$. To do this, we first observe that $\om\in\Dd$ if and only if for some (equivalently for each) $\g>0$, there exists $C=C(\g,\om)>0$ such that
    \begin{equation}\label{eq:Dd-condition}
    \int_r^1 \frac{\omg(s)^\g}{1-s}\,ds\le C \omg(r)^\g,\quad 0\le r<1.
    \end{equation}
Namely, if \eqref{eq:Dd-condition} is satisfied for some $\gamma>0$, then, for each $K>1$, we have
    $$
    C\widehat{\om}(r)^\gamma\ge\int_r^1\frac{\omg(s)^\g}{1-s}\,ds\ge\int_r^{1-\frac{1-r}{K}}\frac{\omg(s)^\g}{1-s}\,ds\ge\omg\left(1-\frac{1-r}{K} \right)^\g\log K,
    $$
and hence, by choosing $K$ sufficiently large such that $\log K>C$ we obtain $\om\in\Dd$. Conversely, if $\om\in\Dd$ and $\gamma>0$, then
    $$
    \int_r^1\frac{\omg(s)^\g}{1-s}\,ds
    \ge C^\g\int_r^1\frac{\omg\left(1-\frac{1-s}{K}\right)^\g}{1-s}\,ds
    =C^\g\int_{1-\frac{1-r}{K}}^1\frac{\omg(s)^\g}{1-s}\,ds,
    $$
and hence
    \begin{equation*}
    \begin{split}
    \int_r^1\frac{\omg(s)^\g}{1-s}\,ds
    &\le\int_r^{1-\frac{1-r}{K}}\frac{\omg(s)^\g}{1-s}\,ds+\frac{1}{C^\g}\int_r^1\frac{\omg(s)^\g}{1-s}\,ds
		\le\omg(r)^\g\log K+\frac{1}{C^\g}\int_r^1\frac{\omg(s)^\g}{1-s}\,ds,
    \end{split}
    \end{equation*}
that is,
    $$
    \int_r^1\frac{\omg(s)^\g}{1-s}\,ds\le\omg(r)^\g \frac{C^\g}{C^\g-1}\log K,\quad 0\le r<1.
    $$
Thus \eqref{eq:Dd-condition} is satisfied.

Let $\om\in\Dd$ and write $\widetilde{\om}(r)=\widehat{\om}(r)/(1-r)$ for all $0\le r<1$. Let $f\in A^p_\om$ with $f(0)=0$. Then \eqref{eq:Dd-condition} with $\gamma=1$ yields
    $$
    M^p_p(r,f)\widehat{\widetilde{\om}}(r)\lesssim M^p_p(r,f)\widehat{\om}(r)\le\frac1r\int_{\D\setminus D(0,r)}|f(z)|^p\om(z)\,dA(z)\to0,\quad r\to1^-.
    $$
Hence, an integration by parts together with another application of \eqref{eq:Dd-condition} gives
  \begin{equation}\label{eq:intbyparts}
  \begin{split}
  \int_{0}^{1}M_p^p(r,f)\widetilde{\om}(r)\,dr
  &=\int_{0}^{1} \frac{\partial}{\partial r}M_p^p(r,f)\widehat{\widetilde{\om}}(r)\,dr
  \lesssim\int_{0}^{1} \frac{\partial}{\partial r}M_p^p(r,f)\omg(r)\,dr\\
  &=\int_{0}^{1}M_p^p(r,f)\om(r)\,dr,\quad f\in A^p_\om,\quad f(0)=0,
  \end{split}
  \end{equation}
and it follows that $\|f\|_{A^p_{\widetilde{\om}}}\lesssim\|f\|_{A^p_\om}$ for all $f\in\H(\D)$.

Denote $\omega^\star(z)=\int_{|z|}^1\log\frac{s}{|z|}\omega(s)s\,ds$ for all $z\in\D\setminus\{0\}$.
Now apply Green's formula and the Cauchy-Schwarz inequality, then use an equivalent norm in $A^1_\om$ in terms of the Laplacian of $|f|$ given in \cite[Theorem~4.2]{PelRat}, and finally the norm inequality just obtained to deduce
    \begin{equation*}
    \begin{split}
    \left|\langle f,g\rangle_{A^2_\om}\right|
    &\lesssim\left|\langle f',g'\rangle_{A^2_{\om^\star}}\right|+|f(0)||g(0)|
    \le\|g\|_\B\int_\D|f'(z)|\frac{\om^\star(z)}{1-|z|}\,dA(z)+|f(0)||g(0)|\\
    &\lesssim\|g\|_\B\|\Delta|f|\|_{L^1_{\om^\star}}^\frac12\left(\int_\D|f(z)|\frac{\om^\star(z)}{(1-|z|)^2}\,dA(z)\right)^\frac12
		+|f(0)||g(0)|\\
    &\lesssim\|g\|_\B\|f\|_{A^1_\om}^\frac12\left(\int_\D|f(z)|\frac{\widehat{\om}(z)}{|z|(1-|z|)}\,dA(z)\right)^\frac12
		\lesssim\|g\|_\B\|f\|_{A^1_\om},\quad g\in\B,\quad f\in A^1_\om.
    \end{split}
    \end{equation*}
Therefore $\om\in\M$ by Theorem~\ref{th:otromundo}.
\end{Prf}

\medskip

Note that (iv) implies (i) in Theorem~\ref{th:projectionboundedonto} is easy to prove without appealing to Theorem~\ref{th:otromundo}, where a concrete preimage is constructed, see~\cite{PeralaRattya2017} for details. The proof of Theorem~\ref{th:projectionboundedonto} shows that $\Dd\subset\M$. This inclusion is strict as is seen next.

\begin{proposition}\label{pr:counterxample}
The class $\Dd$ is a proper subset of $\M$.
\end{proposition}

\begin{proof}
To construct a weight in $\M\setminus\Dd$, let $\om$ be a radial weight and $\{t_n\}_{n\in\N}$ an increasing sequence on $[0,1)$ such that $t_n\to1^-$ as $n\to\infty$. Let $\vp:[t_1,1)\to[1,\infty)$ be an increasing unbounded function such that $s_n=1-\frac{1-t_n}{\vp(t_n)}<t_{n+1}$ for all $n\in\N$. For $N\in\N$, define
    $$
    W(r)=W_{\om,N}(r)=\om(r)\sum_{n=N}^\infty\chi_{[s_n,t_{n+1}]}(r),\quad 0\le r<1.
    $$
Since $\vp$ is unbounded, for each prefixed $K>1$, $\widehat{W}(t_n)=\widehat{W}\left(1-\frac{1-t_n}{K}\right)$ for all $n\in\N$ sufficiently large; this is the case when $t_n>\vp^{-1}(K)\in(0,1)$. Consequently, $W\notin \Dd$.

We will show next that if $t_n=1-2^{-2^n}$ for all $n\in\N$, the weight $\om$ is defined by
    $$
    \widehat{\om}(r)=\exp\left(-\exp\frac1{1-r}\right),\quad 0\le r<1,
    $$
and
    $$
    \vp(r)=\log_2\log_2\frac1{1-r},\quad t_1=\frac34\le r<1,
    $$
then $W=W_{\om,2}\in\M$. First, observe that with these choices
    $$
    s_n=1-\frac{1-t_n}{\vp(t_n)}=1-\frac1{n2^{2^n}}<1-\frac1{2^{2^{n+1}}}=t_{n+1},\quad n\in\N.
    $$
For $\frac{15}{16}\le r<1$, choose $n\in\N\setminus\{1\}$ such that $t_n\le r<t_{n+1}$, and let $K>1$. Then
    \begin{equation*}
    \widehat{W}(r)\le\widehat{W}(s_n)<\widehat{\om}(s_n)=\exp\left( -\exp\left( \frac{1}{1-s_n}\right)\right)
    \end{equation*}
and for $K>4/3$ we have
		\begin{equation*}
    \begin{split}
    \int_0^r s^{\frac{1}{K(1-r)}}W(s)\,ds & \ge  \int_0^{t_n} s^{\frac{1}{K(1-r)}}W(s)\,ds
    \ge s_{n-1}^{\frac{1}{K(1-t_{n+1})}} \int_{s_{n-1}}^{t_n} \om(s)\,ds\\
    &=\left( s_{n-1}^{\frac{s_{n-1}}{1-s_{n-1}}}\right)^{\frac{1-s_{n-1}}{s_{n-1}K(1-t_{n+1})}}
    \left[\exp\left( -\exp\left( \frac{1}{1-s_{n-1}}\right)\right)-\exp\left( -\exp\left( \frac{1}{1-t_{n}}\right)\right) \right]\\
    &\ge\exp\left(- \frac{1-s_{n-1}}{s_{n-1}K(1-t_{n+1})}\right) \left[\exp\left( -\exp\left( \frac{1}{1-s_{n-1}}\right)\right)-\exp\left( -		\exp\left(  \frac{1}{1-t_{n}}\right)\right) \right]\\
    &\ge\frac{e-1}{e}\exp\left(- \frac{1-s_{n-1}}{s_{n-1}K(1-t_{n+1})}-\exp\left( \frac{1}{1-s_{n-1}}\right)\right)\\
    &\ge\frac{e-1}{e}\exp\left(-\exp\left( \frac{1}{1-s_{n}}\right)\right).
    \end{split}
    \end{equation*}
It follows that \eqref{10} is satisfied for some $C=C(\om)>0$ and $K=K(\om)>4/3$, and hence $W\in\M$ by the observation right after the proof of Theorem~\ref{th:otromundo}.
\end{proof}

By Theorem~\ref{th:projectionboundedonto}, $P_\om:L^\infty\to\B$ is bounded and onto if and only if $\om\in\DDD$. We next show that for given $\om\in\DDD$ we can find a radial weight $W$ which induces a weighted Bergman space only slightly larger than $A^p_\om$ but
the weight lies outside of $\DD\cup\M$.

\begin{theorem}\label{theorem:non-regular weights close to D}
Let $\om\in\DDD$ and $\vp:[0,1]\to[0,\infty)$ decreasing such that $\lim_{r\to1^-}\vp(r)=0$. Then there exists a weight $W=W_{\om,\vp}$ such that $P_W:L^\infty\to\B$ is neither bounded nor~$\B$ is continuously embedded into $P_W(L^\infty)$, but $A^p_\om\subset A^p_W\subset A^p_{\om\vp}$ for all $0<p<\infty$. If $\vp$ satisfies $-\vp'(t)/\vp(t)\lesssim1/(1-t)$ for all $0\le t<1$, then $\vp\om\in\DDD$.
\end{theorem}

\begin{proof}
We first observe that $\om\in\Dd$ if and only if there exist $C=C(\om)>0$ and $\b=\b(\om)>0$ such that
    \begin{equation}\label{6}
    \begin{split}
    \widehat{\om}(t)\le C\left(\frac{1-t}{1-r}\right)^{\b}\widehat{\om}(r),\quad 0\le r\le t<1.
    \end{split}
    \end{equation}
To construct a weight $W$ with desired properties, let $\psi:[1,\infty)\to(0,\infty)$ an increasing unbounded function. Let $W(r)=W_{\om,\psi}(r)=\om(r)\sum_{j=N}^\infty\chi_{[r_{2j+1},r_{2j+2}]}(r)$, where $N=N(\om,\psi)\in\N$ will be fixed later and $r_x=1-\frac{1}{2^{x\psi(x)}}$ for all $x\ge1$. Then
    \begin{equation}\label{20'}
    \frac{1-r_x}{1-r_{x+1}}=\frac{2^{(x+1)\psi(x+1)}}{2^{x\psi(x)}}=2^{x(\psi(x+1)-\psi(x))+\psi(x+1)}\ge 2^{\psi(x+1)}\to\infty,\quad x\to\infty.
    \end{equation}
Since $\om\in\Dd$, \eqref{6} shows that for each $C>1$ there exists $K=K(C,\om)>1$ such that
    \begin{equation}\label{22'}
    \widehat{\om}(t)\le\frac1C\widehat{\om}\left(1-K(1-t)\right),\quad t\ge1-\frac1K.
    \end{equation}
Take $C=2$ and fix $K>4$ accordingly. Then, for $1\le x<\infty$ and $0<y<\log_2 K$,
    \begin{equation}\label{23'}
    \widehat{\om}(r_{x+y})
    \le\frac12\widehat{\om}\left(1-K(1-r_{x+y})\right)
    =\frac12\widehat{\om}\left(1-\frac1{2^{(x+y)\psi(x+y)-\log_2K}}\right)\le\frac12\widehat{\om}(r_x)
    \end{equation}
if $x\ge\psi^{-1}\left(\frac{\log_2K}{y}\right)$. For a moment take $N=N(\om,\psi)\ge\frac12\left(\psi^{-1}\left(\log_2K\right)-1\right)$. Then \eqref{23'} can be applied to all $x=2j+1\ge2N+1$ if $1\le y<\log_2 K\in(2,\infty)$. The case $y=2$ implies
    \begin{equation}\label{19'}
    \begin{split}
    \widehat{W}(r_{2j})
    &=\sum_{k=j}^\infty\left(\widehat{\om}(r_{2k+1})-\widehat{\om}(r_{2k+2})\right)
    \le\sum_{k=j}^\infty\widehat{\om}(r_{2k+1})\le2\widehat{\om}(r_{2j+1}),\quad j\ge N.
    \end{split}
    \end{equation}
Let $r_j^\star=1-K(1-r_{2j+2})$ for all $j\ge N$. Then $r_{2j+1}<r_j^\star<1-\frac{1-r_j^\star}{K}=r_{2j+2}$ by the last inequality in \eqref{23'} with $y=1$, and hence $\widehat{W}\left(1-\frac{1-r_j^\star}{K}\right)=\widehat{W}\left(r_{2j+2}\right)$ and
    \begin{equation*}
    \begin{split}
    \widehat{W}\left(r_j^\star\right)=\int_{r_j^\star}^{1-\frac{1-r_j^\star}{K}}\om(r)\,dr+\widehat{W}\left(1-\frac{1-r_j^\star}{K}\right).
    \end{split}
    \end{equation*}
Therefore there exists $C_1=C_1(\om,K)>1$ such that
    \begin{equation}\label{21'}
    \widehat{W}\left(r^\star_{j}\right)\le C_1\widehat{W}\left(1-\frac{1-r_j^\star}{K}\right), \quad j\ge N,
    \end{equation}
if and only if
    $$
    \int_{r_j^\star}^{1-\frac{1-r_j^\star}{K}}\om(r)\,dr\le(C_1-1)\widehat{W}\left(1-\frac{1-r_j^\star}{K}\right), \quad j\ge N.
    $$
But $\int_{r_j^\star}^{1-\frac{1-r_j^\star}{K}}\om(r)\,dr\ge(C-1)\widehat{\om}\left(1-\frac{1-r_j^\star}{K}\right)$ by \eqref{22'}, and $\widehat{W}\left(1-\frac{1-r_j^\star}{K}\right)=\widehat{W}(r_{2j+2})\le2\widehat{\om}(r_{2j+3})$ by \eqref{19'}, so a necessary condition for \eqref{21'} to hold is $\widehat{\om}(r_{2j+2})\lesssim\widehat{\om}(r_{2j+3})$ for all $j\ge N$. But since $\om\in\Dd$, this fails for $j$ large enough by \eqref{6} and \eqref{20'}, and hence we deduce $W\notin\DD$. Thus $P_W:L^\infty\to\B$ is not bounded by Theorem~\ref{ELTEOREMA}.

The first part of the proof does not require much of $\psi$. Now assume there exists a constant $C_2>0$ such that $\psi(x+1)-\psi(x)\le C_2/x$ for all $x\ge1$. Then \eqref{20'} shows that
    \begin{equation}\label{17'}
    2^{\psi(x+1)}\le\frac{1-r_x}{1-r_{x+1}}\le2^{2C_2}2^{\psi(x)},\quad x\ge1.
    \end{equation}
To see that $W\notin\M$, it is equivalent to find a contradiction with \eqref{10} for each prefixed $M>1$. Assume for a moment that
    \begin{equation}\label{18'}
    \int_{0}^{r_{2j+1}} s^{\frac{1}{M(1-r_{2j+1})}}W(s)\,ds\lesssim r_{2j}^{\frac{1}{M(1-r_{2j+1})}}
    \widehat{\omega}(r_{2j-1}),\quad j\ge N.
    \end{equation}
Then, since $\om\in\DD$ by the hypothesis, \cite[Lemma~2.1]{PelSum14} and \eqref{17'} imply
    \begin{equation*}
    \begin{split}
    \int_{0}^{r_{2j+1}} s^{\frac{1}{M(1-r_{2j+1})}}W(s)\,ds
    &\lesssim r_{2j}^{\frac{1}{M(1-r_{2j+1})}}\widehat{\om}(r_{2j-1})\\
    &\lesssim\left(r_{2j}^{\frac{1}{(1-r_{2j})}}\right)^{\frac{1-r_{2j}}{M(1-r_{2j+1})}}\left(\frac{1-r_{2j-1}}{1-r_{2j+1}}\right)^\b
		\widehat{\om}(r_{2j+1})\\
    &\lesssim e^{-\frac{2^{\psi(2j+1)}}{M}}2^{\b(\psi(2j+1)+\psi(2j))}\widehat{\om}(r_{2j+1})\\
    &\le e^{-\frac{2^{\psi(2j+1)}}{M}}2^{2\b\psi(2j+1)}\widehat{\om}(r_{2j+1}),\quad j\ge N,
    \end{split}
    \end{equation*}
and
    $$
    \widehat{W}(r_{2j+1})\ge\widehat{\om}(r_{2j+1})-\widehat{\om}(r_{2j+2})\ge\frac12\widehat{\om}(r_{2j+1}),\quad j\ge N,
    $$
by \eqref{23'}. These estimates yield a contradiction as $j\to\infty$. It remains to prove \eqref{18'}. First observe that
    \begin{equation*}
    \begin{split}
    \int_{0}^{r_{2j+1}}s^{\frac{1}{M(1-r_{2j+1})}}W(s)\,ds
    &\le\sum_{k=N}^jr_{2k}^{\frac{1}{M(1-r_{2j+1})}}\widehat{\om}(r_{2k-1})
		=r_{2j}^{\frac{1}{M(1-r_{2j+1})}}\widehat{\om}(r_{2j-1})S,
    \end{split}
    \end{equation*}
where, by \cite[Lemma~2.1]{PelSum14},
	\begin{equation*}
	\begin{split}
	S&=\sum_{k=N}^j\frac{\left(r_{2k}^{\frac{1}{1-r_{2k}}}\right)^\frac{1-r_{2k}}{M(1-r_{2j+1})}}
	{\left(r_{2j}^{\frac{1}{1-r_{2j}}}\right)^\frac{1-r_{2j}}{M(1-r_{2j+1})}}\frac{\widehat{\om}(r_{2k-1})}{\widehat{\om}(r_{2j-1})}\\
	&\lesssim1+\sum_{k=N}^{j-1}
	e^{-\frac{1-r_{2k}}{M(1-r_{2j+1})}+\frac1{r_{2j}}\frac{1-r_{2j}}{M(1-r_{2j+1})}+\b\log\frac{1-r_{2k-1}}{1-r_{2j-1}}}\\
	&\le1+\sum_{k=N}^{j-1}
	e^{-\frac{1-r_{2k}}{M(1-r_{2j+1})}+\frac1{r_{2j}}\frac{1-r_{2j}}{M(1-r_{2j+1})}+\b\log\frac{1-r_{2k}}{1-r_{2j+1}}
	+\b\log\frac{1-r_{2k-1}}{1-r_{2k}}}\\
	&\le1+\sum_{k=N}^{j-1}
	e^{-\frac{1-r_{2k}}{M(1-r_{2j+1})}+\frac1{r_{2j}}\frac{1-r_{2j}}{M(1-r_{2j+1})}+\b\log\frac{1-r_{2k}}{1-r_{2j+1}}
	+2\b C_2\log2+\b\log\frac{1-r_{2j}}{1-r_{2j+1}}}\\
	&\le1+e^{2\b C_2\log2}\sum_{k=N}^{j-1}
	e^{-\frac{1-r_{2k}}{M(1-r_{2j+1})}+\left(\frac1{r_{2N}}+\b\right)\frac{1-r_{2j}}{M(1-r_{2j+1})}+\b\log\frac{1-r_{2k}}{1-r_{2j+1}}}
	\end{split}
	\end{equation*}
for some constant $\b=\b(\om)>0$. By \eqref{17'} we may now fix $N=N(\om,\psi)$ sufficiently large such that
	$$
	\b\log\frac{1-r_{2k}}{1-r_{2j+1}}\le\frac1{2M}\frac{1-r_{2k}}{1-r_{2j+1}},\quad N\le k\le j-1,
	$$
and
	$$
	\frac{1-r_{2j}}{1-r_{2j-2}}\le\frac{1}{4\left(\frac1{r_{2N}}+\b\right)},\quad N\le j-1.
	$$
Then another application of \eqref{17'} yields
	\begin{equation*}
	\begin{split}
	S&\lesssim1+e^{2\b C_2\log2}\sum_{k=N}^{j-1}
	e^{-\frac{1-r_{2k}}{2M(1-r_{2j+1})}+\left(\frac1{r_{2N}}+\b\right)\frac{1-r_{2j}}{M(1-r_{2j+1})}}\\
	&\le1+e^{2\b C_2\log2}\sum_{k=N}^{j-1}
	e^{-\frac{1-r_{2k}}{4M(1-r_{2j+1})}}
	\le1+e^{2\b C_2\log2}\sum_{k=N}^{j-1}
	e^{-\frac{2^{\psi(2k+1)+\cdots+\psi(2j+1)}}{4M}}\\
	&\le1+e^{2\b C_2\log2}\sum_{k=N}^{j-1}
	e^{-\frac{2^{(j-k)\psi(2N+1)}}{4M}}
	\le1+e^{2\b C_2\log2}\sum_{\ell=1}^\infty
	e^{-\frac{2^{\ell\psi(2N+1)}}{4M}}<\infty,
	\end{split}
	\end{equation*}
and thus \eqref{18'} is valid. Therefore $W\notin\M$, and hence $\B$ is not continuously embedded into $P_W(L^\infty)$ by Theorem~\ref{th:otromundo}.

We next prove $A^p_\om\subset A^p_W\subset A^p_{\om\vp}$ for any prefixed $\vp$ tending to zero. Since we are interested in the case in which $\vp$ tends to zero very slowly, we may assume without loss of generality that $\vp$ is a differentiable concave function such that $-\vp'(t)/\vp(t)\lesssim1/(1-t)$. Moreover, since the statement concerns inclusions and thus only the behavior of $\vp$ near one matters, we may further assume that $\vp(t)\le(1-t)^{(1-t)2\b}$ for all $0\le t<1$. If $r_{2n}\le r<r_{2n+1}$, then \eqref{23'}, \cite[Lemma~2.1(ii)]{PelSum14} and \eqref{17'} yield
    $$
    \widehat{\om}(r)\ge\widehat{W}(r)
    =\widehat{W}(r_{2n+1})
    \ge\frac12\widehat{\om}(r_{2n+1})
    \gtrsim\widehat{\om}(r_{2n})\left(\frac{1-r_{2n+1}}{1-r_{2n}}\right)^\b
    \gtrsim\widehat{\om}(r)2^{-\psi(2n)\b}.
    $$
The same estimate with $2\b$ in place of $\b$ is valid for $r_{2n+1}\le r<r_{2n+2}$, and thus
    $$
    \widehat{\om}(r)\ge\widehat{W}(r)
    \gtrsim\widehat{\om}(r)2^{-\psi(2n)2\b},\quad r_{2n}\le r<r_{2n+2}.
    $$
Since $\vp$ is decreasing, it therefore remains to find $\psi=\psi_{\om,\vp}$ such that
\begin{equation}\label{eq:vp}
    \vp\left(r_x\right)=\vp\left(1-\frac{1}{2^{x\psi(x)}}\right)\lesssim\frac{1}{2^{\psi(x)2\b}},\quad 1\le x<\infty.
    \end{equation}
Define $\psi(x)=-\frac1{2\b}\log_2\vp\left(1-\frac1x\right)$ for all $1\le x<\infty$. Then \eqref{eq:vp} holds whenever
    $$
    \vp\left(1-\vp\left(1-\frac1x\right)^\frac{x}{2\b}\right)\le\vp\left(1-\frac1x\right),\quad 1\le x<\infty,
    $$
which is equivalent to $\vp(t)\le(1-t)^{(1-t)2\b}$ for all $0\le t<1$ because $\vp$ is decreasing by the hypothesis. To be precise, we also have to require $\psi(x+1)-\psi(x)\lesssim1/x$ for all $1\le x<\infty$, but this easily follows from the concavity of $\vp$ and $-\vp'(t)/\vp(t)\lesssim1/(1-t)$. Therefore we have
    $$
    \widehat{\om}(r)\ge\widehat{W}(r)
    \gtrsim\widehat{\om}(r)\vp(r)\ge\widehat{\om\vp}(r),\quad 0\le r<1,
    $$
for each differentiable concave function $\vp$ tending to zero sufficiently slowly and satisfying $\vp(t)\le(1-t)^{2\b(1-t)}$. A reasoning similar to that in \eqref{eq:intbyparts} now gives $A^p_\om\subset A^p_W\subset A^p_{\om\vp}$.

It remains to show that $\vp\om\in\DDD$ whenever $\om\in\DDD$ and $-\vp'(t)/\vp(t)\lesssim1/(1-t)$. But this is straightforward. Namely, if $\om\in\Dd$ and $\vp$ is decreasing, then simple estimates yield $\vp\om\in\Dd$. Next consider the sequence $r_j=1-K^{-j}$, where $K=K(\om)>1$ is the one appearing in the definition of $\Dd$. Observe that the hypothesis on $\vp$ implies $\vp(r_j)\lesssim\vp(r_{j+1})$, and $\om\in\DDD$ easily gives $\int_{r_j}^{r_{j+1}}\om(t)\,dt\lesssim\int_{r_{j+1}}^{r_{j+2}}\om(t)\,dt$, both inequalities being valid for all $j$. By using these estimates it follows that $\widehat{\vp\om}(r_N)\lesssim\widehat{\vp\om}(r_{N+1})$ for all $N\in\N$, from which we deduce $\om\in\DD$.
\end{proof}

\section{Littlewood-Paley estimates}\label{Sec:L-P-estimates}

In this section, we will prove Theorems~\ref{th:L-P-D} and~\ref{theorem:L-P-D-hat}.

\medskip

\begin{Prf}{\em{Theorem~\ref{theorem:L-P-D-hat}.}}
Assume first $\om\in\DD$. For $0\le r<1$ and $n\in\N\cup\{0\}$, let $r_n=r_n(r)=1-\frac{1-r}{2^n}$ so that $r_0=r$, $r_{n+1}=\frac{1+r_n}{2}$ and $r_{n}-r_{n-1}=\frac{1-r}{2^n}$ for all $n\in\N$.
 Then,  $k$ applications of \cite[Lemma~3.1]{PavP}, with $\r=r_n$ for $n=1,\ldots,k$, and $q=p$,
an integration by parts, the fact that $M_p^p(r,f)\widehat{\om}(r)\to 0$ as $r\to 1^-$ for each $f\in A^p_\om$ and the definition of $\DD$ yield
		\begin{equation}\label{eq:L-P-D-hat2prueba}
		\begin{split}
		&\int_\D|f^{(k)}(z)|^p(1-|z|)^{kp}\om(z)\,dA(z)
		\lesssim\int_0^1M_p^p\left(r_1,f^{(k-1)}\right)(1-r)^{(k-1)p}\om(r)\,dr\\
		&\lesssim\cdots\lesssim\int_0^1M_p^p\left(r_k,f\right)\om(r)\,dr\\
    & = \int_{1-2^{-k}}^1M_p^p\left(r,f\right)2^{k}\om\left(1-2^k(1-r)\right)\,dr\\
		& = M_p^p\left(1-2^{-k} ,f\right)\widehat{\om}\left(0\right)+
    \int_{1-2^{-k}}^1\frac{d}{dr}M_p^p\left(r,f\right)\widehat{\om}\left(1-2^k(1-r)\right)\,dr\\
		&\le M_p^p\left(1-2^{-k} ,f\right)\widehat{\om}\left(0\right)+
    C^k\int_{1-2^{-k}}^1\frac{d}{dr}M_p^p\left(r,f\right)\widehat{\om}\left(r\right)\,dr\\
		&\le M_p^p\left(1-2^{-k} ,f\right)\widehat{\om}\left(0\right)+
    C^k\int_{1-2^{-k}}^1 M_p^p\left(r,f\right)\om(r)\,dr
		\lesssim\|f\|_{A^p_\om}^p,
    \end{split}
    \end{equation}
and \eqref{eq:desLP} follows.

Assume now that \eqref{eq:desLP} is valid for some $0<p<\infty$ and $k\in\N$. By choosing $f(z)=z^n$, we obtain
    $$
    \left(\prod_{j=0}^{k-1}(n-j)\right)^p\int_0^1r^{(n-k)p}\left(1-r\right)^{kp}\om(r)\,dr\lesssim\int_0^1r^{np}\om(r)\,dr,\quad n\ge k.
    $$
If $m\le kx<m+1$ for $m\ge k$, then
    \begin{equation*}
    \begin{split}
    x^{kp}\int_0^1r^{kpx}\left(1-r\right)^{kp}\om(r)\,dr
    &\le e^p\left(\prod_{j=0}^{k-1}(m-j)\right)^p\int_0^1r^{(kx+1-k)p}\left(1-r\right)^{kp}\om(r)\,dr\\
    &\le e^p\left(\prod_{j=0}^{k-1}(m+1-j)\right)^p\int_0^1r^{(m+1-k)p}\left(1-r\right)^{kp}\om(r)\,dr\\
    &\lesssim\int_0^1r^{(m+1)p}\om(r)\,dr
    \le\om_{pkx},\quad 1\le x<\infty.
    \end{split}
    \end{equation*}
By writing $q=kp\in(0,\infty)$ and extending the conclusion trivially for small values of $x$, we can write the above inequality in the form
    \begin{equation}\label{eq:moment-condition}
    \begin{split}
    x^{q}\int_0^1r^{qx}\left(1-r\right)^{q}\om(r)\,dr
    &\le C^q\om_{qx},\quad 0\le x<\infty,
    \end{split}
    \end{equation}
for some $C=C(\om,p,k)>0$. It follows that
    \begin{equation*}
    \begin{split}
    \int_0^{1-\frac{C}{x}}r^{qx}\left((x(1-r))^q -C^q\right)\om(r)\,dr
    &\le\int_{1-\frac{C}{x}}^1 r^{qx}\left(C^q-(x(1-r))^q \right)\om(r)\,dr\\
    &\le C^q\widehat{\om}\left(1-\frac{C}{x}\right),\quad x\ge C,
    \end{split}
    \end{equation*}
and hence
    \begin{equation*}
    \begin{split}
    \widehat{\om}\left(1-\frac{C}{x}\right)
    &\ge C^{-q}\int_0^{1-\frac{C\left(\frac32\right)^\frac1q}{x}}r^{qx}\left((x(1-r))^q -C^q\right)\om(r)\,dr\\
    &\ge\frac{1}{2}\int_0^{1-\frac{C\left(\frac32\right)^\frac1q}{x}}r^{qx}\om(r)\,dr,\quad x>\left(\frac{3}{2}\right)^{\frac1q}C.
    \end{split}
    \end{equation*}
This estimate together with Fubini's theorem yields
    \begin{equation*}
    \begin{split}
    2\widehat{\om}\left(1-\frac{C}{x}\right)
    &\ge\int_0^{1-\frac{C\left(\frac{3}{2}\right)^{\frac1q}}{x}}\om(t)\left(\int_0^tqxs^{qx-1}\,ds\right)\,dt\\
    &=\int_0^{1-\frac{C\left(\frac{3}{2}\right)^{\frac1q}}{x}}qxs^{qx-1}
    \left(\int_{s}^{1-\frac{C\left(\frac{3}{2}\right)^{\frac1q}}{x}}\om(t)\,dt\right)\,ds\\
    &=\int_0^{1-\frac{C\left(\frac{3}{2}\right)^{\frac1q}}{x}}qxs^{qx-1}
    \left(\widehat{\om}\left(s\right)-\widehat{\om}\left(1-\frac{C\left(\frac{3}{2}\right)^{\frac1q}}{x}\right) \right)\,ds\\
    &\ge\int_0^rqxs^{qx-1}\widehat{\om}\left(s\right)\,ds
    -\widehat{\om}\left(1-\frac{C\left(\frac{3}{2}\right)^{\frac1q}}{x}\right)\int_0^{1-\frac{C\left(\frac{3}{2}\right)^{\frac1q}}{x}}qxs^{qx-1}\,ds\\
    &\ge\widehat{\om}(r)r^{qx}-\widehat{\om}\left(1-\frac{C\left(\frac{3}{2}\right)^{\frac1q}}{x}\right) \left(1-\frac{C\left(\frac{3}{2}\right)^{\frac1q}}{x}\right)^{qx}
    \end{split}
    \end{equation*}
for all $0<r<1-\frac{C\left(\frac{3}{2}\right)^{\frac1q}}{x}<1$.
Therefore
    $$
    \widehat{\om}(r)r^{qx}
    \le\left(2+\left(1-\frac{C\left(\frac{3}{2}\right)^{\frac1q}}{x}\right)^{qx}\right)
    \widehat{\om}\left(1-\frac{C\left(\frac{3}{2}\right)^{\frac1q}}{x}\right)
    \le3\widehat{\om}\left(1-\frac{C\left(\frac{3}{2}\right)^{\frac1q}}{x}\right)
    $$
for all $0<r<1-\frac{C\left(\frac{3}{2}\right)^{\frac1q}}{x}<1$.
By choosing $x$ such that $1-\frac{C\left(\frac{3}{2}\right)^{\frac1q}}{x}=\frac{1+r}{2}$, we obtain
    $$
    \widehat{\om}(r)
    \le3\left(\inf\left\{r^{\frac{2C\left(\frac{3}{2}\right)^{\frac1q}}{1-r}}:\frac12\le r<1\right\}\right)^{-q}\widehat{\om}\left( \frac{1+r}{2}\right)\lesssim\widehat{\om}\left( \frac{1+r}{2}\right),\quad \frac12\le r<1,
    $$
and it follows that $\om\in\DD$.
\end{Prf}

\medskip

If $\om\in\DD$, then \eqref{eq:desLP} is satisfied by Theorem~\ref{theorem:L-P-D-hat}, and this in turn implies \eqref{eq:moment-condition} by the latter part of the proof of Theorem~\ref{theorem:L-P-D-hat}. By replacing $x$ by $x/\b$ and choosing $q=\b$ we deduce that \eqref{eq:moment-condition} is equivalent to $x^\b(\om_{[\b]})_x\le (C\b)^\b\om_x$ for all $0\le x<\infty$. Now that \eqref{eq:moment-condition} implies $\om\in\DD$  by the proof of Theorem~\ref{theorem:L-P-D-hat}, we deduce that a radial weight $\om$ belongs to $\DD$ if and only if for some (equivalently for each) $\b>0$, there exists a constant $C=C(\om,\b)>0$ such that
    \begin{equation}\label{Corollary:D-hat}
    x^\b(\om_{[\b]})_x\le C\om_x,\quad 0\le x<\infty.
    \end{equation}
		
Some more notation is needed for the proof of Theorem~\ref{th:L-P-D}.
For a radial weight $\om$, $K>1$ and $0\le r<1$, let $\r_n^r=\r_n^r(\om,K)$ be defined by
$\widehat{\om}(\r_n^r)=\widehat{\om}(r)K^{-n}$. Write $\r_n^0=\r_n$ for short. By defining $\r_n^r=\min\{\r:\widehat{\om}(\r)=\widehat{\om}(r)K^{-n}\}$, the sequence becomes unique.

\medskip
\par \begin{Prf}{\em{Theorem~\ref{th:L-P-D}.}}
Let first $\om\in\DDD$. Then the left hand side dominates the right hand side in \eqref{Eq:L-P-D} by Theorem~\ref{theorem:L-P-D-hat}. For the converse inequality,
 a calculation shows that $(1-s)^\g\om(s)\in\Dd$, whenever $\om\in\Dd$ and $\gamma\ge 0$, so it is enough to prove that 
$$\|f\|_{A^p_\om}^p\lesssim |f(0)|^p+ \int_{\D}|f'(z)|^p(1-|z|)^p\om(z)\,dA(z),\quad f\in\H(\D).$$
Next, by choosing $t=\r_{n+1}$ and $r=\r_n$ in \eqref{6}, we deduce that for each $K>1$, there exists $C=C(\om,K)>0$ such that $1-\r_n\le C(1-\r_{n+1})$ for all $n\in\N\cup\{0\}$.
So, applying \cite[Lemma~3.4]{PavP} (with $\gamma=p$ and $A_n=M_p(\r_{n+1},f)$ for $p\ge 1$, and $\gamma=1$ and $A_n=M^p _p(\r_{n+1},f)$ for $p<1$)
 and \eqref{dinkeli} yield
    \begin{equation*}
    \begin{split}
    \|f\|_{A^p_\om}^p
    &\le\int_0^1 M_p^p(r,f)\om(r)\,dr
    \le\left(1-\frac1K\right)\sum_{n=-1}^\infty M_p^p(\r_{n+1},f)\frac{\widehat{\om}(0)}{K^{n}}\\
    &\lesssim\sum_{n=-1}^\infty M_p^p(\r_{n+2},f')(\r_{n+2}-\r_{n+1})^p\frac{\widehat{\om}(0)}{K^{n}}+|f(0)|^p\\
   &\lesssim\sum_{n=-1}^\infty M_p^p(\r_{n+2},f')(1-\r_{n+3})^p\frac{\widehat{\om}(0)}{K^{n}}+|f(0)|^p\\
   &\lesssim\sum_{n=-1}^\infty M_p^p(\r_{n+2},f')(1-\r_{n+3})^p\int_{\r_{n+2}}^{\r_{n+3} }\om(s)\,ds+|f(0)|^p\\
   &\lesssim \int_{\r_1}^1 M_p^p(s,f)(1-s)^p\om(s)\,ds +|f(0)|^p, \quad f\in\H(\D).
    \end{split}
    \end{equation*}
Therefore we have shown that \eqref{Eq:L-P-D} is satisfied if $\om\in\DDD$.

Conversely, assume that \eqref{Eq:L-P-D} is satisfied for some $k\in\N$ and $0<p<\infty$. Then $\om\in\DD$ by Theorem~\ref{theorem:L-P-D-hat}. By applying \eqref{Eq:L-P-D} to the monomials $z^n$, we deduce
    \begin{equation*}
    \begin{split}
    \om_{np+1}
    &\lesssim n^{kp}\int_\D|z|^{(n-k)p}(1-|z|)^{kp}\om(z)\,dA(z)
    \asymp n^{kp}(\om_{[kp]})_{np+1},\quad n\ge k,
    \end{split}
    \end{equation*}
where the last step follows by the monotonicity of $L^p$-means of analytic functions. It follows that $\om\in\M$, because \eqref{eq:characterization-M} characterizes $\M$ for some (equivalently for each) $\b>0$. Now that $\DD\cap\M=\DDD$ by Theorem~\ref{th:projectionboundedonto}, we have $\om\in\DDD$.
\end{Prf}

\section{$P_\om: L^p_\om \to D^p_{\om,k}$ plus duality $(A^{p'}_\om)^\star\simeq D^p_{\om,k}$ }\label{sec:projDp}

In this section we will prove Theorems~\ref{th:PomegaDp} and~\ref{th:PwDpwonto}.
\medskip

\begin{Prf}{\em{Theorem~\ref{th:PomegaDp}.}}
It is clear that (i) implies (iii), and (iv) follows from (iii) by Theorem~\ref{theorem:L-P-D-hat}. The fact that $P_\om: L^p_\om \to D^p_{\om,k}$ is bounded whenever $\om\in\DD$ is an immediate consequence of Theorem~\ref{th:proyderivLP} below. To complete the proof we will show that (i) and (ii) are equivalent. But before that, we observe that one can also deduce (i) from (iv) by Theorems~\ref{theorem:L-P-D-hat} and \ref{th:dualityAp}.

First observe that the dual space $(A^{p'}_\om)^\star$ can be identified with $P_\om(L^p_\om)$ under the $A^2_\om$-pairing, and in particular, for each $L\in(A^{p'}_\om)^\star$ there exists $h\in L^p_\om$ such that
    \begin{equation}\label{eq:k1p}
    L(f)=L_{P_\om(h)}(f)=\langle f,P_\om(h)\rangle_{A^2_\om},\quad f\in A^{p'}_\om,\quad\|L\|=\|h\|_{L^p_\om}.
    \end{equation}

Now assume (i). Then, by \eqref{eq:k1p}, for each $L\in (A^{p'}_\om)^\star$ there exists $h\in L^p_\om$ such that $L=L_{P_\om(h)}$, where
	$$
	\| P_\om(h)\|_{D^p_{\om,k}}\le C\| h\|_{L^p_\om}=C\|L\|.
	$$
Thus (ii) holds. Conversely, assume (ii) and let $h\in L^p_\om$. Then, by \eqref{eq:k1p}, $L_{P_\om(h)}\in  (A^{p'}_\om)^\star$
and $\|L_{P_\om(h)}\|\le \|h\|_{L^p_\om}$. Moreover,  by (ii) $L_{P_\om(h)}=L_g$ for some $g\in D^p_{\om,k}$ and
$\| g\|_{D^p_{\om,k}}\le C \|L_{P_\om(h)}\|$. Therefore $g=P_\om(h)$ and $\| P_\om(h)\|_{D^p_{\om,k}}\le C \| h\|_{L^p_\om}$. Since $h$ was arbitrary, (i) follows.
\end{Prf}

\medskip

For $k\in\N\cup\{0\}$ and a radial weight $\om$, consider the operators
    $$
    T_{\om,k}(f)(z)=(1-|z|)^k\int_{\D} f(\z)(B^{\om}_\z)^{(k)}(z)\om(\z)\,dA(\z),\quad z\in\D,\quad f\in L^1_\om,
    $$
and
    $$
    T^+_{\om,k}(f)(z)=(1-|z|)^k\int_{\D} f(\z)\left|(B^{\om}_\z)^{(k)}(z)\right|\om(\z)\,dA(\z),\quad z\in\D,\quad f\in L^1_\om.
    $$
Then obviously $T_{\om,0}=P_\om$ and $T^+_{\om,0}=P^+_\om$.

\begin{theorem}\label{th:proyderivLP}
Let $\om$ be a radial weight, $1<p<\infty$ and $k\in\N$. Then the following statements are equivalent:
    \begin{itemize}
    \item[(i)] $T^+_{\om,k}:L^p_\om\to L^p_\om$ is bounded;
    \item[(ii)] $T_{\om,k}:L^p_\om\to L^p_\om$ is bounded;
    \item[(iii)] $\om\in\DD$.
    \end{itemize}
\end{theorem}

\begin{proof}
We begin with showing that $T^+_{\om,k}:L^p_\om\to L^p_\om$ is bounded if $\om\in\DD$. Let $1<p<\infty$, and choose $h=\widehat{\om}^{-\frac{1}{pp'}}$. By Schur's test, it suffices to show that
    \begin{equation}\label{t1}
    \int_{\D}(1-|z|)^k\left|(B^{\om}_\z)^{(k)}(z)\right|h^{p'}(\z)\omega(\z)\,dA(\z)\lesssim h^{p'}(z),\quad z\in\D,
    \end{equation}
and
    \begin{equation}\label{t2}
    \int_{\D}(1-|z|)^k\left|(B^{\om}_\z)^{(k)}(z)\right|h^{p}(z)\omega(z)\,dA(z)\lesssim h^{p}(\z),\quad \z\in\D.
     \end{equation}
By applying \cite[Theorem~1]{PelRatproj} and \cite[Lemma~2.1(ii)]{PelSum14}, we deduce
    \begin{equation*}
    \begin{split}
    &(1-|z|)^k\int_{\D}\left|(B^{\om}_\z)^{(k)}(z)\right| h^{p'}(\z)\omega(\z)\,dA(\z)\\
    &\asymp(1-|z|)^k\int_0^1\left(\int_0^{s|z|}\frac{dt}{\widehat{\om}(t)(1-t)^{k+1}}+1\right)
    \frac{\om(s)}{\widehat{\om}(s)^{\frac{1}{p}}}ds\\
    &\asymp (1-|z|)^k\left(\int_0^1
    \frac{\om(s)}{(1-s|z|)^k\widehat{\om}(sz)\widehat{\om}(s)^{\frac{1}{p}}}ds \right)\\
    &\le\int_0^{|z|}
    \frac{\om(s)}{\widehat{\om}(s)^{1+\frac{1}{p}}}ds
    +\frac1{\widehat{\om}(z)}\int_{|z|}^1
    \frac{\om(s)}{\widehat{\om}(s)^{\frac{1}{p}}}ds
		\lesssim\frac1{\widehat{\om}(z)^\frac1p}=h^{p'}(z),\quad z\in\D,
    \end{split}
    \end{equation*}
and hence \eqref{t1} is satisfied. A similar reasoning gives \eqref{t2}. Thus $T^+_{\om,k}:L^p_\om\to L^p_\om$ is bounded if $\om\in\DD$ and $1<p<\infty$.

Since (i) trivially implies (ii), it remains to show that $\om\in\DD$ if $T_{\om,k}:L^p_\om\to L^p_\om$ is bounded. The representation formula gives
$f=P_\om(f)$ for each $f\in A^1_\om$. Hence $(1-|z|)^kf^{(k)}(z)=T_{\om,k}(f)(z)$ and
    \begin{equation*}
    \begin{split}
    \|f^{(k)}(\cdot)(1-|\cdot|)^k\|_{L^p_\om}=\| T_{\om,k}(f)\|_{L^p_\om}\lesssim\|f\|_{A^p_\om},\quad f\in A^p_\om,
    \end{split}
    \end{equation*}
if $T_{\om,k}:L^p_\om\to L^p_\om$ is bounded. Therefore $\om\in\DD$ by Theorem~\ref{theorem:L-P-D-hat}.
\end{proof}

The proof of Theorem~\ref{th:proyderivLP} shows that if $\om\in\DD$, then $\|f^{(k)}(\cdot)(1-|\cdot|)^k\|_{L^p_\om}\lesssim\|f\|_{A^p_\om}$ for all $f\in A^p_\om$ provided $1<p<\infty$. This yields an alternate proof of \eqref{eq:desLP} under the hypotheses $1<p<\infty$ and $\om\in\DD$. Moreover, the equivalence between (i) and (iv) in Theorem~\ref{ELTEOREMA} can be considered as the limit case $p=\infty$ of Theorem~\ref{th:proyderivLP}.

Now we characterize the radial weights $\om$ such that $P_\om: L^p_\om\to D^p_{\om,k}$ is bounded and onto.
\medskip\par
\begin{Prf}{\em{Theorem~\ref{th:PwDpwonto}.}}
Theorem~\ref{th:L-P-D} implies that (iii) and (iv) are equivalent. Moreover, if (iv) holds, then
$P_\om: L^p_\om\to D^p_{\om,k}$ is bounded by Theorem~\ref{th:PomegaDp}, and it is onto since $A^p_\om=D^p_{\om,k}$ by Theorem~\ref{th:L-P-D}. Thus (i) is satisfied.

Assume (i). Then $(A^{p'}_\om)^\star$ is continuously embedded into $D^p_{\om,k}$ via the $A^2_\om$-pairing by Theorem~\ref{th:PomegaDp}. Moreover, since $P_\om: L^p_\om \to D^p_{\om,k}$ is onto, there exists $C>0$ such that
for each $g\in D^p_{\om,k}$ there is $h\in L^p_\om$ such that $g=P_\om(h)$ and $\| h\|_{L^p_\om}\le C\| g\|_{D^p_{\om,k}}$.
Then \eqref{eq:fubinisubs} yields
	$$
	|L_g(f)|\le \|h\|_{L^p_\om}\| f\|_{A^{p'}_\om}
	\le C\|g\|_{D^p_{\om,k}}\|f\|_{A^{p'}_\om} ,\quad f\in A^{p'}_\om,\quad g\in D^p_{\om,k},
	$$
that is, $D^p_{\om,k}$ is continuously embedded into $(A^{p'}_\om)^\star$, and thus (ii) holds.

We complete the proof by showing that (ii) implies (iv). Since $A^p_\om$ is continuously embedded into $(A^{p'}_\om)^\star$, it follows
  that $I_d:A^p_\om \to D^p_{\om,k}$ is bounded whenever (ii) is satisfied. Hence $\om\in\DD$ by Theorem~\ref{theorem:L-P-D-hat}. It remains to show that $\om\in\Dd$. Since each $g\in D^p_{\om,k}$ induces a bounded linear functional $L_g$, defined by $L_g(f)=\langle f,g\rangle_{A^2_\om}$, on $A^{p'}_\om$, there exists a constant $C>0$ such that
	\begin{equation}\label{eq:pairingdp}
	\left| \langle f,g\rangle_{A^2_\om}\right|\le C \|f\|_{A^{p'}_\omega}\|g\|_{D^p_{\om,k}}, \quad f,\,g\in\H(\D).
	\end{equation}
For each $t\in \left(\frac12,1\right)$, consider the analytic functions $g_t(z)=\sum_{k=0}^\infty\frac{(tz)^{2^k}}{\om_{2^{k+1}}}$ and $f_{t,\beta}(z)=\sum_{k=0}^\infty\frac{(tz)^{2^k}}{(\om_{2^{k+1}})^{\beta}}$, where $\beta>\frac{2}{p'}$. Since $f_t$ and $g_t$ are lacunary series with the same gaps, for each $0<q<\infty$ there exist constants $C_1=C_1(q)>0$ and $C_2=C_2(q)>0$ such that
    \begin{equation}\label{eq:lacunary}
    C_1M_2(r,h)\le M_q(r,h)\le C_2M_2(r,h),\quad h\in\{f_{t,\b},g_t\},\quad r,t\in(0,1).
    \end{equation}
The choices $f=f_{t,\b}$ and $g=g_t$ in \eqref{eq:pairingdp} yield
    \begin{equation}\label{eq:onto1}
    \sum_{k=0}^\infty\frac{t^{2^{k+1}}}{(\om_{2^{k+1}})^\beta}\lesssim \|f_{t,\beta}\|_{A^{p'}_\omega}\|g_t\|_{D^p_{\om,k}}.
    \end{equation}
For $\om\in\DD$, $N\in\N\cup\{0\}$ and $\alpha,\gamma>0$ we have the estimate
		\begin{equation}\label{eq:onto2}
		\begin{split}
    \int_{0}^r\frac{dx}{\widehat{\om}(t)^\alpha(1-t)^\gamma}
		&\asymp\sum_{n=N}^\infty\frac{r^{2^{n+1}}}{2^{n(1-\gamma)}\om_{2^{n+1}}^\alpha},\quad r\to1^-,
    \end{split}
		\end{equation}
the proof of which is postponed for a moment. This together with \eqref{eq:lacunary} and \eqref{eq:onto2} imply
	\begin{equation}
	\begin{split}\label{eq:fpprima}
	\|f_{t,\beta}\|^{p'}_{A^{p'}_\omega}
	&\asymp\int_{0}^1\left( \sum_{k=0}^\infty\frac{(st)^{2^{k+1}}}{\om^{2\beta}_{2^{k+1}}}\right)^{\frac{p'}{2}}\,\om(s)\,ds
	\asymp\int_{0}^1\left(\int_{0}^{st}\frac{dx}{\widehat{\om}(x)^{2\beta}(1-x)}\right)^{\frac{p'}{2}}\,\om(s)\,ds+1\\
	&\le\left(\int_{0}^{t}\frac{dx}{\widehat{\om}(x)^{\beta}(1-x)}\right)^{\frac{p'}{2}}\int_{0}^1
	\frac{\om(s)}{\widehat{\om}(st)^{\frac{p'	\beta}{2}}}\,ds+1\\
	&\lesssim\left(\int_{0}^{t}\frac{dx}{\widehat{\om}(x)^\beta(1-x)}\right)^{\frac{p'}{2}}
	\left(\int_{0}^t\frac{\om(s)}{\widehat{\om}(s)^{\frac{p'\beta}{2}}}\,ds
	+\frac{1}{\widehat{\om}(t)^{\frac{p'\beta}{2}-1}} \right)+1\\
	&\asymp\left(\int_{0}^{t}\frac{dx}{(1-x)\widehat{\om}(x)^\beta}\right)^{\frac{p'}{2}}\frac{1}{\widehat{\om}(t)^{\frac{p'\beta}{2}-1}}
	,\quad t\to1^-.
  \end{split}
	\end{equation}
In a similar manner \eqref{eq:lacunary} and \eqref{eq:onto2} yield
	\begin{equation}
	\begin{split}\label{eq:gpdp}
	\|g_t\|^p_{D^p_{\om,k}}
	&\asymp \int_{0}^1\left(\sum_{j>\frac{\log_2k}{\log2}}
	\frac{2^{2kj}s^{2^{j+1}-2k}t^{2^{j+1}}}{\om^2_{2^{j+1}}}\right)^{\frac{p}{2}}\,\om(s)(1-s)^{kp}\,ds\\
	&\lesssim\int_{0}^1\left(\int_{0}^{st}\frac{dx}{\widehat{\om}(x)^2(1-x)^{2k+1}}\right)^{\frac{p}{2}}\,\om(s)(1-s)^{kp}\,ds+1\\
	&\asymp\int_{0}^1 \frac{\om(s)}{\widehat{\om}(st)^{p}}\,ds+1
	\asymp\frac{1}{\widehat{\om}(t)^{p-1}},\quad t\to1^-.
  \end{split}
	\end{equation}
By combining \eqref{eq:onto1}--\eqref{eq:gpdp}, we deduce
	$$
	\int_{0}^{t}\frac{dx}{\widehat{\om}(x)^\beta(1-x)}
	\lesssim\left(\int_{0}^{t}\frac{dx}{\widehat{\om}(x)^\beta(1-x)}\right)^{\frac{1}{2}}\frac{1}{\widehat{\om}(t)^{\frac{\beta}{2}}}
	,\quad t\to1^-,
	$$
and it follows that
	$$
	\int_{0}^{t}\frac{dx}{\widehat{\om}(x)^\beta(1-x)}\lesssim\frac{1}{\widehat{\om}(t)^{\beta}},\quad 0\le t<1.
	$$
If now $0\le r\le t<1$, then this estimate yields
    \begin{equation*}
    \begin{split}
    \frac{1}{\widehat{\om}(r)^\b}\log\frac{1-r}{1-t}
    =\frac{1}{\widehat{\om}(r)^\b}\int_r^t\frac{ds}{1-s}
    \le\int_r^t\frac{ds}{\widehat{\om}(s)^\b(1-s)}
    \le \frac{C}{\widehat{\om}(t)^\b},
    \end{split}
    \end{equation*}
where $C=C(\b,\om)>0$. By setting $t=1-\frac{1-r}{K}$, where $K>1$, we deduce
    $$
    \widehat{\om}(r)\ge\left(\frac{\log K}{C}\right)^\frac1{\b}\widehat{\om}\left(1-\frac{1-r}{K}\right),\quad 0\le r<1,
    $$
from which it follows that $\om\in\Dd$ by choosing $K$ sufficiently large.

It remains to prove \eqref{eq:onto2}. Let $r>1-2^{-N-1}$ and choose $N^*\in\N$ such that $1-2^{-N^*}\le r<1-2^{-N^*-1}$. Then \cite[Lemma~2.1(ii)(vi)]{PelSum14} yields
	\begin{equation*}
	\begin{split}
	\sum_{n=N}^{N^*}\frac{r^{2^{n+1}}}{2^{n(1-\gamma)}\om_{2^{n+1}}^\alpha}
	&\asymp\sum_{n=N}^{N^*}\frac{1}{2^{n(1-\gamma)}\om_{2^{n+1}}^\alpha}
	\asymp\sum_{n=N}^{N^*}\frac{1}{2^{n(1-\gamma)}\widehat{\om}\left(1-\frac1{2^{n+1}}\right)^\alpha}\\
	&\asymp\int_N^{N^*}\frac{ds}{2^{s(1-\gamma)}\widehat{\om}\left(1-\frac1{2^{s}}\right)^\alpha}
	\asymp\int_{1-\frac1{2^N}}^{1-\frac1{2^{N^*+1}}}\frac{dt}{(1-t)^\gamma\widehat{\om}(t)^\alpha}\\
	&\ge\int_{1-\frac1{2^N}}^r\frac{dt}{(1-t)^\gamma\widehat{\om}(t)^\alpha}
	\asymp\int_0^r\frac{dt}{(1-t)^\gamma\widehat{\om}(t)^\alpha},\quad r\to1^-.
	\end{split}
	\end{equation*}
The same upper bound can be established in a manner similar to that above. By \cite[Lemma~2.1(ii)]{PelSum14} we may fix 
$\b>0$ such that $\widehat{\om}(r)^\alpha(1-r)^{-\beta}$ is essentially increasing. This and \cite[Lemma~2.1(ii)(vi)]{PelSum14} yield
	\begin{equation*}
	\begin{split}
	\sum_{n=N^*}^\infty\frac{r^{2^{n+1}}}{2^{n(1-\gamma)}\om_{2^{n+1}}^\alpha}
	&\asymp\sum_{n=N^*}^\infty\frac{r^{2^{n+1}}}{2^{n(1-\gamma)}\widehat{\om}\left(1-\frac1{2^{n}}\right)^\alpha}
	\lesssim\frac1{2^{\beta N^*}\widehat{\om}\left(1-\frac1{2^{N^*}}\right)^\alpha}\sum_{n=N^*}^\infty\frac{r^{2^{n+1}}}{2^{n(1-\gamma-\b)}}\\
	&\asymp\frac1{2^{\beta N^*(1-\gamma)}\widehat{\om}\left(1-\frac1{2^{N^*}}\right)^\alpha}
	\asymp\frac1{(1-r)^{\gamma-1}\widehat{\om}(r)^\alpha}\\
	&\asymp\int_{\frac{4r-1}{3}}^r\frac{dt}{(1-t)^\gamma\widehat{\om}(t)^\alpha},\quad r\to1^-.
	\end{split}
	\end{equation*}
By combining the estimates established we deduce \eqref{eq:onto2}.
\end{Prf}

\section{Bergman projection $P_\om:L^p_\om\to L^p_\om$}\label{Sec:P_w}

We begin with some necessary definitions. Recall that for a radial weight $\om$ and $K>1$, $\{\r_n\}_{n=0}^\infty=\{\r_n(\om,K)\}_{n=0}^\infty$ is a sequence defined by $\widehat{\om}(\r_n)=\widehat{\om}(0)K^{-n}$ for all $n\in\N\cup\{0\}$. Denote
\begin{equation}\label{eq:Mn}
M_n=M_n(\om,K)=E\left(\frac{1}{1-\r_{n}}\right),
\end{equation}
 where $E(x)\in\N\cup\{0\}$ satisfies $E(x)\le x<E(x)+1$ for all $x>0$. Write
    $$
    I(0)=I_{\om,K}(0)=\left\{k\in\N\cup\{0\}:k<M_1\right\}
    $$
and
   \begin{equation*}
    I(n)=I_{\om,K}(n)=\left\{k\in\N:M_n\le
    k<M_{n+1}\right\},\quad n\in\N.
  \end{equation*}

For a Banach space $X\subset\H(\D)$, $s\in\mathbb{R}$, $0<q\le\infty$ and a sequence of polynomials $\{P_n\}_{n=0}^\infty$, let
    \begin{equation}\label{eq:lqXPn}
    \ell^q_s(X,\{P_n\})
    =\left\{f\in\H(\D):\|f\|^q_{\ell^q_s(X,\{P_n\})}=\sum_{n=0}^\infty\left(2^{-ns}\|P_n*f\|_X\right)^q<\infty\right\}, \quad 0<q<\infty,
    \end{equation}
and
    \begin{equation}\label{eq:linftyXPn}
    \ell^\infty_s(X,\{P_n\})
    =\left\{f\in\H(\D):\|f\|_{\ell^\infty_s(X,\{P_n\})}=\sup_n\left(2^{-ns}\|P_n*f\|_X\right)<\infty\right\}.
    \end{equation}
For quasi-normed spaces $X,Y\subset\H(\D)$, let $[X,Y]=\{g\in\H(\D):f*g\in Y\,\text{for all} \,f\in X\}$ denote the space of coefficient multipliers from $X$ to $Y$ equipped with the norm
    $$
    \|g\|_{[X,Y]}=\sup\left\{\|f*g\|_Y:f\in X,\,\|f\|_X\le1\right\},
    $$
and finally, let $\mathcal{A}$ denote the disk algebra.

\medskip

\begin{Prf}{\em{Theorem~\ref{th:dualityAp}.}} It is well known that if $P_\om: L^{p'}_\om \to L^{p'}_\om$ is bounded, then $(A^p_\om)^\star\simeq A^{p'}_\om$ via the $A^2_\om$-pairing with equivalence of norms. The converse implication is true also. Namely, assume that $(A^p_\om)^\star\simeq A^{p'}_\om$ via the $A^2_\om$-pairing with equivalence of norms. Let $h\in L^{p'}_\om$, and consider the bounded linear functional $T_h(f)=\langle f,h\rangle_{L^2_\om}$ on $A^{p}_{\om}$ that satisfies $\|T_h\|\le \|h\|_{L^{p'}_\om}$.
Then $T_h(f)=\langle f,P_\om(h)\rangle_{A^2_\om}$ for all~$f\in A^p_\om$ by \eqref{eq:fubinisubs}.
Further, by the hypothesis, there exists $g\in A^{{p'}}_{\om}$ such that $T_h(f)=\langle f,g\rangle_{A^2_\om}$ for all $f\in A^{p}_{\om}$, and $\|T_h\|\asymp\|g\|_{A^{p'}_\om}$. It follows that $P_\om(h)=g$, and hence $\|P_\om(h)\|_{A^{{p'}}_{\om}}\asymp\|T_h\|\le\|h\|_{L^{p'}_{\om}}$. Since $h\in L^{p'}_\om$ was arbitrary, this shows that $P_\om: L^{p'}_\om\to L^{p'}_{\om}$ is bounded.

Consequently, it suffices to show that $\left(A^p_\om \right)^\star$ can be identified with $A^{p'}_\om$ via the $A^2_\om$-pairing with equivalence of norms.

For an analytic function $f$ in $\D$, with Maclaurin series $f(z)=\sum_{n=0}^\infty\widehat{f}(n)z^n$, define the
polynomials $\Delta^{\om}_nf$ by
    $$
    \Delta_n^{\om}f(z)=\sum_{k\in I_{\om,2}(n)}\widehat{f}(k)z^k,\quad n\in\N\cup\{0\},
    $$
so that $f=\sum_{n=0}^\infty\Delta_n^\om f$. Then the case $\alpha=1$ and $q=p$ of \cite[Theorem~3.4]{PRAntequera} (see also \cite[Theorem~4.1]{PavMixnormI} and \cite[Theorem~4]{PelRathg} for similar results) shows that for $1<p<\infty$ and $\om\in\DD$, we have
    \begin{equation}\label{18}
    \|f\|_{A^p_\om}^p
    \asymp\sum_{n=0}^\infty 2^{-n}\|\Delta^{\om}_n f\|_{H^p}^p,\quad f\in\H(\D).
    \end{equation}
Therefore $A^p_\om=\ell^p_{\frac{1}{p}}(H^p,\{\Delta^{\om}_n\})$ with equivalent norms.

To prove the theorem, first observe that $\left(A^p_\om\right)^\star$ can be identified with $[A^p_\om,H^\infty]=[A^p_\om,\mathcal{A}]$ via the $H^2$-pairing with equality of norms by \cite[Proposition~1.3]{PavMixnormI}. That is, for each $L\in\left(A^p_\om\right)^\star$ there exists a unique $g\in[A^p_\om, H^\infty]$ such that $L(f)=\langle f, g\rangle_{H^2}$ and $\|L\|=\|g\|_{[A^p_\om,H^\infty]}$, and conversely, each functional $L(f)=\langle f,g\rangle_{H^2}$ induced by $g\in[A^p_\om,H^\infty]$ belongs to $\left(A^p_\om\right)^\star$ and satisfies $\|L\|=\| g\|_{[A^p_\om,H^\infty]}$. Next, use \cite[Theorem~5.4]{PavMixnormI} and the fact that $(H^p)^\star\simeq[H^p,H^\infty]=H^{p'}$ via the $H^2$-pairing to deduce
    \begin{equation}\label{eq:dualap1}
    \left[\ell^p_{\frac{1}{p}}\left(H^p,\{\Delta^{\om}_n\}\right),H^\infty\right]
    \simeq\ell^{p'}_{-\frac{1}{p}}\left([H^p,H^\infty],\{\Delta^{\om}_n\}\right)
    =\ell^{p'}_{-\frac{1}{p}}\left(H^{p'},\{\Delta^{\om}_n\}\right)
    \end{equation}
with equivalent norms. Since $A^p_\om\simeq\ell^p_{\frac{1}{p}}(H^p,\{\Delta^{\om}_n\})$ with equivalent norms, it follows that
    \begin{equation}\label{eq:ap1j}
    \left(A^p_\om\right)^\star
    \simeq[A^p_\om,H^\infty]\simeq\ell^{p'}_{-\frac{1}{p}}(H^{p'},\Delta^{\om}_n)
    \end{equation}
via the $H^2$-pairing with equivalence of norms.

Finally, define the operator $I^\om:\H(\D)\to\H(\D)$ by
$I^\om(g)(z)=\sum_{k=0}^\infty \widehat{g}(k)\om_{2k+1}z^{k}$ for all $z\in\D$.
The sequence $\{\om_{2k+1}\}_{k=0}^\infty$ is non-increasing,
and hence there exists a constant $C>0$ such that
    $$
    C^{-1}\om_{M_{n+1}}\|\Delta^{\om}_n g\|_{H^{p'}}
    \le\|\Delta^{\om}_n I^\om(g)\|_{H^{p'}}
    \le C\om_{M_n}\|\Delta^{\om}_n g\|_{H^{p'}},\quad g\in\H(\D),
    $$
by \cite[Lemma~E]{PelRathg}. This combined with $\om_x\asymp\widehat{\om}\left(1-\frac1x\right)$, valid for all $1\le x<\infty$ by \cite[Lemma~2.1(vi)]{PelSum14}, gives $\|\Delta^{\om}_nI^\om(g)\|_{H^{p'}}\asymp 2^{-n}\|\Delta^{\om}_n g\|_{H^{p'}}$ for all $g\in\H(\D)$.
Therefore \eqref{18} yields
    \begin{equation*}
    \begin{split}
    \| I^\om(g)\|^{p'}_{\ell^{p'}_{-\frac{1}{p}}(H^{p'} ,\Delta^{\om}_n )}
    &=\sum_{n=0}^\infty  2^{\frac{np'}{p}}\|\Delta^{\om}_n I^\om(g)\|^{p'}_{H^{p'}}
    \asymp \sum_{n=0}^\infty  2^{\frac{np'}{p}-np'}\|\Delta^{\om}_n g\|^{p'}_{H^{p'}}\\
    &=\sum_{n=0}^\infty  2^{-n}\|\Delta^{\om}_n g\|^{p'}_{H^{p'}}
    \asymp\| g\|^{p'}_{A^{p'}_\om},\quad g\in\H(\D),
    \end{split}
    \end{equation*}
and hence $I^\om:A^{p'}_\om\to\ell^{p'}_{-\frac{1}{p}}(H^{p'},\{\Delta^{\om}_n\})$ is isomorphic and onto. This converts $\left(A^p_\om\right)^\star\simeq\ell^{p'}_{-\frac{1}{p}}(H^{p'},\Delta^{\om}_n)$ via the $H^2$-pairing, valid by \eqref{eq:ap1j}, to $\left(A^p_\om\right)^\star\simeq A^{p'}_\om$ via the $A^2_\om$-pairing, and thus finishes the proof.
\end{Prf}

\medskip

Dostani\'c~\cite{Dostanic} showed that for a radial weight $\om$ and $1<p<\infty$, the reverse H\"older's inequality
    \begin{equation}\label{Eq:Dostanic-condition}
    \om_{np+1}^\frac1p\om_{np'+1}^\frac1{p'}\lesssim \om_{2n+1},\quad n\in\N,
    \end{equation}
is a necessary condition for $P_\om:L^p_\om\to L^p_\om$ to be bounded. If $\om\in\DD$ and $k>0$ are fixed, then $\om_x\asymp\om_{kx}$ for all $x\ge1$ by \cite[Lemma~2.1]{PelSum14}, and hence \eqref{Eq:Dostanic-condition} is satisfied. Of course this also follows by Theorem~\ref{th:dualityAp}. We offer two simple proofs of the necessary condition \eqref{Eq:Dostanic-condition}. In fact, we prove the following more general result.

\begin{proposition}\label{proposition:dostanic}
Let $\om$ a radial weight and $1<p<\infty$. If $P_\om:L^p_\om\to L^p_\om$ is bounded, then the following statements hold and are equivalent:
\begin{itemize}
\item[\rm(i)] $\displaystyle\sup_{n\in\N}\frac{\om_{np+1}^\frac1p\om_{np'+1}^\frac1{p'}}{\om_{2n+1}}<\infty$;
\item[\rm(ii)] $\displaystyle \sup_{m-n\in\N\cup\{0\}}\left(\frac{\om_{2m+1}}{\om_{2(m-n)+1}}\right)^p\frac{\om_{p(m-n)+1}}{\om_{p(m+n)+1}}<\infty$;
\item[\rm(iii)] $\displaystyle \left(\frac1{\om_{2n+1}}\int_0^1f(r)r^{2n+1}\om(r)\,dr\right)^p\lesssim\frac1{\om_{np+1}}\int_0^1f(r)^pr^{pn+1}\om(r)\,dr,\,\,\, f\in L^p_\om([0,1)),\,\,\, f\ge 0$.
\end{itemize}
\end{proposition}

\begin{proof}
Let $1<p<\infty$ and $\om$ a radial weight such that $P_\om:L^p_\om\to L^p_\om$ is bounded. Further,
for $m,n\in\mathbb{R}$ with $m-n\in\N\cup\{0\}$, let $f_{m,n}(z)=z^{m-n}|z|^{2n}$ for all $z\in\D$. Then $P_\om(f_{m,n})(z)=\frac{\om_{2m+1}}{\om_{2(m-n)+1}}z^{m-n}$, and hence
    \begin{equation*}
    \begin{split}
    2\|P_\om\|^p_{L^p_\om\to L^p_\om}\om_{p(m+n)+1}
    &=\|P_\om\|^p_{L^p_\om\to L^p_\om}\|f_{m,n}\|_{L^p_\om}^p
    \ge\|P_\om(f_{m,n})\|_{L^p_\om}^p\\
    &=2\left(\frac{\om_{2m+1}}{\om_{2(m-n)+1}}\right)^p\om_{p(m-n)+1}.
    \end{split}
    \end{equation*}
Therefore (ii) is satisfied and the supremum is bounded by $\|P_\om\|^p_{L^p_\om\to L^p_\om}$. The choices $n=m-k$ and $m=kp'/2$ with $k\in\N$ in (ii) give
    $$
    \sup_{k\in\N}\left(\frac{\om_{kp'+1}}{\om_{2k+1}}\right)^p\frac{\om_{kp+1}}{\om_{kp'+1}}<\infty,
    $$
which is equivalent to (i). Further, (i) together with H\"older's inequality yields
    \begin{equation*}
    \begin{split}
    \frac1{\om_{2n+1}}\int_0^1f(r)r^{2n+1}\om(r)\,dr
    &\le\frac1{\om_{2n+1}}\left(\int_0^1f(r)^pr^{pn+1}\om(r)\,dr\right)^\frac1p\om_{np'+1}^\frac1{p'}\\
    &\lesssim\left(\frac1{\om_{np+1}}\int_0^1f(r)^pr^{pn+1}\om(r)\,dr\right)^\frac1p,
    \end{split}
    \end{equation*}
and thus (iii) is satisfied. By replacing $n$ by $m-n$ in (iii), and choosing $f(r)=r^{2n}$, one obtains (ii).

An alternate way to see that each $\om\in\DD$ satisfies the Dostani\'c condition \eqref{Eq:Dostanic-condition} reads as follows.
Since $P_\om:L^p_\om\to L^p_\om$ is bounded by Theorem~\ref{th:dualityAp}, the identification $\left(A^p_\om\right)^\star$ of the dual of $A^p_\om$ via the $A^2_\om$-pairing is a subset of $A^{p'}_\om$. Therefore there exists a constant $C=C(\om,p)>0$ such that
      $$
      \| g\|_{A^{p'}_\om }\le C \sup_{f\in A^p_\om\setminus\{0\}}\frac{\left|\langle f,g\rangle_{A^2_\om}\right|}{\| f\|_{A^p_\om}},\quad g\in\H(\D).
      $$
By choosing $g(z)=z^n$, and using the boundedness of the Riesz projection $P$ on $H^p$ for $p>1$, we deduce
      \begin{equation*}
      \om_{np'+1}^\frac1{p'}
      \le C\sup_{f\in A^p_\om\setminus\{0\}}\frac{|\widehat{f}(n)|\om_{2n+1}}{\|f\|_{A^p_\om}}
      \le C\|P\|_{L^p\to H^p}\frac{\om_{2n+1}}{\om_{np+1}^\frac1p},\quad n\in\N,
      \end{equation*}
and the assertion follows.
\end{proof}

It is obvious that \eqref{Eq:Dostanic-condition} implies $\om_{xp}^\frac1p\om_{xp'}^\frac1{p'}\lesssim\om_{2x}$ for all $x\ge1$. Namely, since for any fixed $y\ge0$, there exists $C=C(y,\om)>0$ such that $\om_x\le C\om_{x+y}$ for all $x\ge1$, by choosing $n\le x<n+1$, we deduce
    $$
    \om_{2x}
    \ge\om_{2x+1}
    \ge\om_{2(n+1)+1}
    \gtrsim\om_{(n+1)p+1}^\frac1p\om_{(n+1)p'+1}^\frac1{p'}
    \ge\om_{xp+p+1}^\frac1p\om_{xp'+p'+1}^\frac1{p'}
    \asymp\om_{xp}^\frac1p\om_{xp'}^\frac1{p'},\quad x\ge1.
    $$

With these preparations we can embark on the proof of Theorem~\ref{th:P+1peso}.

\medskip

\begin{Prf}{\em{Theorem~\ref{th:P+1peso}.}} Denote $J_\om(r)=\int_{0}^{r}\frac{dt}{\widehat{\om}(t)(1-t)}$ for all $0\le r<1$. If $\om\in\Dd$, then \eqref{6} yields
	\begin{equation}\label{eq:JDbelowdescription}
	J_\om(r)\le \frac{C(1-r)^\beta}{\widehat{\om}(r)}\int_{0}^{r}\frac{dt}{(1-t)^{1+\beta}}\le \frac{C}{\beta\widehat{\om}(r)},\quad 0\le r <1.
	\end{equation}
The proof of Theorem~\ref{th:proyderivLP} with $k=0$ together with $J_\om(r)\widehat{\om}(r)\lesssim1$
readily gives that $P^+_\om:L^p_\om\to L^p_\om$ is bounded.

Conversely, assume that $P^+_\om :\, L^p_\om\to L^p_\om$ is bounded. Then \cite[Proposition~3]{KorhonenPelaezRattya2018} implies
    \begin{equation}\label{muck7}
    \sup_{0<t<1}\left(\int_0^t(J_\om(r))^{p}\om(r)\,dr+1\right)
    \widehat{\om}(t)^{p-1}<\infty.
    \end{equation}
Recall that for $K>1$ and $n\in\N\cup\{0\}$, $\r_n=\r_n(K,\om)\in[0,1)$ is defined by $\widehat{\om}(\r_n)=\widehat{\om}(0)K^{-n}$. Therefore
    \begin{equation*}
    \begin{split}
    \int_0^{\r_n}J_\om^p(t)\om(t)\,dt
    &\ge\int_{\r_{n-1}}^{\r_{n}}\left(\int_{\r_{n-2}}^{\r_{n-1}}\frac{ds}{\widehat{\om}(s)(1-s)}\right)^p\om(t)\,dt\\
    &\ge\frac{\widehat{\om}(\r_{n-1})-\widehat{\om}(\r_{n})}{\widehat{\om}(\r_{n-2})^p}\left(\log\frac{1-\r_{n-2}}{1-\r_{n-1}}\right)^p\\
    &=\frac{K-1}{\widehat{\om}(\r_n)^{p-1}K^{2p}}\left(\log\frac{1-\r_{n-2}}{1-\r_{n-1}}\right)^p,\quad n\in\N\setminus\{1\},
    \end{split}
    \end{equation*}
and hence \eqref{muck7} implies
	\begin{equation}\label{Dbelowdescriptionroene}
	1-\r_n\le C_11-\r_{n+1}, \quad n\in\N\cup\{0\},
	\end{equation}
for some constant $C_1=C_1(p,\om)>0$. Let now $\b>0$, and for $0<r<1$, choose $n\in\N\cup\{0\}$
such that $\r_n\le r<\r_{n+1}$. Then \eqref{Dbelowdescriptionroene} yields
    \begin{equation*}
    \begin{split}
    \frac{\widehat{\om_{[\b]}}(r)}{(1-r)^\b}
    &\ge\frac{\widehat{\om_{[\b]}}(\r_{n+1})}{(1-\r_n)^\b}
    =\frac1{(1-\r_n)^\b}\sum_{j=n+1}^\infty\int_{\r_j}^{\r_{j+1}}\om(t)(1-t)^\b\,dt\\
    &\ge\frac1{(1-\r_n)^\b}\sum_{j=n+1}^\infty(1-\r_{j+1})^\b\int_{\r_j}^{\r_{j+1}}\om(t)\,dt\\
    &=\frac{K-1}{K^2}\sum_{j=n+1}^\infty\left(\frac{1-\r_{j+1}}{1-\r_n}\right)^\b\frac{\widehat{\om}(0)}{K^{j-1}}\\
    &\ge\frac{K-1}{K^2}\left(\frac{1-\r_{n+2}}{1-\r_n}\right)^\b\widehat{\om}(\r_n)
    \ge\frac{K-1}{K^2C^{2\b}}\widehat{\om}(r),
    \end{split}
    \end{equation*}
and therefore there exists a constant $C_2=\frac{K^2C_1^2}{K-1}>1$ such that
    \begin{equation}\label{7}
    \widehat{\om}(r)\le C_2\frac{\widehat{\om_{[\beta]}}(r)}{(1-r)^{\beta}},\quad 0\le r<1.
    \end{equation}
We will show next that if \eqref{7} is satisfied for some $\b\in(0,\infty)$,
then $\om\in\Dd$. To see this, first integrate by parts to get
    $$
    \widehat{\om_{[\b]}}(r)=\int_r^1\om(t)(1-t)^\b\,dt=\widehat{\om}(r)(1-r)^\b-\int_r^1\widehat{\om}(t)\b(1-t)^{\b-1}\,dt.
    $$
Hence, for a given $\b>0$, the inequality in \eqref{7} is equivalent to
    $$
    \widehat{\om}(r)\le C_2\widehat{\om}(r)-\frac{C_2\int_r^1\widehat{\om}(t)\b(1-t)^{\b-1}\,dt}{(1-r)^\b},
    $$
that is,
    \begin{equation}\label{7prima}
    \frac{\int_r^1\widehat{\om}(t)\b(1-t)^{\b-1}\,dt}{(1-r)^{\b}}\le C_3\widehat{\om}(r),\quad 0\le r<1,
    \end{equation}
where $C_3=\frac{C_2-1}{C_2}\in(0,1)$. By choosing $M$ sufficiently large
such that $1-\frac{1}{M^\b}>C_3$ we obtain
    \begin{equation*}
    \begin{split}
    C_3\widehat{\om}(r)
    \ge\frac{\int_r^{1-\frac{1-r}{M}}\widehat{\om}(t)\b(1-t)^{\b-1}\,dt}{(1-r)^\b}
    \ge\widehat{\om}\left(1-\frac{1-r}{M}\right)\left(1-\frac{1}{M^\b}\right),
    \end{split}
    \end{equation*}
and thus $\om\in\Dd$. Since $\om\in\DD$ by the hypothesis, we deduce $\om\in\DDD$.
\end{Prf}

\medskip

We make a couple of observations on the proof of Theorem~\ref{th:P+1peso}. If $\om\in\Dd$ and $\g>0$, then \eqref{6} implies
	\begin{equation}\label{A}
	\int_0^r\frac{dt}{\widehat{\om}(t)^\g(1-t)}\le \frac{C_1}{\widehat{\om}(r)^\g},\quad 0\le r<1,
	\end{equation}
for some $C_1=C_1(\g,\om)$, see \eqref{eq:JDbelowdescription} for the case $\g=1$. Conversely, \eqref{A} implies
	\begin{equation*}
	\frac{C_1}{\widehat{\om}\left(1-\frac{1-r}{K}\right)^\g}
	\ge\int_r^{1-\frac{1-r}{K}}\frac{dt}{\widehat{\om}(t)^\g(1-t)}
	\ge\frac{\log K}{\widehat{\om}(r)^\g},\quad 0\le r<1,\quad K>1,
	\end{equation*}
and by choosing $K$ sufficiently large we deduce $\om\in\Dd$. Further, for each $K>1$, \eqref{6} with $\r_n=r\le t=\r_{n+1}$ easily implies \eqref{Dbelowdescriptionroene} for some constant $C_1=C_1(K,\om)>0$. The proof of Theorem~\ref{th:P+1peso} shows that this further implies \eqref{7}, which in turn yields \eqref{7prima}. Therefore \eqref{Dbelowdescriptionroene}, \eqref{7}, \eqref{7prima} and \eqref{A} are equivalent characterizations of the class $\Dd$.
	
\begin{proposition}\label{proposition:dostanic fails}
The weight $W=W_{\om,\vp}$ defined in the proof of Theorem~\ref{theorem:non-regular weights close to D} with the choice $\om\equiv1$ does not belong to the Dostani\'c class for any $p\in(1,\infty)\setminus\{2\}$.
\end{proposition}

\begin{proof}
The weight defined in the proof of Theorem~\ref{theorem:non-regular weights close to D} with $\om\equiv1$ is
    $$
    W(r)=W_{\om,\psi}(r)=\sum_{n=N}^\infty\chi_{[r_{2n+1},r_{2n+2}]}(r),\quad r_x=1-\frac{1}{2^{x\psi(x)}},\quad x\ge1,
    $$
where $\psi:[1,\infty)\to(0,\infty)$ an increasing unbounded function. Denote $x=\frac{p}{1-r_{2j+t}}$, where $0\le t\le2$. If $0\le t\le 1$, then \eqref{23'} and the last inequality in \eqref{19'} imply
    \begin{equation*}
    \begin{split}
    W_x&=\int_0^{r_{2j}}r^xW(r)\,dr+\int_{r_{2j+1}}^1r^xW(r)\,dr
    \asymp\int_0^{r_{2j}}r^xW(r)\,dr+\int_{r_{2j+1}}^1W(r)\,dr\\
    &=\sum_{k=N}^{j-1}\int_{r_{2k+1}}^{r_{2k+2}}r^x\om(r)\,dr+\sum_{k=j}^{\infty}\int_{r_{2k+1}}^{r_{2k+2}}\om(r)\,dr\\
    &=\sum_{k=N}^{j-1}\int_{r_{2k+1}}^{r_{2k+2}}r^x\om(r)\,dr+\sum_{k=j}^{\infty}\left(\widehat{\om}(r_{2k+1})-\widehat{\om}({r_{2k+2}})\right)\\
    &\asymp\sum_{k=N}^{j-1}\int_{r_{2k+1}}^{r_{2k+2}}r^x\om(r)\,dr+\sum_{k=j}^{\infty}\widehat{\om}(r_{2k+1})
    \asymp\sum_{k=N}^{j-1}\int_{r_{2k+1}}^{r_{2k+2}}r^x\om(r)\,dr+\widehat{\om}(r_{2j+1})=I_x+II_x,
    \end{split}
    \end{equation*}
where
    \begin{equation*}
    \begin{split}
    I_x&=\sum_{k=N}^{j-1}\int_{r_{2k+1}}^{r_{2k+2}}r^x\,dr
    \asymp\frac1x\sum_{k=N}^{j-1}\left(r_{2k+2}^x-r_{2k+1}^x\right).
    \end{split}
    \end{equation*}
Now
    $$
    r_{2j}^x-r_{2j-1}^x=r_{2j}^x\left(1-\left(\frac{r_{2j-1}}{r_{2j}}\right)^x\right),\quad
    \left(\frac{r_{2j-1}}{r_{2j}}\right)^x\le\left(\frac{r_{2j-1}}{r_{2j}}\right)^\frac{p}{1-r_{2j}}
    \asymp r_{2j-1}^{\frac{p}{1-r_{2j}}}\to0,\quad j\to\infty,
    $$
and for $k\le j-2$ we have
    $$
    r_{2k+2}^x-r_{2k+1}^x\le r_{2k+2}^x\le r_{2j-2}^x\le  r_{2j-1}^{\frac{p}{1-r_{2j}}}\to 0,\quad j\to\infty,
    $$
and thus $I_x\asymp\frac1xr_{2j}^x$ as $j\to\infty$. Therefore
    $$
    W_x\asymp\frac1xr_{2j}^x+(1-r_{2j+1}),\quad j\to\infty.
    $$
Therefore the Dostani\'c condition for $W$ in the case $0\le t<1$ boils down to the inequality
    \begin{equation}
    \begin{split}
    &\left((1-r_{2j+t})^\frac1pr_{2j}^{\frac{1}{1-r_{2j+t}}}+(1-r_{2j+1})^\frac1p\right)
    \left((1-r_{2j+t})^\frac1{p'}r_{2j}^{\frac{1}{1-r_{2j+t}}}+(1-r_{2j+1})^\frac1{p'}\right)\\
    &\quad\lesssim\left((1-r_{2j+t})r_{2j}^{\frac{2}{1-r_{2j+t}}}+(1-r_{2j+1})\right),\quad j\to\infty.
    \end{split}
    \end{equation}
This inequality may only fail for the ``crossing'' terms resulting from the multiplication on the left, the rest of the terms are trivially fine. Let us look at the first inequality involving a crossing term:
    $$
    (1-r_{2j+t})^\frac1pr_{2j}^{\frac{1}{1-r_{2j+t}}}(1-r_{2j+1})^\frac1{p'}
    \le C\left((1-r_{2j+1})+(1-r_{2j+t})r_{2j}^{\frac{2}{1-r_{2j+t}}}\right).
    $$
By dividing the inequality by $(1-r_{2j+t})r_{2j}^{\frac{2}{1-r_{2j+t}}}$ and completing the square, we deduce
    $$
    \left(1-\left(\frac{1-r_{2j+1}}{1-r_{2j+t}}\right)^\frac12r_{2j}^{-\frac{1}{1-r_{2j+t}}}\right)^2
    +r_{2j}^{-\frac{1}{1-r_{2j+t}}}
    \left(2\left(\frac{1-r_{2j+1}}{1-r_{2j+t}}\right)^\frac12-C^{-1}\left(\frac{1-r_{2j+1}}{1-r_{2j+t}}\right)^\frac{1}{p'}\right)\ge0.
    $$
Because of how $r_x$ is defined, there exists a sequence $\{t_j\}$ tending to zero from the right such that the square term above with $t_j$ in place of $t$ vanishes for all large $j$. If $1<p<2$, then for each prefixed $C$ the sign of the second term with $t_j$ in place of $t$ becomes negative for all $j$ large enough because $t_j\to0^+$ as $j\to\infty$. Thus the Dostani\'c condition is not valid. If $2<p<\infty$, then the other crossing term makes the inequality fail.
\end{proof}

Let $\om=\varphi:[0,1)\to[0,\infty)$ decreasing to zero arbitrarily slowly and such that $-\vp'(t)/\vp(t)\lesssim1/(1-t)$ for all $0\le t<1$. Then $\om\in\DDD$ by Theorem~\ref{theorem:non-regular weights close to D}, and therefore Proposition~\ref{proposition:dostanic fails} and Theorem~\ref{theorem:non-regular weights close to D} combined with Theorem~\ref{th:P+1peso} show that there exists a weight $W$ such that $A^p\subset A^p_W\subset A^p_\om$, $P^+:L^p\to L^p$ and $P^+_\om:L^p_\om\to L^p_\om$ are bounded but $P_W:L^p_W\to L^p_W$ is not. This shows that the boundedness of $P_\om$ is equally much related to the regularity of the weight~$\om$ as to its growth.

\section{Bergman projection $P_\om:L^p_\nu\to L^p_\nu$}\label{sec:Pwnu}

The main objective of this section is to prove Theorem~\ref{Theorem:P_w-L^p_v}. We begin with auxiliary results which are of independent interest.

\begin{proposition}\label{proposition:P-necessary}
Let $1<p<\infty$ and $\om,\nu,\eta$ radial weights. If $P_\om:L^p_\nu\to L^p_\eta$ is bounded, then
    \begin{equation}\label{Eq:P-necessary-condition}
    \sup_{m-n\in\N\cup\{0\}}\left(\frac{\om_{2m+1}}{\om_{2(m-n)+1}}\right)^p\frac{\eta_{p(m-n)+1}}{\nu_{p(m+n)+1}}\le\|P_\om\|^p_{L^p_\nu\to L^p_\eta}.
    \end{equation}
In particular, if $P_\om:L^p_\nu\to L^p_\nu$ is bounded, then the following statements hold:
\begin{itemize}
\item[\rm(i)] $\om\in\DD\Rightarrow\nu\in\DD$;
\item[\rm(ii)] $\nu\in\M\Rightarrow\om\in\M$.
\end{itemize}
\end{proposition}

\begin{proof}
For $m,n\in\mathbb{R}$ with $m-n\in\N\cup\{0\}$, let $f_{m,n}(z)=z^{m-n}|z|^{2n}$ for all $z\in\D$. Then $P_\om(f_{m,n})(z)=\frac{\om_{2m+1}}{\om_{2(m-n)+1}}z^{m-n}$, and hence
    \begin{equation*}
    \begin{split}
    2\|P_\om\|^p_{L^p_\nu\to L^p_\eta}\nu_{p(m+n)+1}
    &=\|P_\om\|^p_{L^p_\nu\to L^p_\eta}\|f_{m,n}\|_{L^p_\nu}^p
    \ge\|P_\om(f_{m,n})\|_{L^p_\eta}^p\\
    &=2\left(\frac{\om_{2m+1}}{\om_{2(m-n)+1}}\right)^p\eta_{p(m-n)+1},
    \end{split}
    \end{equation*}
and \eqref{Eq:P-necessary-condition} follows. Now assume $\om\in\DD$, and take $\eta=\nu$ and $m=2n$. Then \eqref{Eq:P-necessary-condition} and \cite[Lemma~2.1(x)]{PelSum14} yield
    \begin{equation*}
    \begin{split}
    1\gtrsim\left(\frac{\om_{4n+1}}{\om_{2n+1}}\right)^p\frac{\nu_{pn+1}}{\nu_{3pn+1}}
    \gtrsim\frac{\nu_{pn+1}}{\nu_{3pn+1}},\quad n\in\N\cup\{0\}.
    \end{split}
    \end{equation*}
Let $x\ge3p+2$, and choose $n\in\N$ such that $pn+1\le x<p(n+1)+1$. Then
    $$
    \nu_x\le\nu_{pn+1}\lesssim\nu_{3pn+1}\le\nu_{3x-3p-2}\le\nu_{2x},
    $$
and hence $\nu\in\DD$ by \cite[Lemma~2.1(x)]{PelSum14}. Thus (i) is satisfied. Next assume $\nu\in\M$, and take $m=\frac{1}{2}(K+1)j$ and $n=\frac{1}{2}(K-1)j$, where $K\ge2$ will be fixed later and $j\in\N$. Then $m-n=j$ and $m+n=Kj$, and hence \eqref{Eq:P-necessary-condition} yields
    \begin{equation*}
    C_1\ge\left(\frac{\om_{(K+1)j+1}}{\om_{2j+1}}\right)^p\frac{\nu_{pj+1}}{\nu_{pKj+1}},\quad j\in\N,
    \end{equation*}
with $C_1=\|P_\om\|^p_{L^p_\nu\to L^p_\nu}$. Since $\nu\in\M$ by the hypothesis, \eqref{eq:11} shows that $K=K(\nu)>1$ can be chosen sufficiently large such that $\nu_{pj+1}\ge C_2\nu_{pKj+1}$ with $C_2=C_2(\nu,K)>C_1$ for all $j\in\N$. It follows that
    \begin{equation}\label{pirre}
    \om_{2j+1}\ge\left(\frac{C_2}{C_1}\right)^\frac1p\om_{(K+1)j+1},\quad j\in\N.
    \end{equation}
Let $x\ge\frac{K+3}{K-1}$, and choose $j\in\N$ such that $2j-1<x\le2j+1$. Then \eqref{pirre} implies
    $$
    \om_x\ge\om_{2j+1}\ge\left(\frac{C_2}{C_1}\right)^\frac1p\om_{(K+1)j+1}
    \ge\left(\frac{C_2}{C_1}\right)^\frac1p\om_{Kx},
    $$
and thus $\om\in\M$. Therefore (ii) is valid.
\end{proof}

The reasoning employed in the proof of the following result offers an alternate way the show the necessity of the condition $\om\in\DDD$ for $P_\om^+:L^p_\om\to L^p_\om$ to be bounded in Theorem~\ref{th:P+1peso}.

\begin{proposition}\label{pr:P+otromundo}
Let $1<p<\infty$ and $\om,\nu$ radial weights. If $P^+_\om:L^p_\nu\to L^p_\nu$ is bounded, then $\om,\nu\in\M$. If in addition either $\om\in\DD$ or $\nu\in\DD$, then $\om,\nu\in\DDD$.
\end{proposition}

\begin{proof}
Let $\phi$ be a positive radial function. Then Hardy's inequality implies
    \begin{equation*}
    \begin{split}
    P^+_\om(\phi)(z)
    =2\int_{0}^1 M_1(r,B^\om_z)\phi(r)\om(r)r\,dr
    \gtrsim\sum_{n=0}^\infty \frac{|z|^n}{(n+1)\om_{2n+1}}\int_{0}^1r^{n+1}\phi(r)\om(r)\,dr,
    \end{split}
    \end{equation*}
and hence the hypothesis gives
    \begin{equation*}
    \begin{split}
    \int_0^1\left(\sum_{n=0}^\infty\frac{s^n}{(n+1)\om_{2n+1}}\int_{0}^1  r^{n+1}\phi(r)\om(r)\,dr\right)^p\nu(s)s\,ds
    \lesssim\int_0^1\phi(s)^p\nu(s)s\,ds.
    \end{split}
    \end{equation*}
Then, by choosing $\phi(s)=s^{KN}$, where $K,N\in\N$ and $K\ge4$, we deduce
    \begin{equation*}
    \begin{split}
    \nu_{pKN+1}&\gtrsim\int_0^1 \left(\sum_{n=0}^\infty \frac{s^n}{n+1}\frac{\om_{n+KN+1}}{\om_{2n+1}}\right)^p\nu(s)s\,ds
    \ge\int_0^1 \left(\sum_{n=N}^{KN}\frac{s^n}{n+1}\frac{\om_{n+KN+1}}{\om_{2n+1}}\right)^p\nu(s)s\,ds\\
    &\ge\nu_{pKN+1}\left(\frac{\om_{2KN+1}}{\om_{2N+1}}\right)^p
    \left(\sum_{n=N}^{KN}\frac{1}{n+1}\right)^p
    \ge\nu_{pKN+1}\left(\frac{\om_{2KN+1}}{\om_{2N+1}}\right)^p\left(\log\frac{KN+2}{N+1}\right)^p\\
    &\ge\nu_{pKN+1}\left(\frac{\om_{2KN+1}}{\om_{2N+1}}\right)^p\left(\log\frac{K}{2}\right)^p,
    \end{split}
    \end{equation*}
and hence
    $$
    \om_{2KN+1}\le\frac{C}{\log\frac{K}{2}}\om_{2N+1},\quad N\in\N,
    $$
for some constant $C=C(\om,\nu,p)>0$. Let now $1+K^{-1}\le x<\infty$, and choose $N\in\N$ such that $2N-1\le x<2N+1$. Then
    $$
    \om_x
    \ge\om_{2N+1}
    \ge\frac{\log\frac{K}{2}}{C}\om_{2KN+1}
    \ge\frac{\log\frac{K}{2}}{C}\om_{Kx+K+1}
    \ge\frac{\log\frac{K}{2}}{C}\om_{2Kx},\quad 1+K^{-1}\le x<\infty.
    $$
By choosing $K\ge4$ sufficiently large we deduce $\om\in\M$.

To see that $\nu\in\M$, choose $\vp(s)=s^N$ for $N\in\N$. An argument similar to that above with the trivial inequality $\om_{n+N+1}\ge\om_{2n+1}$, valid for $n\ge N$, gives
    $$
    \nu_{Np+1}
    \gtrsim\int_0^1\left(\sum_{n=N}^{KN}\frac{s^n}{n+1}\frac{\om_{n+N+1}}{\om_{2n+1}}\right)^p\nu(s)s\,ds
    \ge\nu_{KNp+1}\left(\log\frac{K}{2}\right)^p,\quad N\in\N,\quad K\ge4.
    $$
By choosing $K=K(\om,\nu,p)$ sufficiently large, we deduce $\nu\in\M$.

If $\om\in\DD$, then $\nu\in\DD$ by Proposition~\ref{proposition:P-necessary}.
If $\nu\in\DD$, then choose $\vp(s)=s^N$ and $K_1,K_2\in\N$ such that $2K_1\ge K_2>K_1$, and argue as above to deduce
    \begin{equation*}
    \begin{split}
    \nu_{Np+1}
    &\gtrsim\int_0^1\left(\sum_{n=K_1N}^{K_2N}\frac{s^n}{n+1}\frac{\om_{n+N+1}}{\om_{2n+1}}\right)^p\nu(s)s\,ds
    \ge\left(\frac{\om_{K_2N+N+1}}{\om_{2K_1N+1}}\right)^p\left(\sum_{n=K_1N}^{K_2N}\frac1{n+1}\right)^p\nu_{pK_2N+1}\\
    &\ge\left(\frac{\om_{N(K_2+1)+1}}{\om_{2K_1N+1}}\right)^p\left(\log\frac{K_2N+2}{K_1N+1}\right)^p\nu_{pK_2N+1}\\
    &\ge\left(\frac{\om_{N(K_2+1)+1}}{\om_{2K_1N+1}}\right)^p\left(\log\frac{K_2}{K_1}\right)^p\nu_{pK_2N+1},\quad N\in\N.
    \end{split}
    \end{equation*}
Since $\nu\in\DD$, \cite[Lemma~2.1(ii)(vi)]{PelSum14} implies
    $$
    \nu_{pK_2N+1}\gtrsim\nu_{Np+1}\left(\frac{Np+1}{NpK_2+1}\right)^\gamma
    \ge\nu_{Np+1}K_2^{-\gamma},\quad N\in\N,
    $$
for some $\gamma=\gamma(\nu)>0$, and therefore
    \begin{equation*}
    \begin{split}
    1
    &\gtrsim\left(\frac{\om_{N(K_2+1)+1}}{\om_{2K_1N+1}}\right)^p\left(\log\frac{K_2}{K_1}\right)^pK_2^{-\gamma},\quad N\in\N.
    \end{split}
    \end{equation*}
The choice $K_1=3=K_2-1$ gives $\om_{5N+1}\lesssim\om_{6N+1}$ for all $N\in\N$, and it follows that $\om\in\DD$ by \cite[Lemma~2.1(x)]{PelSum14}.
Therefore $\om\in\DDD$ by Theorem~\ref{th:projectionboundedonto}.
\end{proof}

Next we describe, under mild assumptions, the pairs of radial weights $(\om,\nu)$ such that $P^+_\om:L^p_\nu\to L^p_\nu$ is bounded. One interesting feature of this result concerns the case when a priori one knowns that $\nu\in\DD$. Namely, then one does not have to know anything about the weight that induces the maximal projection, yet the boundedness can be characterized neatly. This means that the space under which the projection acts, imposes severe restrictions to the inducing kernel.

\begin{theorem}\label{theorem:P+wLpnu}
Let $1<p<\infty$ and $\om,\nu$ radial weights such that either $\om\in\DD$ or $\nu\in\DD$. Then $P^+_\om:L^p_\nu\to L^p_\nu$ is bounded if and only if $M_p(\om,\nu)<\infty$ and $\om,\nu\in\DDD$. Moreover,
	$$
	M_p(\om,\nu)\lesssim \|P^+_\om\|_{L^p_\nu\to L^p_\nu}\lesssim M^{2-\frac{1}{p}}_p(\om,\nu).
	$$
\end{theorem}

\begin{proof}
Assume that $P^+_\om:L^p_\nu\to L^p_\nu$ is bounded and either $\om\in\DD$ or $\nu\in\DD$. Then $\om,\nu\in\DDD$ by Proposition~\ref{pr:P+otromundo}, and $\|P^+_\om\|_{L^p_\nu\to L^p_\nu}\gtrsim M_p(\om,\nu)$ by \cite[Proposition~3]{KorhonenPelaezRattya2018} and the last part of \cite[Lemma~2.1]{PelSum14}.

Conversely, assume that $M_p=M_p(\om,\nu)<\infty$ and $\om,\nu\in\DDD$. We will show first that
		\begin{equation}\label{eq:mu0n}
		\sup_{0\le r<1}\frac{\widehat{\nu}(r)^{\frac1p}}{\widehat{\om}(r)}\left(\int_r^1\left(\frac{\om(s)}{\nu(s)}\right)^{p'}s\nu(s)\,ds\right)
		^\frac1{p'}\lesssim M_p(\om,\nu)<\infty.
    \end{equation}
By \cite[Lemma~2.1(ii)]{PelSum14} we may choose $\b=\b(p,\om)>0$ such that $(1-s)^\b\widehat{\om}(s)^{-p}$ is essentially decreasing. Therefore
\begin{equation*}
    \begin{split}
    \int_0^r\frac{s\nu(s)}{\widehat{\om}(s)^p}\,ds
    \gtrsim\frac{(1-r)^\b}{\widehat{\om}(r)^p}\int_0^r\frac{s\nu(s)}{(1-s)^\b}\,ds,\quad 0\le r<1.
    \end{split}
    \end{equation*}
Let $M>1$ and $C>1$ such that $\widehat{\nu}(r)\ge C\widehat{\nu}\left(1-\frac{1-r}{M}\right)$, and define
   $r_n=1-M^{-n}$ for all $n\in\N\cup\{0\}$. For $r_2\le r<1$, let $m=m(r)\in\N$ such that $r_m\le r<r_{m+1}$. Then
    \begin{equation*}
    \begin{split}
    (1-r)^\b\int_0^r\frac{s\nu(s)}{(1-s)^\b}\,ds
    &\ge\int_0^{r_{m}}s\left(\frac{1-r}{1-s}\right)^\b\nu(s)\,ds
    =\sum_{n=0}^{m-1}\int_{r_n}^{r_{n+1}}s\left(\frac{1-r}{1-s}\right)^\b\nu(s)\,ds\\
    &\ge\sum_{n=0}^{m-1}r_n\left(\frac{1-r_{m+1}}{1-r_{n}}\right)^\b\left(\widehat{\nu}(r_n)-\widehat{\nu}(r_{n+1})\right)\\
    &\ge\sum_{n=0}^{m-1}r_n\frac{(C-1)}{M^{\b(m-n+1)}}\widehat{\nu}(r_{n+1})
    \ge\cdots\\
    &\ge\widehat{\nu}(r_{m})(C-1)C^{-2}\sum_{n=0}^{m-1}r_n\left(\frac{C}{M^{\b}}\right)^{m-n+1}\\
    &\ge\widehat{\nu}(r)(C-1)C^{-2}\sum_{j=2}^{m+1} r_{m-j+1}\left(\frac{C}{M^{\b}}\right)^{j}\\
	& \ge	 r_{m-1}\widehat{\nu}(r)\frac{C-1}{M^{2\b}}
\ge r_{1}\widehat{\nu}(r)\frac{C-1}{M^{2\b}},\quad r_2\le r<1.
    \end{split}
    \end{equation*}
Hence
    \begin{equation*}
    \begin{split}
    \int_0^r\frac{s\nu(s)}{\widehat{\om}(s)^p}\,ds
    \gtrsim\frac{\widehat{\nu}(r)}{\widehat{\om}(r)^p},
    \quad r\to1^-,
    \end{split}
    \end{equation*}
and \eqref{eq:mu0n} follows from the assumption $M_p(\om,\nu)<\infty$.

To show that $P^+_\om:L^p_\nu\to L^p_\nu$ is bounded, it suffices to consider non-negative functions $f\in L^p_\nu$.
For such functions and $1<p<\infty$, $(P^+_\om(f))^p$ is subharmonic, and hence their $L^p$-means are increasing. It follows that
	\begin{equation}\label{eq:initialstep}
	\|P^+_\om(f)\|^p_{L^p_\nu}\le C\int_{\D\setminus D(0,\frac12)}|P^+_\om(f)(z)|^p\nu(z)\,dA(z),\quad f\ge0,\quad f\in L^p_\nu,
	\end{equation}
for some constant $C=C(\nu)>0$. By using this we deduce
    \begin{equation*}
    \begin{split}
    \|P^+_\om(f)\|^p_{L^p_\nu}
    &\lesssim\int_{\D\setminus D(0,\frac12)}\left(\int_{\D\setminus D(0,\frac12)}f(\z)|B^\om_z(\z)|\om(\z)\,dA(\z)\right)^p\nu(z)\,dA(z)\\
		&\quad+\left(\int_{D(0,\frac12)}f(\z)\om(\z)\,dA(\z)\right)^p=I(f)+II(f).
    \end{split}
    \end{equation*}
Let now $h(z)=\frac{\nu(z)^{\frac1p}}{|z|^{\frac1{p'}}}\left(\int_{|z|}^1\left(\frac{\om(s)}{\nu(s)}\right)^{p'}s\nu(s)\,ds\right)^{\frac{1}{pp'}}$ for all $z\in\D\setminus\{0\}$. Then an integration gives
    \begin{equation}
    \begin{split}\label{eq:mu1n}
    \int_{t}^1 \left(\frac{\om(s)}{h(s)}\right)^{p'}ds
    =p'\left(\int_t^1\left(\frac{\om(s)}{\nu(s)}\right)^{p'}s\nu(s)\,ds\right)^{\frac{1}{p'}}.
    \end{split}
    \end{equation}
H\"older's inequality yields
		\begin{equation}
    \begin{split}\label{eq:mun2}
    I(f)
		&\le\int_{\D\setminus D(0,\frac12)}\left(\int_{\D\setminus D(0,\frac12)}f(\z)^ph(\z)^p|B^\om_z(\z)|\,dA(\z)\right)\\
    &\quad\cdot\left(\int_{\D}|B^\om_z(\z)|\left(\frac{\om(\z)}{h(\z)}\right)^{p'}dA(\z)\right)^{\frac{p}{p'}}\nu(z)\,dA(z),
    \end{split}
    \end{equation}
where, by \cite[Theorem~1(i)]{PelRatproj}, \eqref{eq:mu1n}, \eqref{eq:mu0n} and
\eqref{A},
    \begin{equation}
    \begin{split}\label{eq:mun3}
    \int_{\D}|B^\om_z(\z)|\left(\frac{\om(\z)}{h(\z)}\right)^{p'}dA(\z)
    &\lesssim \int_0^{1}\left(\frac{\om(s)}{h(s)}\right)^{p'}\left(\int_0^{s|z|}\frac{dt}{\widehat{\om}(t)(1-t)}+1\right)\,ds\\
    &=\int_0^{|z|}\left(\int_{t/|z|}^1\left(\frac{\om(s)}{h(s)}\right)^{p'}ds\right)\frac{dt}{\widehat{\om}(t)(1-t)}+1\\
    &\le\int_0^{|z|}\left(\int_{t}^1\left(\frac{\om(s)}{h(s)}\right)^{p'}ds\right)\frac{dt}{\widehat{\om}(t)(1-t)}+1\\
    &=p'\int_0^{|z|}\left(\int_t^1\left(\frac{\om(s)}{\nu(s)}\right)^{p'}s\nu(s)\,ds\right)^{\frac{1}{p'}}\frac{dt}{\widehat{\om}(t)(1-t)}+1\\
    &\lesssim M_p(\om,\nu) \left(\int_0^{|z|}\frac{dt}{\widehat{\nu}(t)^{\frac1p}(1-t)}+1\right)
    \lesssim\frac{M_p(\om,\nu)}{\widehat{\nu}(z)^{\frac1p}}
    \end{split}
    \end{equation}
for all $z\in\D$. This together with \eqref{eq:mun2}, Fubini's theorem and another application of \cite[Theorem~1(i)]{PelRatproj} gives
    \begin{equation}
    \begin{split}\label{eq:mun4}
    I(f)
    &\lesssim M^{p-1}_p(\om,\nu)
    \int_{\D\setminus D(0,\frac12)}\left(\int_{\D\setminus D(0,\frac12)}f(\z)^ph(\z)^p|B^\om_z(\z)|\,dA(\z)\right)
    \frac{\nu(z)}{\widehat{\nu}(z)^{\frac{1}{p'}}}\,dA(z)\\
    &\lesssim M^{p-1}_p(\om,\nu)\int_{\D\setminus D(0,\frac12)} f(\z)^ph(\z)^{p}\\
    &\quad\cdot\left(\int_\frac12^1\left(\int_0^{r|\z|}\frac{dx}{\widehat{\om}(x)(1-x)}\right)
		\frac{\nu(r)r}{\widehat{\nu}(r)^{\frac1{p'}}}\,dr\right)\,dA(\z).
    \end{split}
    \end{equation}
We split the integral over $(\frac12,1)$ into two parts at $|\z|$. Since $\om\in\Dd$, \eqref{A} yields
    \begin{equation}
    \begin{split}\label{eq:mun5}
    \int_{|\z|}^1\left(\int_0^{r|\z|}\frac{dx}{\widehat{\om}(x)(1-x)}\right) \frac{\nu(r)r}{\widehat{\nu}(r)^{\frac1{p'}}}\,dr
    \lesssim\frac{1}{\widehat{\om}(\z)}
    \int_{|\z|}^1\frac{\nu(r)}{\widehat{\nu}(r)^{\frac1{p'}}}\,dr
    =p\frac{\widehat{\nu}(\z)^{\frac{1}{p}}}{\widehat{\om}(\z)},\quad \z\in\D,
    \end{split}
    \end{equation}
which together with \eqref{eq:mu0n} gives
    \begin{equation}
    \begin{split}\label{eq:mun6}
    &h^p(\z)\int_{|\z|}^1\left(\int_0^{r|\z|}\frac{dx}{\widehat{\om}(x)(1-x)}\right) \frac{\nu(r)r}{\widehat{\nu}(r)^{\frac1{p'}}}\,dr\\
    &\lesssim\frac{\nu(\z)}{|\z|^{p-1}}
    \left(\int_{|\z|}^1\left(\frac{\om(s)}{\nu(s)}\right)^{p'}s\nu(s)\,ds\right)^{\frac{1}{p'}}
    \frac{\widehat{\nu}(\z)^{\frac{1}{p}}}{\widehat{\om}(\z)}\lesssim M_p(\om,\nu)\nu(\z),\quad \z\in\D\setminus D\left(0,\frac12\right).
    \end{split}
    \end{equation}
Moreover, H\"older's inequality shows that
    $$
    \sup_{\frac12\le r<1}\frac{\widehat{\om}(r)^p}{\widehat{\nu}(r)}\int_0^r\frac{t\nu(t)}{\widehat{\om}(t)^p}\,dt\lesssim M^p_p(\om,\nu)<\infty,
    $$
and hence \eqref{A} yields
    \begin{equation*}
    \begin{split}
    &\int_\frac12^{|\z|}\left(\int_0^{r|\z|}\frac{dx}{\widehat{\om}(x)(1-x)}\right) \frac{\nu(r)r}{\widehat{\nu}(r)^{\frac1{p'}}}\,dr
    \lesssim\int_\frac12^{|\z|}\frac{r\nu(r)}{\widehat{\om}(r)\widehat{\nu}(r)^{\frac{1}{p'}}}\,dr\\
    &\lesssim M^{p-1}_p(\om,\nu)\int_0^{|\z|}\frac{r\nu(r)}{\widehat{\om}(r)^{p}\left(\int_0^r
    \frac{s\nu(s)}{\widehat{\om}(s)^p}\,ds\right)^{\frac{1}{p'}}}dr
		\asymp M^{p-1}_p(\om,\nu)\left(\int_0^{|\z|}\frac{s\nu(s)}{\widehat{\om}(s)^p}\,ds\right)^{\frac{1}{p}}
    \end{split}
    \end{equation*}
for all $\z\in\D\setminus D\left(0,\frac12\right)$. This together with the hypothesis $M_p(\om,\nu)<\infty$ gives
    \begin{equation}
    \begin{split}\label{eq:mun7}
    &h^p(\z)\int_\frac12^{|\z|}\left(\int_0^{r|\z|}\frac{dx}{\widehat{\om}(x)(1-x)}\right)
		\frac{\nu(r)r}{\widehat{\nu}(r)^{\frac{1}{p'}}}\,dr\\
    &\lesssim M^{p-1}_p(\om,\nu)\frac{\nu(\z)}{|\z|^{p-1}}
    \left(\int_{|\z|}^1\left(\frac{\om(s)}{\nu(s)}\right)^{p'}s\nu(s)\,ds\right)^{\frac{1}{p'}}
    \left(\int_0^{|\z|}
    \frac{s\nu(s)}{\widehat{\om}(s)^p}\,ds\right)^{\frac{1}{p}}
    \lesssim M^{p}_p(\om,\nu)\nu(\z)
    \end{split}
    \end{equation}
for all $\z\in\D\setminus D\left(0,\frac12\right)$. Consequently, by combining \eqref{eq:mun4}, \eqref{eq:mun6} and
\eqref{eq:mun7} we obtain $I(f)\lesssim M^{2p-1}_p(\om,\nu)\|f\|_{L^p_\nu}$. Moreover, since H\"older's inequality shows that
		\begin{equation*}
		\begin{split}
		II(f)\lesssim\|f\|_{L^p_\nu}^p\left(\int_{D(0,\frac12)}\left(\frac{\om(\z)}{\nu(\z)^\frac1p}\right)^{p'}\,dA(\z)\right)^{p-1}
		\lesssim M_p^p(\om,\nu)\|f\|_{L^p_\nu}^p,
		\end{split}
		\end{equation*}
we finally deduce	$\|P^+_\om(f)\|_{L^p_\nu}\lesssim M^{2-\frac{1}{p}}_p(\om,\nu)\|f\|_{L^p_\nu}$.
\end{proof}

Finally, we put together  all the previous results of this section to get a proof of
Theorem~\ref{Theorem:P_w-L^p_v}.

\par\medskip
\begin{Prf}{\em{Theorem~\ref{Theorem:P_w-L^p_v}}.}
Theorem~\ref{theorem:P+wLpnu} shows that (iv) implies (i), which in turn implies (ii) by Proposition~\ref{pr:P+otromundo}. Assume now that $P_\om:L^p_\nu\to L^p_\nu$ is bounded and $\nu\in\M$. Since $\om\in\DD$ by the hypothesis, and $\DDD=\DD\cap\mathcal{M}$ by Theorem~\ref{th:projectionboundedonto}, Proposition~\ref{proposition:P-necessary} implies $\om,\nu\in\DDD$.
Now, since $P_\om:L^p_\nu\to L^p_\nu$ is bounded, it is easy to see that $\int_{0}^1 \left(\frac{\om(s)}{\nu(s)^{\frac1p}}\right)^{p'}\,ds<\infty$. Next, if $m_n(z)=z^n$ for all $n\in\N\cup\{0\}$, then $P^\star_\om(m_n)(\z)=\frac{\om(\z)\nu_{2n+1}}{\nu(\z)\om_{2n+1}}\z^n$
for almost every $\z\in\D$. Since the adjoint operator $P^\star_\om:L^{p'}_\nu\to L^{p'}_\nu$ is bounded by the hypothesis,
arguing as in \cite[p.~121]{PelRatproj} we deduce
    $
		\|P_\om\|_{L^p_\nu\to L^p_\nu}\gtrsim    A_p(\om,\nu)
    $
and thus (iii) is satisfied.

Assume next $A_p(\om,\nu)<\infty$. An integration by parts and H\"{o}lder's inequality give
    \begin{equation}
    \begin{split}\label{eq:c1}
    \int_{0}^{r}s\frac{\nu(s)}{\widehat{\om}(s)^p}\,ds
    & \le  \int_{0}^{\frac{1}{2}}\frac{\nu(s)}{\widehat{\om}(s)^p}\,ds+\int_{\frac{1}{2}}^r\frac{\nu(s)}{\widehat{\om}(s)^p}\,ds\\
		&\le\frac{\widehat{\nu}(0)}{\widehat{\om}(\frac12)^p}
    +p\int_{\frac12}^{r}\frac{\widehat{\nu}(s)}{\widehat{\om}(s)^{p}}\,
		\frac{\om(s)}{\nu(s)^{\frac{1}{p}}}\,\frac{\nu(s)^{\frac{1}{p}}}{\widehat{\om}(s)}\,ds\\
    &\le\frac{\widehat{\nu}(0)}{\widehat{\om}(\frac12)^p}
    +2p\int_{\frac12}^{r}s\frac{\widehat{\nu}(s)}{\widehat{\om}(s)^{p}}\,
		\frac{\om(s)}{\nu(s)^{\frac{1}{p}}}\,\frac{\nu(s)^{\frac{1}{p}}}{\widehat{\om}(s)}\,ds\\
    &\le \frac{\widehat{\nu}(0)}{\widehat{\om}(\frac12)^p}
    +2p\left(\int_{\frac12}^{r}s\left(\frac{\widehat{\nu}(s)}{\widehat{\om}(s)^{p}}\right)^{p'}
    \left(\frac{\om(s)}{\nu(s)^{\frac{1}{p}}}\right)^{p'}\,ds\right)^{\frac{1}{p'}}
    \left(\int_{\frac12}^{r} s\frac{\nu(s)}{\widehat{\om}(s)^p}\,ds\right)^{\frac{1}{p}}\\
    &=\frac{\widehat{\nu}(0)}{\widehat{\om}(\frac12)^p}+2pI(r)^{\frac{1}{p'}}
		\left(\int_{\frac12}^{r} s\frac{\nu(s)}{\widehat{\om}(s)^p}\,ds\right)^{\frac{1}{p}},\quad\frac12< r<1.
    \end{split}
    \end{equation}
Since $A_p(\om,\nu)<\infty$, we have
    \begin{equation}
    \begin{split}\label{eq:c2}
    I(r)^{\frac{1}{p'}}
		\le A^p_p(\om,\nu)\left(\int_{\frac12}^{r}\frac{\sigma(s)}{\widehat{\sigma}(s)^{p}}\,ds\right)^{\frac{1}{p'}}
    \le\frac{A^p_p(\om,\nu)}{(p-1)^{\frac{1}{p'}}}\frac{1}{\widehat{\sigma}(r)^{\frac{p-1}{p'}}},\quad \frac12< r<1,
    \end{split}
    \end{equation}
and therefore
    \begin{equation*}
    \begin{split}
    \widehat{\sigma}(r)^{\frac{1}{p'}}\left(\int_{0}^{r}s\frac{\nu(s)}{\widehat{\om}(s)^p}\,ds\right)^{\frac{1}{p}}
    &\lesssim\widehat{\sigma}(r)^{\frac{1}{p'}}
    +A_p(\om,\nu)\widehat{\sigma}(r)^{\frac{1}{p'}}
    \left(\frac{1}{\widehat{\sigma}(r)^{\frac{p-1}{p'}}}\right)^{\frac{1}{p}}
    \left(\int_{0}^{r}s\frac{\nu(s)}{\widehat{\om}(s)^p}\,ds\right)^{\frac{1}{p^2}}\\
    &=\widehat{\sigma}(r)^{\frac{1}{p'}}+A_p(\om,\nu)\widehat{\sigma}(r)^{\frac{1}{pp'}}
    \left(\int_{0}^{r}s\frac{\nu(s)}{\widehat{\om}(s)^p}\,ds\right)^{\frac{1}{p^2}},\quad \frac12<r<1.
    \end{split}
    \end{equation*}
Fix $r_0=r_0(\nu)\in(\frac12,1)$ such that $\int_0^{r_0}\nu(s)\,ds>0$. Then
    \begin{equation*}
    \begin{split}
    \widehat{\sigma}(r)^{\frac{1}{(p')^2}}\left(\int_{0}^{r}s\frac{\nu(s)}{\widehat{\om}(s)^p}\,ds\right)^{\frac{1}{pp'}}
    &\lesssim\frac{\widehat{\sigma}(r)^{\frac{1}{(p')^2}}}{\left(\int_{0}^{r}s\frac{\nu(s)}{\widehat{\om}(s)^p}\,ds\right)^{\frac{1}{p^2}}}
    +A_p(\om,\nu)\\
    &\le\frac{\widehat{\sigma}(0)^{\frac{1}{(p')^2}}}{\left(\int_{0}^{r_0}s\frac{\nu(s)}{\widehat{\om}(s)^p}\,ds\right)^{\frac{1}{p^2}}}
    +A_p(\om,\nu),\quad r_0\le r<1,
    \end{split}
    \end{equation*}
and it follows that $M_p(\om,\nu)\lesssim A^{p'}_p(\om,\nu)$. Thus (iv) is satisfied. The inequalities in \eqref{eq:quantativePdospesos} follow by Theorem~\ref{theorem:P+wLpnu} and the proof above.
\end{Prf}

\section{Dual of $A^1_\om$ when $\om\in\DD$}\label{Sec:duality-A^p}

The dual of $A^1_\om$ can be idenfied with the Bloch space $\B$ via the $A^2_\om$-pairing if and only if $\om\in\DDD$ by Theorem~\ref{th:projectionboundedonto}. The main aim of this section is to describe $(A^1_\om)^\star$ when $\om\in\DD$, and thus cover the case $\om\in\DD\setminus\DDD$ by proving Theorem~\ref{th:A1dual}.
Since $\om\in\DD$, in the statement of Theorem~\ref{th:A1dual} one may replace $\langle\cdot,\cdot\rangle_{\om\circ\om}$ by the $A^2_{\om\widehat{\om}}$-pairing.

 A vector space $X\subset\H(\D)$ endowed with a quasi-norm $\|\cdot\|_X$ is admissible if
    \begin{enumerate}
    \item $\|\cdot\|_X$-convergence implies the uniform convergence on compact subsets of $\D$;
    \item for each $r>1$, the uniform convergence on compact subsets of $D(0,r)$ implies $\|\cdot\|_X$-convergence;
    \item the function $f_\zeta$ defined by $f_\zeta(z)=f(\zeta z)$ for all $z\in\D$ satisfies $\|f_\z\|_X\le\|f\|_X$ for all $f\in X$ and $\zeta\in\overline{\D}$.
    \end{enumerate}
 For $0<q\le\infty$, a radial weight $\om$ and $X$, being either the Hardy space $H^p$ for some $0<p\le\infty$ or an admissible Banach space, denote
    $$
    X(q,\om)=\left\{f\in\H(\D):\|f\|^{q}_{X(q,\om)}=\int_{0}^{1}\|f_r\|^q_X\om(r)\,dr<\infty\right\},\quad 0<q<\infty,
    $$
and
    $$
    X(\infty,\om)=\left\{f\in\H(\D):\|f\|_{X(\infty,\om)}=\sup_{0<r<1}\|f_r\|_X\widehat{\om}(r)<\infty\right\}.
    $$

To prove Theorem~\ref{th:A1dual} some more definitions and known results are needed. Let $X^0$ denote the closure of polynomials in an admissible Banach space $X$. An increasing sequence $\{\lambda_n\}_{n=0}^\infty$ of natural numbers is lacunary if there exists $C>1$ such that $\lambda_{n+1}\ge C\lambda_{n}$ for all $n\in\N$. The next result follows from \cite[Theorem~2.1]{PavMixnormI} and its proof.

\begin{lettertheorem}\label{th:Xpolynomials}
Let $\{\lambda_n\}_{n=0}^\infty$ be a lacunary sequence and $N\in\N$. Then there exists a sequence of polynomials $\{P_n\}_{n=0}^\infty$ and constants $C_1=C_1(\{\lambda_n\},N)>0$ and $C_2=C_2(\{\lambda_n\},N)>0$ such that
   \begin{equation}\label{eq:P1}
   f=\sum_{n=0}^\infty P_n\ast f,\quad f\in \H(\D);
   \end{equation}

   \begin{equation}\label{eq:P2}
   \widehat{P_n}(j)=0,\quad j\notin \left[\lambda_{n-1},\lambda_{n+N}\right),\quad \lambda_{-1}=0,\quad n\in\N\cup\{0\};
   \end{equation}

   \begin{equation}\label{eq:P3}
   \left\|\sum_{n=0}^k P_n*f\right\|_X\le C_1 \|f\|_X, \quad f \in \H(\D),\quad k\in\N\cup\{0\};
   \end{equation}

   \begin{equation}\label{eq:P4}
   \left\|P_n*f\right\|_X\le C_2 \|f\|_X, \quad f \in \H(\D),\quad n\in\N\cup\{0\};
   \end{equation}

   \begin{equation}\label{eq:P5}
   \lim_{k\to\infty}\left\|f-\sum_{n=0}^k P_n\ast f\right\|_X=0, \quad f \in X^0,
   \end{equation}
for all admissible Banach space $X$ and for all $X=H^p$ with $p>\frac{1}{N+1}$.
\end{lettertheorem}

Now, we will prove three technical lemmas in order to obtain a universal (valid for any $X$) decomposition norm for $X(q,\om)$,
which is a key step in the proof of Theorem~\ref{th:A1dual},
 see Theorem~\ref{th:Xdecomposition} below.
\begin{lemma}\label{le:XP}
Let $X$ be an admissible Banach space or $H^p$ with $0<p\le\infty$. Then
  \begin{equation}\label{eq:XP}
   r^n\|p\|_X\le\|p_r\|_X\le r^{m}\|p\|_X,\quad r\in (0,1],
   \end{equation}
for each polynomial $p=p_{m,n}$ with Maclaurin series $p(z)=\sum_{k=m}^{n}\widehat{p}(k)z^k$.
\end{lemma}

\begin{proof}
The case $X=H^p$ is \cite[Lemma~3.1]{MatPavStu84}. The assertion for an admissible Banach space~$X$ follows by \eqref{eq:XP} with $p=\infty$ and \cite[Proposition~1.3]{PavMixnormI}: since $p$ is a polynomial, it belongs to~$X^0$, and hence the proposition implies
    $$
    \|p_r\|_X=\|p_r\|_{X^0}=\sup\left\{\|p_r*h\|_{H^\infty}:h\in[X^0,H^\infty],\,\|h\|_{[X^0,H^\infty]}\le1\right\},
    $$
where $(p_r*h)(z)=\sum_{k=m}^n\widehat{p}(k)\overline{\widehat{h}(k)}(rz)^k$.
\end{proof}

The following lemma should be compared with \cite[Lemma~4.2]{PavMixnormI}.

\begin{lemma}\label{le:XP2}
Let $\{\lambda_n\}_{n=0}^\infty$ be a lacunary sequence with $\lambda_{-1}=0$, and $N\in\N$, and let $\{P_n\}_{n=0}^\infty$ be the associated sequence of polynomials of Theorem~\ref{th:Xpolynomials}. Then there exists a constant $C=C(\{\lambda_n\},N)>0$ such that
    \begin{equation}\label{eq:fr}
    C\sup_{n\in \N\cup\{0\}} \| P_n\ast f\|_X r^{\lambda_{n+N}}\le \| f_r\|_X\le \left(\sum_{n=0}^{\infty} \| P_n\ast f\|^\beta_X r^{\lambda_{n-1}\beta}  \right)^{\frac1\beta},
    \end{equation}
for all admissible Banach space $X$ (with $\beta=1$) and for all $X=H^p$ with $\frac{1}{N+1}<p<1$ (with $\beta=p$).
\end{lemma}

\begin{proof}
On one hand, by \eqref{eq:P4}, \eqref{eq:P2} and Lemma~\ref{le:XP},
    \begin{equation*}
    \begin{split}
    \| f_r\|_X
    \ge C_2^{-1}\|P_n*f_r\|_X
    =C_2^{-1}\|(P_n*f)_r\|_X
    \ge C_2^{-1}\|P_n*f\|_X r^{\lambda_{n+N}},\quad  n\in\N\cup\{0\},
    \end{split}
    \end{equation*}
and the first inequality in \eqref{eq:fr} with $C=C_2^{-1}$ follows. On the other hand, \eqref{eq:P1}, the triangle inequality, the subadditivity of $x^\b$ for $0<\b<1$ and Lemma~\ref{le:XP} together with \eqref{eq:P2} give
    \begin{equation*}
    \begin{split}
    \|f_r\|^\beta_X
    &=\left\|\sum_{n=0}^\infty P_n*f_r\right\|^\beta_X
    \le\sum_{n=0}^\infty\| P_n*f_r\|^\beta_X
    =\sum_{n=0}^\infty\|(P_n*f)_r\|^\beta_X
    \le\sum_{n=0}^\infty\|P_n*f\|^\beta_X r^{\lambda_{n-1}\beta},
    \end{split}
    \end{equation*}
which is the second inequality in \eqref{eq:fr}.
\end{proof}

The next lemma is an easy consequence of \cite[Lemma~2.1 (ii)]{PelSum14}.

\begin{lemma}\label{le:Mnstar}
Let $\om\in\DD$. Then there exist $K=K(\om)>1$, $C=C(\om)>1$ and $n_0=n_0(\om)\in\N\cup\{0\}$ such that the sequence $\{M_n\}_{n=0}^\infty$ defined in \eqref{eq:Mn}
satisfies that $M_{n+1}\ge CM_n$ for all $n\ge n_0$.
\end{lemma}

For $\om\in\DD$, let $K=K(\om)>1$ and $n_0=n_0(\om)\in\N\cup\{0\}$ be those of Lemma~\ref{le:Mnstar}. Define $M^\star_n=M_{n+{n_0}}$ for all $n\in\N\cup\{0\}$, and set $M^\star_{-1}=0$. Then Lemma~\ref{le:Mnstar} shows that $\{M_n^\star\}_{n=0}^\infty$ is lacunary. We will need the following generalization of \eqref{18}.

\begin{theorem}\label{th:Xdecomposition}
Let $\om\in\DD$, $0<q\le\infty$ and $N\in\N\cup\{0\}$. Let $K=K(\om)>1$ such that $\{M^\star_n\}_{n=0}^\infty$ is a lacunary sequence, and let $\{P_n\}_{n=0}^\infty$ be the sequence of polynomials associated to $N$ and $\{M^\star_n\}_{n=0}^\infty$ via Theorem~\ref{th:Xpolynomials}. Then there exist constants $C_1=C_1(\om,q,N)>0$ and $C_2=C_2(\om,q,N)>0$ such that
    $$
    C_1\sum_{n=0}^\infty K^{-n}\|P_n*f\|^q_{X}
    \le\|f\|^{q}_{X(q,\om)}
    \le C_2\sum_{n=0}^\infty K^{-n}\| P_n*f\|^q_{X},\quad 0<q<\infty,\quad f\in\H(\D),
    $$
and
    $$
    C_1 \sup_{n\in\N\cup\{0\}}K^{-n} \|P_n*f\|_{X}
    \le\|f\|_{X(\infty,\om)}
    \le C_2\sup_{n\in\N\cup\{0\}}K^{-n}\|P_n*f\|_{X},\quad f\in\H(\D),
    $$
for all admissible Banach spaces $X$ and for all $X=H^p$ with $p>\frac{1}{N+1}$.
\end{theorem}

\begin{proof}
Assume without loss of generality that $N\in\N$. Let first $0<q<\infty$. By Lemma~\ref{le:Mnstar},
there exists $K=K(\om)>1$ such that $\{M^\star_n\}_{n=0}^\infty$ is lacunary, and hence there exists a sequence
$\{P_n\}_{n=0}^\infty$ of polynomials associated to $N$ and $\{M^\star_n\}_{n=0}^\infty$ via Theorem~\ref{th:Xpolynomials}. Further, by Lemma~\ref{le:XP2},
    \begin{equation*}
    \begin{split}
    \|f\|^{q}_{X(q,\om)}
    &\ge\sum_{n=0}^{\infty}\int_{\r_{n+N+n_0}}^{\r_{n+N+n_0+1}}\|f_r\|^q_{X}\om(r)\,dr\\
    &\gtrsim\sum_{n=0}^{\infty}\int_{\r_{n+N+n_0}}^{\r_{n+N+n_0+1}}
		\left(\sup_{k\in\N\cup\{0\}}\|P_k*f\|^q_{X}r^{q M^\star_{k+N}}\right)\om(r)\,dr\\
    &\ge\sum_{n=0}^{\infty}\|P_n*f\|^q_{X}\int_{\r_{n+N+n_0}}^{\r_{n+N+n_0+1}}r^{q M^\star_{n+N}}\om(r)\,dr
    \asymp\sum_{n=0}^{\infty}\|P_n*f\|^q_{X}\int_{\r_{n+N+n_0}}^{\r_{n+N+n_0+1}}\om(r)\,dr\\
    &\asymp\sum_{n=0}^{\infty}K^{-n}\|P_n*f\|^q_{X},
    \end{split}
    \end{equation*}
which is the desired lower estimate. Lemma~\ref{le:XP2}, \cite[Proposition~3.2]{PRAntequera} with $K=2^\alpha$ (see also \cite[Proposition~9]{PelRathg}) and Lemma~\ref{le:Mnstar} give
    \begin{equation*}
    \begin{split}
    \|f\|^{q}_{X(q,\om)}
    &\le\int_0^1\left(\sum_{n=0}^{\infty}\|P_n*f\|^\beta_X r^{M^\star_{n-1}\beta}\right)^{\frac{q}{\beta}}\om(r)\,dr\\
    &\le\int_0^1\left(\sum_{n=1}^{\infty}\|P_n*f\|^\beta_X r^{M_{n-1}\beta}+\|P_0*f\|_X^\beta\right)^{\frac{q}{\beta}}\om(r)\,dr\\
    &\lesssim\int_0^1\left(\sum_{n=1}^{\infty}\|P_n*f\|^\beta_X r^{E(M_{n-1}\b)}\right)^{\frac{q}{\beta}}\om(r)\,dr+\|P_0*f\|^q_X
    \asymp\sum_{n=0}^{\infty}K^{-n}\|P_n*f\|^q_{X},
    \end{split}
    \end{equation*}
and thus the upper estimate is proved.

If $f\in X(\infty,\om)$, then the choice $r=\r_{n+N+n_0}$ in the first inequality in \eqref{eq:fr} of Lemma~\ref{le:XP2}, which may be applied because of Lemma~\ref{le:Mnstar}, gives
    \begin{equation*}
    \begin{split}
    \|f\|_{X(\infty,\om)}\ge\|f_{\r_{n+N+n_0}}\|_X\widehat{\om}(\r_{n+N+n_0})
    \gtrsim\sup_{j\in\N\cup\{0\}}\|P_j*f\|_X\r_{n+N+n_0}^{M^\star_{j+N}}K^{-n}
    \gtrsim K^{-n}\| P_n*f\|_{X}
    \end{split}
    \end{equation*}
for all $n\in\N\cup\{0\}$, and the lower estimate for $\|f\|_{X(\infty,\om)}$ follows. To see the upper estimate, write $M=\sup_{n\in\N\cup\{0\}}K^{-n}\| P_n*f\|_{X}<\infty$ for short. Then the second inequality in \eqref{eq:fr}, \cite[Lemma~3.2]{PRAntequera} with $\gamma=\b\frac{\log K}{\log2}$ and $\alpha=\frac{\log K}{\log 2}$ (see also \cite[Lemma~8]{PelRathg}) and \cite[Lemma~2.1(ii)]{PelSum14} give
    \begin{equation*}
    \begin{split}
    \|f_r\|^\beta_X
    &\le M^\beta\sum_{n=0}^\infty r^{M^\star_{n-1}\beta}K^{n\beta}
    =M^\beta\sum_{n=0}^\infty (r^\b)^{M^\star_{n-1}}2^{n\gamma}
    \lesssim M^\beta\widehat{\om}(r^\b)^{-\frac{\gamma}{\alpha}}\\
    &=M^\beta\widehat{\om}(r^\b)^{-\b}
    \lesssim M^\beta\widehat{\om}(r)^{-\b}, \quad 0<r<1,
    \end{split}
    \end{equation*}
and the upper estimate follows.
\end{proof}

For $n\in\N$, let $K_n$ denote the Ces\'aro kernel defined by $K_n(z)=\sum_{j=0}^n\left( 1-\frac{j}{n+1}\right)z^j$ for all $z\in\D$, and write $\sigma_n (f)=K_n*f$ for the Ces\'aro means of $f$. Recall that $I^\om:\H(\D)\to\H(\D)$ is defined by
$I^\om(g)(z)=\sum_{k=0}^\infty \widehat{g}(k)\om_{2k+1}z^{k}$ for all $z\in\D$. The next result shows, in particular, that $I^\om\circ I^\om$ is a bijection from $l^\infty_{\log_2K}(X,\{P_n\})$ to $l^\infty_{-\log_2K}(X,\{P_n\})$,
where $l^\infty_{s}(X,\{P_n\})$ is defined in \eqref{eq:linftyXPn}. We shall use this observation in the proof of Theorem~\ref{th:A1dual} for $X=\BMOA$.

\begin{theorem}\label{th:IomPn}
Let $\om\in\DD$ and $N=1$. Let $K=K(\om)>1$ such that $\{M^\star_n\}_{n=0}^\infty$ is a lacunary sequence, and let $\{P_n\}_{n=0}^\infty$ be the sequence of polynomials associated to $N$ and $\{M^\star_n\}_{n=0}^\infty$ via Theorem~\ref{th:Xpolynomials}. Then there exist constants $C_1=C_1(\om,N)>0$ and $C_2=C_2(\om,N)>0$ such that
    \begin{equation}
    \begin{split}\label{eq:Io}
    C_1\| I^\om(P_n\ast g)\|_{X}\le K^{-n}\|P_n*g\|_{X}\le C_2\| I^\om(P_n*g)\|_{X},\quad n\in\N\cup\{0\},\quad g\in\H(\D),
    \end{split}
    \end{equation}
for all admissible Banach space $X$.
\end{theorem}

\begin{proof}
The argument we employ follows the proof of \cite[Lemma~5.3]{PavMixnormI}. Theorem~\ref{th:Xpolynomials} yields
    \begin{equation*}
    I^\om(P_n\ast g)(z)= \sum_{j= M^\star_{n-1}}^{M^\star_{n+N}} \om_{2j+1} \widehat{P_n}(j)\widehat{g}(j)z^j = \sum_{j= M^\star_{n-1}}^{M^\star_{n+N}} \om_{2j+1}
    g^j_n(z),\quad z\in\D,\quad n\in\N\cup\{0\},
    \end{equation*}
where $g(z)=\sum_{j=0}^\infty\widehat{g}(j)z^j$ and $g^j_n(z)=\widehat{P_n}(j)\widehat{g}(j)z^j$ for all $n\in\N\cup\{0\}$. Write $G^j_n=\sum_{m=0}^j g^m_n$ for short. Then $G^k_n\equiv0$ for all $k< M^\star_{n-1}$ by Theorem~\ref{th:Xpolynomials}, and hence a summation by parts gives
    \begin{equation}
    \begin{split}\label{eq:Io1}
    I^\om(P_n*g)
		&=\sum_{j=M^\star_{n-1}}^{M^\star_{n+N}-1}(\om_{2j+1}-\om_{2j+3})G_n^j+\om_{2M^\star_{n+N}+1}\left(P_n*g\right),\quad n\in\N\cup\{0\}.
    \end{split}
    \end{equation}
Now $\sigma_k(P_n*g)\equiv0$ for $k< M^\star_{n-1}$, and hence a direct calculation yields
    \begin{equation*}
    \begin{split}
    &\sum_{j=M^\star_{n-1}}^{M^\star_{n+N}-1}(\om_{2j+1}-\om_{2j+3})G_n^j\\
    &=\sum_{j=M^\star_{n-1}}^{M^\star_{n+N}-1}(\om_{2j+1}-\om_{2j+3})
    \left[(j+1)\sigma_j(P_n\ast g)-j\sigma_{j-1}(P_n\ast g)\right]\\
    &=\sum_{j=M^\star_{n-1}}^{M^\star_{n+N}-2}\left[(\om_{2j+1}-\om_{2j+3})-(\om_{2j+3}-\om_{2j+5})\right]
    (j+1)\sigma_j (P_n*g)\\
    &\quad+(\om_{2M^\star_{n+N}-1}-\om_{2M^\star_{n+N}+1})M^\star_{n+N}\sigma_{M^\star_{n+N}-1}(P_n*g),\quad n\in\N\cup\{0\},
    \end{split}
    \end{equation*}
which together with \eqref{eq:Io1} gives
    \begin{equation}
    \begin{split}\label{eq:Io2}
    I^\om(P_n\ast g)
    &=\sum_{j=M^\star_{n-1}}^{M^\star_{n+N}-2}\left[(\om_{2j+1}-\om_{2j+3})-(\om_{2j+3}-\om_{2j+5})\right](j+1)\sigma_j (P_n\ast g)\\
    &\quad+(\om_{2M^\star_{n+N}-1}-\om_{2M^\star_{n+N}+1})M^\star_{n+N}\sigma_{M^\star_{n+N}-1}(P_n*g)\\
    &\quad+\om_{2M^\star_{n-1}+1}(P_n*g),\quad n\in\N\cup\{0\}.
    \end{split}
    \end{equation}
Next, the $X$-norm of the three terms appearing on the right hand side of \eqref{eq:Io2} will be estimated separately. First, write $\om_{\{\b\}}(z)=(1-|z|^2)^\b\om(z)$ for all $\b\in\mathbb{R}$ and $z\in\D$, and use \cite[Theorem~2.2]{PavMixnormI} and \cite[Lemma~2.1(vi)]{PelSum14} to deduce
    \begin{equation}
    \begin{split}\label{eq:Io3}
    &\left\|\sum_{j=M^\star_{n-1}}^{M^\star_{n+N}-2}\left[(\om_{2j+1}-\om_{2j+3})-(\om_{2j+3}-\om_{2j+5})\right]
    (j+1)\sigma_j (P_n*g)\right\|_X\\
    &\le\sum_{j=M^\star_{n-1}}^{M^\star_{n+N}-2}
    \left[(\om_{2j+1}-\om_{2j+3})-(\om_{2j+3}-\om_{2j+5})\right]
    (j+1)\left\|\sigma_j(P_n*g)\right\|_X\\
    &\le\left\|P_n*g\right\|_X\sum_{j=M^\star_{n-1}}^{M^\star_{n+N}-2}\left[(\om_{2j+1}-\om_{2j+3})-(\om_{2j+3}-\om_{2j+5})\right] (j+1)\\
    &=\left\|P_n*g\right\|_X\sum_{j=M^\star_{n-1}}^{M^\star_{n+N}-2}(\om_{\{2\}})_{2j+1}(j+1)\\
    &=\left\|P_n*g\right\|_X\int_{0}^1 \left(\sum_{j=M^\star_{n-1}}^{M^\star_{n+N}-2}(j+1)r^{2j+1}\right) (1-r^2)^2\om(r)\,dr\\
    &\le\left\|P_n*g\right\|_X\int_{0}^1 r^{M^\star_{n-1}+1} \left(\sum_{j=0}^\infty(j+1)r^j\right)  (1-r^2)^2\om(r)\,dr\\
    &\le4\left\|P_n*g\right\|_X\int_{0}^1 r^{M^\star_{n-1}+1} \om(r)\,dr
    \asymp K^{-n}\left\|P_n*g\right\|_X,\quad n\in\N\cup\{0\}.
    \end{split}
    \end{equation}
Second, by \cite[Theorem~2.2]{PavMixnormI} and \eqref{Corollary:D-hat},
    \begin{equation}
    \begin{split}\label{eq:Io4}
    &\left\|(\om_{2M^\star_{n+N}-1}-\om_{2M^\star_{n+N}+1})M^\star_{n+N}\sigma_{M^\star_{n+N}-1}(P_n*g)\right\|_X
    \le2(\om_{[1]})_{M^\star_{n+N}}M^\star_{n+N}\left\|P_n*g\right\|_X\\
    &\quad\lesssim\om_{M^\star_{n+N}}\left\|P_n*g\right\|_X
    \asymp K^{-n} \left\|P_n*g\right\|_X,\quad n\in\N\cup\{0\}.
    \end{split}
    \end{equation}
Therefore, by combining \eqref{eq:Io2}, \eqref{eq:Io3} and \eqref{eq:Io4}, the first inequality in \eqref{eq:Io} follows.

An argument similar to that in \eqref{eq:Io2} gives
    \begin{equation}
    \begin{split}\label{eq:Io5}
    P_n*g
    &=\sum_{j=M^\star_{n-1}}^{M^\star_{n+N}-2}\left[(\om_{2j+1}^{-1}-\om_{2j+3}^{-1})-(\om_{2j+3}^{-1}-\om_{2j+5}^{-1})\right]
    (j+1)\sigma_j I^\om(P_n\ast g)\\
    &\quad+(\om_{2M^\star_{n+N}-1}^{-1}-\om_{2M^\star_{n+N}+1}^{-1})M^\star_{n+N}\sigma_{M^\star_{n+N}-1}I^\om(P_n*g)\\
    &\quad+\om_{2M^\star_{n-1}+1}^{-1} I^\om(P_n*g),\quad n\in\N\cup\{0\}.
    \end{split}
    \end{equation}
By \cite[Theorem~2.2]{PavMixnormI}, \cite[Lemma~2.1(vi)]{PelSum14}, Lemma~\ref{le:Mnstar} and \eqref{Corollary:D-hat},
    \begin{equation*}
    \begin{split}
    &\left\|\sum_{j=M^\star_{n-1}}^{M^\star_{n+N}-2}\left[(\om_{2j+1}^{-1}-\om_{2j+3}^{-1})-(\om_{2j+3}^{-1}-\om_{2j+5}^{-1})\right]
    (j+1)\sigma_j\left(I^\om(P_n*g)\right)\right\|_X\\
    &\le\left\|I^\om(P_n\ast g)\right\|_X \sum_{j=M^\star_{n-1}}^{M^\star_{n+N}-2}\left|(\om_{2j+1}^{-1}-\om_{2j+3}^{-1})-(\om_{2j+3}^{-1}-\om_{2j+5}^{-1})\right|(j+1)\\
    &=\left\|I^\om(P_n\ast g)\right\|_X \sum_{j=M^\star_{n-1}}^{M^\star_{n+N}-2}\left|\frac{\left(\om_{\{2\}}\right)_{2j+1}\om_{2j+5}-\left(\om_{\{1\}}\right)_{2j+3}\left(\left(\om_{\{1\}}\right)_{2j+1}+\left(\om_{\{1\}}\right)_{2j+3}\right)}
    {\om_{2j+1}\om_{2j+3}\om_{2j+5}}\right|(j+1)\\
    &\le\left\|  I^\om(P_n\ast g)\right\|_X \sum_{j=M^\star_{n-1}}^{M^\star_{n+N}-2}(j+1)
    \left(\frac{(\om_{\{2\}})_{2j+1}}{\om_{2j+1}\om_{2j+3}}+ 2\frac{(\om_{\{1\}})^2_{2j+1}}{\om_{2j+1}\om_{2j+3}\om_{2j+5}} \right)\\
    &\lesssim\left\|  I^\om(P_n\ast g)\right\|_XK^{2n}\sum_{j=M^\star_{n-1}}^{M^\star_{n+N}-2}(j+1)(\om_{\{2\}})_{2j+1}
    +\left\|  I^\om(P_n\ast g)\right\|_X \sum_{j=M^\star_{n-1}}^{M^\star_{n+N}-2}\frac{(\om_{\{1\}})_{2j+1}}{\om_{2j+3}\om_{2j+5}}\\
    &\lesssim\left\|I^\om(P_n\ast g)\right\|_X \left(K^{n} + K^{2n}\sum_{j=M^\star_{n-1}}^{M^\star_{n+N}-2} (\om_{\{1\}})_{2j+3}\right)
    \lesssim K^n\left\|I^\om(P_n*g)\right\|_X,\quad n\in\N\cup\{0\},
    \end{split}
    \end{equation*}
and
    \begin{equation*}
    \begin{split}
    &\left\|(\om_{2M^\star_{n+N}-1}^{-1}-\om_{2M^\star_{n+N}+1}^{-1})M^\star_{n+N}\sigma_{M^\star_{n+N}-1}I^\om(P_n*g)\right\|_X\\
    &\le\left\|I^\om(P_n*g)\right\|_X  \frac{(\om_{\{1\}})_{2M^\star_{n+N}-1}}{\om_{2M^\star_{n+N}-1}\om_{2M^\star_{n+N}+1}}M^\star_{n+N}\\
    &\lesssim\left\|I^\om(P_n*g)\right\|_X \frac{1}{\om_{2M^\star_{n+N}+1}}
    \asymp K^n\left\|  I^\om(P_n*g)\right\|_X,\quad n\in\N\cup\{0\},
    \end{split}
    \end{equation*}
which together with \eqref{eq:Io5} gives the right hand side of \eqref{eq:Io}.
\end{proof}

Now we are ready for the proof of Theorem~\ref{th:A1dual}.

\medskip\par
\begin{Prf}{\em{Theorem~\ref{th:A1dual}}.}
First observe that $\left(A^1_\om\right)^\star$ can be identified with $[A^1_\om, H^\infty]=[A^1_\om, \mathcal{A}]$ via the
$H^2$-pairing with equality of norms by \cite[Proposition~1.3]{PavMixnormI}. For $N=1$, let $\{P_n\}_{n=0}^\infty$ be the polynomials constructed in Theorem~\ref{th:Xdecomposition} such that
    \begin{equation}\label{eq:A11}
    A^1_\om=\ell^1_{\log_2K}(H^1,\{P_n\}) \quad\text{and}\quad \BMOA(\infty,\om)=\ell^\infty_{\log_2K}\left(\BMOA,\{P_n\}\right)
    \end{equation}
with equivalent norms, where $l^q_{s}(X,\{P_n\})$ and $l^\infty_{s}(X,\{P_n\})$ are defined in \eqref{eq:lqXPn} and \eqref{eq:linftyXPn}. It follows from \cite[Theorem~5.4]{PavMixnormI} and the fact $(H^1)^\star\simeq[H^1,H^\infty]=\BMOA$ via the $H^2$-pairing that
    \begin{equation}\label{eq:duala111}
    [\ell^1_{\log_2K}(H^1,\{P_n\}),H^\infty]
    =\ell^{\infty}_{-\log_2K}([H^1,H^\infty],\{P_n\})
    =\ell^{\infty}_{-\log_2K}(\BMOA,\{P_n\})
    \end{equation}
with equivalent norms. Therefore, by combining \eqref{eq:A11} and \eqref{eq:duala111} we deduce
    \begin{equation}\label{eq:a11}
    \left(A^1_\om\right)^\star\simeq  \ell^{\infty}_{-\log_2K}(\BMOA,\{P_n\})
    \end{equation}
via the $H^2$-pairing with equivalence of norms. Now Theorem~\ref{th:IomPn} implies
    $$
    \|(I^\om\circ I^\om)(P_n*g)\|_{\BMOA}\asymp K^{-2n}\|P_n*g\|_{\BMOA},\quad n\in\N\cup\{0\},\quad g\in\H(\D),
    $$
and hence
    \begin{equation*}
    \begin{split}
    \|(I^\om\circ I^\om)(g)\|_{\ell^{\infty}_{-\log_2K}(\BMOA,\{P_n\})}
    &=\sup_{n}K^{n}\|(I^\om\circ I^\om)(P_n*g)\|_{\BMOA}\\
    &\asymp\sup_{n}K^{-n} \| P_n*g\|_{\BMOA}
    \asymp \|g\|_{\BMOA(\infty,\om)},\quad g\in\H(\D),
    \end{split}
    \end{equation*}
by Theorem~\ref{th:Xdecomposition}. Therefore $I^\om\circ I^\om$ is an isomorphism from  $\BMOA(\infty,\om)$ onto the space $\ell^{\infty}_{-\log_2K}(\BMOA,\{P_n\})$. This together with \eqref{eq:a11} finishes the proof.
\end{Prf}

\section{Further comments and open problems}\label{sec:conjetures}

In this section we discuss two open problems closely related to the results presented in this paper, and pose conjectures on them. The first problem concerns Littlewood-Paley estimates and the second the boundedness of the Bergman projection $P_\om$ on $L^p_\om$.

\subsection{Littlewood-Paley estimates}

Theorem~\ref{th:L-P-D} shows that the Littlewood-Paley formula
    $$
    \|f\|_{A^p_\om}^p\asymp\int_\D|f^{(k)}(z)|^p(1-|z|)^{kp}\om(z)\,dA(z)+\sum_{j=0}^{k-1}|f^{(j)}(0)|^p,\quad f\in\H(\D),
    $$
is valid for some (equivalently for each) $0<p<\infty$ and $k\in\N$ if and only if $\om\in\DDD$. Moreover, if $\asymp$ is replaced by $\gtrsim$, then the characterizing condition is $\om\in\DD$ by Theorem~\ref{theorem:L-P-D-hat}. We next discuss the question of when
\begin{equation}\label{eq:LPextra}
    \|f\|_{A^p_\om}^p\lesssim \int_\D|f^{(k)}(z)|^p(1-|z|)^{kp}\om(z)\,dA(z)+\sum_{j=0}^{k-1}|f^{(j)}(0)|^p,\quad f\in\H(\D).
    \end{equation}
It follows from the proof of Theorem~\ref{th:L-P-D} that \eqref{eq:LPextra} holds for each $0<p<\infty$ and $k\in\N$ if $\om\in\Dd$, and conversely, $\om\in\M$ is a necessary condition for \eqref{eq:LPextra} to hold. However, $\Dd$ is a proper subclass of $\M$ by
Proposition~\ref{pr:counterxample}. Even if we cannot confirm that $\om\in\M$ is the characterizing condition in general, we can say that this is indeed the case if $p$ is even.

\begin{proposition}\label{LPeven}
Let $\om$ be a radial weight and $p=2n$ for some $n\in\N$. Then there exists a constant $C=C(\om,p)>0$ such that
    \begin{equation}\label{12}
    \|f\|_{A^p_\om}^p\le C\left(\int_\D|f'(z)|^p(1-|z|)^{p}\om(z)\,dA(z)+|f(0)|^p\right),\quad f\in\H(\D).
    \end{equation}
if and only if $\om\in\M$.
\end{proposition}

\begin{proof}
If \eqref{12} is satisfied then $\om\in\M$ by the proof of Theorem~\ref{th:L-P-D}. Conversely, let $\om\in\M$. Then $\om_{x}\lesssim x^p(\om_{[p]})_{x}$
for any $0<p<\infty$ and for all $1\le x<\infty$ by \eqref{eq:characterization-M}, and this for $p=2$ together with Parseval's relation gives \eqref{12} for $p=2$. Let now $p=2n$ for $n\in\N\setminus\{1\}$. We may assume, without loss of generality, that $f\in H^\infty$ and $f(0)=0$. Let $g=f^n$. Then H\"older's inequality gives
    \begin{equation*}\begin{split}
    \|f\|^{2n}_{A^{2n}_\omega}
    &=\|g\|^{2}_{A^{2}_\omega}
    \lesssim    \int_{\D}|g'(z)|^{2}(1-|z|)^{2}\om(z)\,dA(z)\\
    &\asymp\int_{\D}|f(z)|^{2(n-1)}|f'(z)|^{2}(1-|z|)^{2}\om(z)\,dA(z)\\
    &\le\left(\int_{\D}|f(z)|^{2n}\om(z)\,dA(z)\right)^{\frac{n-1}{n}}
    \left(\int_{\D}|f'(z)|^{2n}(1-|z|)^{2n}\om(z)\,dA(z)\right)^{\frac{1}{n}},
    \end{split}
    \end{equation*}
from which \eqref{12} follows.
\end{proof}

It is obviously Parseval's relation that makes things click in the proof, and it seems non-trivial to circumvent the difficulties caused by the absence of such identities in the case when~$p$ is not even.

In view of our findings on Littlewood-Paley estimates, and also by Theorems~\ref{ELTEOREMA} and~\ref{th:otromundo}, it is natural to pose the following conjecture.

\medskip
\noindent{\bf{Conjecture~1.}} Let $\om$ be a radial weight, $0<p<\infty$ and $k\in\N$.
Then there exists a constant $C=C(\om,p)>0$ such that
    $$
    \|f\|_{A^p_\om}^p\le C\int_\D|f^{(k)}(z)|^p(1-|z|)^{kp}\om(z)\,dA(z)+\sum_{j=0}^{k-1}|f^{(j)}(0)|^p,\quad f\in\H(\D),
    $$
if and only if $\om\in\M$.

\medskip
It is enough to consider the case $k=1$ because $\om\in\M$ implies $\om_{[\b]}\in\M$ for each $\b>0$ as is seen by using the characterization \eqref{10} of $\M$.

\subsection{Boundedness of the Bergman projection}

Let $\om$ be a radial weight and $1<p<\infty$. Theorem~\ref{th:dualityAp} shows that $\om\in\DD$ is a sufficient condition for $P_\om:L^p_\om\to L^p_\om$ to be bounded, and the reverse H\"older's inequality
    \begin{equation}\label{dp}
    \om_{np+1}^\frac1p\om_{np'+1}^\frac1{p'}\lesssim \om_{2n+1},\quad n\in\N,
    \end{equation}
is necessary by a result due to Dostani\'c~\cite{D09}. Let $D_p$ denote the class of radial weights satisfying~\eqref{dp}.
Then obviously $D_p=D_{p'}$, and $D_p\subset D_q$ for $1<p<q<2$ as it is natural to expect. An integration by parts shows that $\om_x=x\widehat{\om}_{x-1}$, and it follows that $\om\in D_p$ if and only if $\widehat\om\in D_p$. An analogous observation is valid for the class $\DD$ as well because $\om\in\DD$ if and only if $\om_x\lesssim\om_{2x}$. Therefore $\om\in D_p$ if and only if the condition (iii) in Proposition~\ref{proposition:dostanic} is satisfied for $\widehat\om$ in place of $\om$. The case $p=1$ of this condition reads as
	\begin{equation}\label{eq:Dgorrodesc}
	\frac{1}{\widehat\om_{2n+1}}\int_0^1 f(r)r^{2n+1}\widehat\om(r)\,dr
	\lesssim\frac{1}{\widehat\om_{n+1}}\int_0^1 f(r)r^{n+1}\widehat\om(r)\,dr,\quad f\in L^1_{\widehat\om}([0,1)),\quad f\ge 0.
	\end{equation}
Each weight $\widehat\om\in\DD$ satisfies this condition by \cite[Lemma~2.1(ix)]{PelSum14}. Conversely, by testing \eqref{eq:Dgorrodesc} with $f(r)=\widehat\om(r)^{-1}$, it follows that \eqref{eq:Dgorrodesc} is actually a characterization of $\DD$. This observation makes $\DD$ a natural candidate for the class $D_1$. However, throughout the paper we have solved several problems whose solutions are independent on $p$. Moreover, Theorem~\ref{th:PomegaDp} ensures that $P_\om:L^p_\om\to D^p_{\om,k}$ is bounded if and only if $\om\in\DD$. Therefore we pose the following conjecture.

\medskip\noindent{\bf{Conjecture~2.}} Let $\om$ be a radial weight and $1<p<\infty$. Then $P_\om :L^p_\om\to L^p_\om$ is bounded if and only if $\om\in\DD$.
\medskip

In the introduction we claimed that each $\om\in D_p$ admitting sufficient regularity must belong to $\DD$. Now we will make this statement rigorous. For a radial weight $\om$ with $\widehat{\om}(r)>0$ for all $0\le r<1$, define $L_\om(x)=-\log\om_x$. Further, denote $\om_{(\b)}(r)=\om(r)\left(\log\frac1{r}\right)^\b$ for all $r\in(0,1)$ and $0<\b<\infty$. Then $L_\om'(x)=\left(\om_{(1)}\right)_x/\om_x$ and
    \begin{equation}\label{eq:second derivative of L_w}
    L_\om''(x)=-\frac{\left(\om_{(2)}\right)_x}{\om_x}+\left(\frac{\left(\om_{(1)}\right)_x}{\om_x}\right)^2
    =L_\om'(x)\left(L_\om'(x)-L'_{\om_{(1)}}(x)\right)\le0,\quad 0<x<\infty,
    \end{equation}
because $\left(\om_{(1)}\right)_x^2\le\left(\om_{(2)}\right)_x\om_x$ by H\"older's inequality. Therefore $L_\om$ is an increasing unbounded concave function. Moreover, $L_\om'$ is decreasing and $\lim_{x\to\infty}L_\om'(x)=0$. This applied to $\om$ and $\om_{(1)}$ together with \eqref{eq:second derivative of L_w} shows that also $\lim_{x\to\infty}L_\om''(x)=0$. Further, note that if $\vp$ is decreasing such that $R^x/\vp(x)\to\infty$, as $x\to\infty$, for $R\in(0,1)$ sufficiently large, then
    $$
    \frac{\om_x}{\vp(x)}\ge\frac{R^x\widehat{\om}\left(R\right)}{\vp(x)}\to\infty,\quad x\to\infty.
    $$
The function $\vp(x)=\vp_{r,\b}(x)=r^x\left(\log x\right)^\b$ satisfies $R^x/\vp(x)\to\infty$ for each $0<r<R<1$ and $0<\b<\infty$. This shows that $\om_x$ is much larger than $\vp(x)$ for big $x$, and thus the decay of $\om_x$ is severely restricted for every radial weight.

Obviously, $\om\in\DD$ if and only if $\limsup_{x\to\infty}\left(L_\om(Kx)-L_\om(x)\right)<\infty$ for some (equivalently for each) $K>1$. The next lemma characterizes $\DD$ in terms of higher order derivatives of $L_\om$.

\begin{lemma}\label{lemma:L-function-characterization-D-hat}
Let $\om$ be a radial weight such that $\widehat{\om}(r)>0$ for all $0<r<1$. 
Then the following statements are equivalent:
\begin{itemize}
\item[\rm(i)] $\om\in\DD$;
\item[\rm(ii)] $\displaystyle\limsup_{x\to\infty} xL_\om'(x)<\infty$;
\item[\rm(iii)] $\displaystyle-\limsup_{x\to\infty} x^2L_\om''(x)<\infty$;
\item[\rm(iv)] $\displaystyle\limsup_{x\to\infty} x^2\left(L_\om'(x)^2-L_\om''(x)\right)<\infty$;
\item[\rm(v)] $\displaystyle\limsup_{x\to\infty} x^3\left(L_\om'''(x)+L_\om'(x)^3-3L_\om'(x)L_\om''(x)\right)<\infty$;
\item[\rm(vi)] $\displaystyle\limsup_{x\to\infty} x^3L_\om'''(x)<\infty$.
\end{itemize}
\end{lemma}

\begin{proof}
Obviously (iv) implies (iii). Further, (iii) yields
    $$
    L_\om'(x)-L_\om'(y)=\int_x^y\left(-L_\om''(s)\right)\,ds\lesssim\frac1x-\frac1y,\quad 1\le x\le y,
    $$
and hence
    $$
    xL_\om'(x)\lesssim xL_\om'(y)+1-\frac{x}{y}\to1,\quad y\to\infty,
    $$
because $L_\om'(y)\to0$ as $y\to\infty$. Thus (ii) is satisfied. Furthermore, $\left(\om_{(1)}\right)_x\asymp\left(\om_{[1]}\right)_x$ for all $x$ uniformly bounded away from zero, and hence \eqref{Corollary:D-hat} with $\b=1$ shows that (i) and (ii) are equivalent. In a similar manner,
    $$
    x^2\left(L_\om'(x)^2-L_\om''(x)\right)=x^2\frac{\left(\om_{(2)}\right)_x}{\om_x}\asymp x^2\frac{\left(\om_{[2]}\right)_x}{\om_x},\quad x\to\infty,
    $$
and hence \eqref{Corollary:D-hat} with $\b=2$ shows that (i) and (iv) are equivalent. Since
		$$
		\frac{\left(\om_{(3)}\right)_x}{\om_x}=L_\om'''(x)+L_\om'(x)^3-3L_\om'(x)L_\om''(x),
    $$
in a similar manner one shows that (i) and (v) are equivalent by \eqref{Corollary:D-hat} with $\b=3$. Thus (i)-(v) are equivalent.
Therefore, if $\om\in\DD$, then $x^3L_\om'''(x)\lesssim1$. This in turn implies
		$$
		L_\om''(y)-L_\om''(x)=\int_x^yL_\om'''(s)\,ds\lesssim\frac1{x^2}-\frac1{y^2},\quad 1\le x\le y,
    $$
and therefore
    $$
    x^2L_\om''(y)-x^2L_\om''(x)\le C\left(1-\left(\frac{x}{y}\right)^2\right).
		$$
By letting $y\to\infty$, we deduce $-x^2L''_\om(x)\le C$, which gives (iii).
\end{proof}

Let us look at the condition $\om\in D_p$ with $1<p<2$ in terms of $L_\om$:
    \begin{equation*}
    \begin{split}
    \left(\om_{px}\right)^\frac1p\left(\om_{p'x}\right)^\frac1{p'}\le e^C\om_{2x}
    &\quad\Longleftrightarrow\quad\frac1p\log\om_{px}+\frac1{p'}\log\om_{p'x}\le C+\log\om_{2x}\\
    &\quad\Longleftrightarrow\quad h_p(x)=\log\frac1{\om_{2x}}-\frac1p\log\frac1{\om_{px}}-\frac1{p'}\log\frac1{\om_{p'x}}\\
    &\qquad\qquad\qquad\quad=\frac1p\left(L_\om(2x)-L_\om(px)\right)-\frac1{p'}\left(L_\om(p'x)-L_\om(2x)\right)\\
    &\qquad\qquad\qquad\quad=\frac1p\int_{px}^{2x}L_\om'(t)\,dt-\frac1{p'}\int_{2x}^{p'x}L_\om'(t)\,dt\le C.
    \end{split}
    \end{equation*}
The change of variable $s=\frac{p}{p'}t+x\left(1-\frac{p}{p'}\right)$ gives
    \begin{equation*}
    \begin{split}
    h_p(x)&=\frac1p\int_{px}^{2x}L_\om'(t)\,dt-\frac1{p}\int_{px}^{2x}L_\om'\left(\frac{p'}{p}s+x\left(1-\frac{p'}{p}\right)\right)\,ds\\
    &=\frac1p\int_{px}^{2x}\left(\int_t^{\frac{p'}{p}t+x\left(1-\frac{p'}{p}\right)}\left(-L_\om''(s)\right)\,ds\right)\,dt,
    \end{split}
    \end{equation*}
where
    $$
    2x\le\frac{p'}{p}t+x\left(1-\frac{p'}{p}\right)\le p'x.
    $$
Thus $\om\in D_p$ if and only if
    \begin{equation}\label{eq:Doatanic-L-description}
    \limsup_{x\to\infty}\frac1p\int_{px}^{2x}\left(\int_t^{\frac{p'}{p}t+x\left(1-\frac{p'}{p}\right)}\left(-L_\om''(s)\right)\,ds\right)\,dt<\infty.
   \end{equation}
With the aid of this observation we can prove Corollary~\ref{Cor-BergmanProjection}.

\noindent{\emph{Proof of }Corollary~\ref{Cor-BergmanProjection}}
If $\om\in\DD$, then $P_\om:L^p_\om\to L^p_\om$ is bounded by Theorem~\ref{th:dualityAp}. Conversely, if $P_\om:L^p_\om\to L^p_\om$ is bounded, then $\om\in D_p$. If in addition \eqref{Eq:extraconditionEpsilon} is satisfied, then it follows that $\om_{xp}\lesssim\om_{(p+\e)x}$, and thus $\om\in\DD$. Moreover, by estimating the double integral on the left-hand side of \eqref{eq:Doatanic-L-description} downwards to
	\begin{equation}
	\int_{px}^{\frac{2+p}{2}x}\left(\int_{\frac{2+p}{2}x}^{2x}\left(-L_\om''(s)\right)\,ds\right)\,dt
	=\left(1-\frac{p}{2}\right)x\int_{\frac{2+p}{2}x}^{2x}\left(-L_\om''(s)\right)\,ds,
	\end{equation}
and using the hypothesis \eqref{oiuy}, we deduce $y^2L_\om''(y)\lesssim1$, and thus $\om\in\DD$ by Lemma~\ref{lemma:L-function-characterization-D-hat}.\hfill$\Box$

\medskip

Recall that $L'_\om\ge0$ and $L''_\om\le0$. If $L'''_\om\ge0$ or $L''_\om$ is simply essentially increasing, then \eqref{oiuy} is trivially satisfied for the choice $x=x(y)=\frac{2}{2+p}y$, and thus $\om\in D_p$ if and only if $\om\in\DD$ in this case.

\end{document}